\documentclass[12pt, final]{article}
\usepackage{fullpage}
\usepackage{amsthm, amsmath, amssymb, amsfonts}
\usepackage{tikz-cd}
\usepackage{chngcntr}
\tikzset{symbol/.style={,draw=none,every to/.append style={edge node={node [sloped, allow upside down, auto=false]{$#1$}}}}}
\usepackage{mathtools}
\usepackage{showlabels}
\usepackage{comment}
\usepackage{float}
\usepackage[style=alphabetic, giveninits]{biblatex}

\def\R{\mathbb{R}}
\def\C{\mathbb{C}}

\def\Z{\mathbb{Z}}
\def\uu{\mathbf{u}}
\def\vv{\mathbf{v}}
\def\ww{\mathbf{w}}

\def\xx{\mathbf{x}}
\def\yy{\mathbf{y}}
\def\zz{\mathbf{z}}
\def\aa{\mathbf{a}}
\def\bb{\mathbf{b}}
\def\cc{\mathbf{c}}
\def\dd{\mathbf{d}}
\def\ee{\mathbf{e}}
\def\ff{\mathbf{f}}

\def\ii{\mathbf{i}}
\def\jj{\mathbf{j}}
\def\kk{\mathbf{k}}

\def\pp{\mathbf{p}}

\def\JJ{\mathbf{J}}
\def\FF{\mathbf{F}}

\def\span{\operatorname{span}}
\def\pr{\operatorname{pr}}
\def\diag{\operatorname{diag}}
\def\Im{\operatorname{Im}}
\def\Re{\operatorname{Re}}
\def\Ad{\operatorname{Ad}}
\def\ad{\operatorname{ad}}
\def\Sp{\operatorname{Sp}}
\def\SL{\operatorname{SL}}
\def\GL{\operatorname{GL}}
\def\SO{\operatorname{SO}}
\def\SU{\operatorname{SU}}
\def\U{\operatorname{U}}
\def\O{\operatorname{O}}
\def\PSL{\operatorname{PSL}}

\def\Mat{\operatorname{Mat}}
\def\Lag{\operatorname{Lag}}
\def\End{\operatorname{End}}
\def\Hom{\operatorname{Hom}}

\def\Fr{\operatorname{Fr}}
\def\Gr{\operatorname{Gr}}
\def\Tr{\operatorname{Tr}}
\def\g{\mathfrak{g}}
\def\sl{\mathfrak{sl}}
\def\su{\mathfrak{su}}
\def\sp{\mathfrak{sp}}
\def\k{\mathfrak{k}}
\def\t{\mathfrak{t}}
\def\h{\mathfrak{h}}
\def\a{\mathfrak{a}}

\def\p{\mathfrak{p}}
\def\q{\mathfrak{q}}
\def\r{\mathfrak{r}}
\def\u{\mathfrak{u}}

\newtheorem{theorem}{Theorem}[section]
\newtheorem*{theorem*}{Theorem}
\newtheorem{proposition}[theorem]{Proposition}
\newtheorem{remark}[theorem]{Remark}
\newtheorem{definition}[theorem]{Definition}
\newtheorem{lemma}[theorem]{Lemma}
\newtheorem*{example}{Example}
\newtheorem{corollary}[theorem]{Corollary}
\newtheorem*{corollary*}{Corollary}

\newif\ifdebug                                                      %
\debugfalse

\usepackage{draftwatermark}

\usepackage{graphicx} 
\graphicspath{{\main/images/}{images/}}

\title{Simultaneous linear symplectic reduction and orbit fibrations}
\author{Hyunmoon Kim}
\date{\today}

\begin{document}

\maketitle

\begin{abstract}
We develop a correspondence between the orbits of the group of linear symplectomorphisms of a real finite dimensional symplectic vector space in the complex Lagrangian Grassmannian and the Grassmannians of linear subspaces of the real symplectic vector space. Under this correspondence, orbit fibration maps whose fibers are holomorphic arc components correspond to fibrations from simultaneous linear symplectic reduction. We use this to compute the homotopy types of Grassmannians of linear subspaces of the symplectic vector space in the general case, recovering the observations of Arnold in the Lagrangian case, Oh-Park in the coisotropic case, and Lee-Leung in the symplectic case. Binary octahedral symmetries, symplectic twistor Grassmannians, and symmetries of Jacobi forms appear within this structure.
\end{abstract}

\section{Introduction}
\label{sec:intro}

This paper achieves three main results.

First, we compute the homotopy type of the Grassmannian of linear subspaces of a real, finite dimensional, symplectic vector space in the general case. In the Lagrangian, (co)isotropic, and symplectic cases, we recover the results of \cite{Arnold}, \cite{OhPark}, \cite{LeeLeung}, \cite{AjayiBanyaga}. Simultaneous linear symplectic reduction naturally realizes these Grassmannians as total spaces of fiber bundles, whose base space is an isotropic Grassmannian and whose fiber space is a symplectic Grassmannian.

Second, we prepare the orbit fibrations of \cite{Takeuchi}, \cite{Wolf}, \cite{Wolf1} for applications to symplectic geometry. We spend some time to reformulate their results in the specific case of the complex Lagrangian Grassmannian, in order to expose the property that the orbit fibration maps \emph{do not} depend on the choice of basepoint of the orbits, but \emph{do} depend on the choice of compatible linear complex structure on the real symplectic vector space.

Third, we remove the mystery in the correspondence observed in \cite{mythesis}. In \cite{mythesis}, it was observed that there is a bijection between the Grassmannians of linear subspaces of a real, finite dimensional symplectic vector space, and the orbits of the group of linear symplectomorphisms of the symplectic vector space, such that the corresponding objects are homotopy equivalent (See Figure \ref{fig:assemblyR2} for an example). We refine this correspondence to a correspondence of fiber bundles.

We now describe the structure of Grassmannians of linear subspaces of a real symplectic vector space.

The linear subspaces of a symplectic vector space interact with the symplectic form in different ways. For instance, if $W$ is a linear subspace of a $2n$-dimensional real symplectic vector space $(V, \omega)$, its symplectic complement
\[ W^\omega := \{ \vv \in V: \omega(\vv, \ww) = 0, \text{for all } \ww \in W\}\]
may intersect itself in any dimension from $0$ to $\dim_\R W$.

It is convenient to label linear subspaces $W \subseteq V$ by \emph{types}, a partition of $n$ into a triple of nonnegative integers $\vec{n} = (n_0, n_+, n_-)$, where $n_0 = \dim_{\R} (W \cap W^\omega)$, $\dim_\R (W/(W\cap W^\omega)) = 2n_+$, and $\dim_{\R} W^\omega/(W \cap W^\omega) = 2n_-$ (in the spirit of Definition 3.1 of \cite{Hess}). The Grassmannians of linear subspaces of $V$ of type $\vec{n}$ are denoted by $\Gr(\vec{n}) = \Gr(\vec{n}; V)$. The Grassmannians $\Gr(\vec{n})$ are the orbits of the real linear symplectic group \[ G:= \Sp(V)\] in the ordinary real Grassmannian manifolds of $V$ (Figure \ref{fig:shelf}). The common assumptions $W = W^\omega$ (Lagrangian), $W^\omega \subseteq W$ (coisotropic), $W \subseteq W^\omega$ (isotropic), and $W\cap W^\omega = \{0\}$ (symplectic), are equivalent to the type label satisfying, respectively $n_0 = n$, $n_- = 0$, $n_+ = 0$, and $n_0 = 0$. In addition, the type label provides a way to also systematically collect linear subspaces of $V$ that are neither symplectic nor isotropic nor coisotropic.

Because linear subspaces interact with the symplectic form in different ways, the Grassmannians $\Gr(\vec{n})$ exhibit different topological properties. For instance, some Grassmannians $\Gr(\vec{n})$ are compact and others are not (e.g. Corollary \ref{cor:compact}), and some are simply connected and others are not (e.g. (\ref{eq:fundamentalgroups})). This differs from the behavior of ordinary real or complex Grassmannians.

The Grassmannians of linear subspaces of real symplectic vector spaces have been studied type by type. In Lemma 1.2 of \cite{Arnold}, it was shown that the Lagrangian Grassmannian $\Gr(n, 0, 0)$ is diffeomorphic to $\U(n)/O(n)$. In Proposition 2.1 of \cite{OhPark}, it was shown that the coisotropic Grassmannian $\Gr(n_0, n_+, 0)$ is diffeomorphic to $\U(n)/(O(n_0) \times \U(n_+))$. It was shown in Theorem 13 of \cite{LeeLeung}, and Theorem 1.1 of \cite{AjayiBanyaga} in more generality, that the symplectic Grasmannian $\Gr(0, n_+, n_-)$ (identified as a connected component of the oriented symplectic Grassmannian) strongly deformation retracts to the ordinary complex Grassmannian (or $\U(n_+ + n_-)/(\U(n_+) \times \U(n_-))$). In the isotropic and coisotropic cases, one can choose a fixed $\omega$-compatible linear complex structure $\JJ$ and work with basis extensions that are both $\omega(\cdot, \JJ \cdot)$-orthonormal and Darboux. In the symplectic and general case, imposing this convenient condition on the bases no longer sees all the linear subspaces, but does capture the entire homotopy type.

As a corollary to our main theorem, we compute the unsurprising homotopy type of the Grassmannians $\Gr(\vec{n})$ in the general case.
\begin{corollary*}[Corollary \ref{cor:homotopytype}]
Let $(V, \omega)$ be a $2n$-dimensional symplectic vector space and $\vec{n} = (n_0, n_+, n_-)$ be a triple of nonnegative integers that sum to $n$. Then the Grassmannian $\Gr(\vec{n}; V)$ is homotopy equivalent to $\U(n)/(O(n_0) \times \U(n_+) \times \U(n_-))$.
\end{corollary*}

We now state our main theorem in a way that emphasizes the effect of making a choice of $\omega$-compatible linear complex structure $\JJ$:

\begin{theorem*}[Theorem \ref{thm:vectorbundle}, Proposition \ref{prop:compactfibrations}] Let $(V, \omega)$ be a real, $2n$-dimensional symplectic vector space, and $G = \Sp(V)$ be the group of linear symplectomorphisms of $V$. Suppose $\JJ$ is an $\omega$-compatible linear complex structure on $V$, and let $K = G^{\Ad(\JJ)}$ be the maximal compact subgroup of $G$, consisting of linear symplectomorphisms of $V$ that commute with $\JJ$. Then there exists a compact submanifold (see Definition \ref{def:orbitsJ} or below) $\Gr_\JJ(\vec{n}; V) \subseteq \Gr(\vec{n}; V)$ diffeomorphic to the left coset space \[ \U(n)/(O(n_0) \times \U(n_+) \times \U(n_-)), \]
and a $K$-equivariant fiber bundle $\gamma_\JJ: \Gr(\vec{n}; V) \to \Gr_{\JJ}(\vec{n}; V)$ such that each fiber has a complex structure induced by $\JJ$, and is biholomorphic to
\[ \U(n_+, n_-)/(\U(n_+) \times \U(n_-)).\]
\end{theorem*}

The fibers of $\gamma_\JJ$ are biholomorphic to a Siegel domain of genus $1$ in the sense of \cite{PyatetskiiShapiro}, and to the underlying complex manifold of a hermitian symmetric space of noncompact type AIII in the notation of \cite{Helgason}. As topological spaces, the fibers are contractible, so by a result of \cite{Dold}, we obtain a strong deformation retract of the total space to the base space.

We now show how the Grassmannians $\Gr(\vec{n})$ are also naturally total spaces of fiber bundles given by simultaneous linear symplectic reduction.

For any linear subspace $W \subseteq V$ its linear symplectic reduction can be expressed as a short exact sequence of vector spaces
\[ 0 \to W\cap W^\omega \to W \to W/(W\cap W^\omega) \to 0.\]
If $W$ has type $\vec{n}$, $W \cap W^\omega$ is a type $(n_0, 0, n-n_0)$ subspace of $V$ and $W/(W\cap W^\omega)$ is a type $(0, n_+, n_-)$ subspace in $(W+W^\omega)/(W\cap W^\omega))$. Since $W+W^\omega = (W\cap W^\omega)^\omega$, we can consider all short exact sequences for all $W\subseteq V$ of a given type simultaneously, and obtain a $G = \Sp(V)$-equivariant fiber bundle (Theorem \ref{thm:reduction})
\[ (\cdot \cap \cdot^\omega): \Gr(\vec{n}; V) \to \Gr(n_0, 0, n-n_0; V)\]
such that the fibers $(\cdot \cap \cdot^\omega)^{-1}(W_0)$ are identified with $\Gr(0, n_+, n_-; W_0^\omega/W_0)$.

Now let $\Gr_\JJ(\vec{n}; V)$ be the submanifold of $\Gr(\vec{n}; V)$ consisting of linear subspaces $W$ of type $\vec{n}$ such that $W \oplus \JJ(W \cap W^\omega)$ is $\JJ$-invariant. In the isotropic and coisotropic cases $\Gr_{\JJ}(\vec{n}; V)$ coincides with $\Gr(\vec{n}; V)$, and in the symplectic case, the submanifolds $\Gr_{\JJ}(\vec{n};V)$ coincide with ordinary complex Grassmannians (when $\Gr(\vec{n}; V)$ is identified with a connected component of the Grassmannian of oriented symplectic subspaces) of the complex vector space $(V, \JJ)$. If $W \in \Gr_\JJ(\vec{n}; V)$, then
$(W\cap W^\omega)^\omega/(W \cap W^\omega)$ has a reduced symplectic form $\tilde{\omega}$ and $\JJ$ prescribes a way to identify it with the $\omega(\cdot, \JJ \cdot)$-orthogonal complement of $(W\cap W^\omega) \oplus \JJ (W \cap W^\omega)$ in $V$, which is a $\JJ$-invariant symplectic subspace of $V$. From this identification $(W\cap W^\omega)^\omega/(W \cap W^\omega)$ obtains an $\tilde{\omega}$-compatible linear complex structure $\tilde{\JJ}$.

The $K$-equivariant ($\tilde{K}$-equivariant) fiber bundles $\gamma_\JJ$ ($\gamma_{\tilde{\JJ}}$) occur in the following commuting diagram
\begin{equation}
    \begin{tikzcd}
        \Gr(0, n_+, n_-) \arrow[r, hook] \arrow[d, two heads,  "\gamma_{\tilde{\JJ}}"] & \Gr(\vec{n}) \arrow[r, two heads, "\cdot \cap \cdot^\omega"] \arrow[d, two heads, "\gamma_\JJ"] & \Gr(n_0, 0, n-n_0) \arrow[d, equal]\\
        \Gr_{\tilde{\JJ}}(0, n_+, n_-) \arrow[r, hook] & \Gr_\JJ(\vec{n}) \arrow[r, two heads, "\cdot \cap \cdot^\omega"] & \Gr_\JJ(n_0, 0, n-n_0)
    \end{tikzcd}
\end{equation}
where the spaces in the bottom row are compact. After choosing a choice of Darboux basis of $V$ compatible with a choice of basepoint from each space, the fiber bundle of the bottom row can be expressed as
\[\U(n)/(\O(n_0) \times \U(n_+) \times \U(n_-)) \to \U(n-n_+)/(\O(n_0) \times \U(n_-)) \]
with fibers diffeomorphic to $\U(n_+ + n_-) / (\U(n_+) \times \U(n_-))$.

We now discuss our reformulation of a specific case of the results of Takeuchi \cite{Takeuchi} and Wolf \cite{Wolf}, \cite{Wolf1}.

A complex Lagrangian subspace $\FF \subseteq V\otimes_\R\C$ is a complex linear subspace that is maximally isotropic with respect to the $\C$-bilinear extension of $\omega$. Complex Lagrangian subspaces can also be labelled by type (Definition 3.1 of \cite{Hess}) where $(n_0, n_+, n_-)$ is the inertia of the hermitian form $-i \omega^\C( \cdot, \overline{\cdot})|_{\FF}$. Complex Lagrangian subspaces with the same type $\vec{n}$ can be collected into $\Lag^\C(\vec{n}) = \Lag^\C(\vec{n}; V)$ and coincide with the $G = \Sp(V)$-orbits of the complex Lagrangian Grassmannian $\Lag^\C(V)$, the collection of all complex Lagrangian subspaces of $V\otimes_\R\C$.

We state the reformulation of the results of \cite{Takeuchi}, \cite{Wolf}, \cite{Wolf1} in a way that emphasizes both the dependence on the choice of $\omega$-compatible linear complex structure $\JJ$, and the structural similarity with our main theorem.

\begin{theorem*}[\cite{Takeuchi}, \cite{Wolf}, \cite{Wolf1}]
Let $(V, \omega)$ be a real, $2n$-dimensional symplectic vector space, and $G = \Sp(V)$ be the group of linear symplectomorphisms of $V$. Suppose $\JJ$ is an $\omega$-compatible linear complex structure on $V$, and let $K$ be the maximal compact subgroup of $G$ consisting of linear symplectomorphisms of $V$ that commute with $\JJ$. Then there exists a compact submanifold (see Definition \ref{def:orbitsJ}) $\Lag^\C_\JJ(\vec{n}; V) \subseteq \Lag^\C(\vec{n}; V)$ diffeomorphic to the left coset space
\[ \U(n)/(O(n_0) \times \U(n_+) \times \U(n_-))\]
and a $K$-equivariant fiber bundle $\beta_\JJ: \Lag^\C(\vec{n}; V) \to \Lag_{\JJ}^\C(\vec{n}; V)$ such that each fiber has a complex structure induced by $\JJ$, and is biholomorphic to
\[ (\Sp(2n_+; \R)/\U(n_+)) \times (\Sp(2n_-; \R)/\U(n_-)).\]
\end{theorem*}

In \cite{Takeuchi}, \cite{Wolf}, \cite{Wolf1}, orbit fibrations were investigated in the generality of orbits of real semisimple Lie groups in irreducible hermitian symmetric spaces of compact type and complex flag manifolds. The complex Lagrangian Grassmannian occurs as a special case (Section 11 of \cite{Wolf}). However, the homotopy type of $\Lag_\JJ^\C(\vec{n}; V)$ was only fully identified in the symplectic case, and the Grassmannians $\Gr(\vec{n})$ occur as $G = \Sp(V)$-orbits of ordinary real Grassmannians, so they were also not discussed in \cite{Takeuchi}, \cite{Wolf}, \cite{Wolf1}. In addition to exposing the role of choosing an $\omega$-compatible linear complex structure $\JJ$, we contribute the identification of the homotopy type in the general case, and an application of their methods to smooth manifolds that are not necessarily complex flag manifolds.

The fibers of $\beta_\JJ$ are biholomorphic to a product of two Siegel domains of genus $1$ in the sense of \cite{PyatetskiiShapiro}, and underlying complex manifolds of a product of two hermitian symmetric spaces of noncompact type CI in the notation of \cite{Helgason}. As topological spaces, they are contractible, so again by \cite{Dold}, we obtain a strong deformation retract of the total space to the base space, also obtaining a computation of the homotopy type of each orbit.

The construction of $\beta_\JJ$ in the context of the diagram 

\begin{equation}
    \begin{tikzcd}
        \Lag^\C(0, n_+, n_-) \arrow[r, hook] \arrow[d, two heads,  "\beta_{\tilde{\JJ}}"] & \Lag^\C(\vec{n}) \arrow[r, two heads, "\Re \cdot_0"] \arrow[d, two heads, "\beta_\JJ"] & \Gr(n_0, 0, n-n_0) \arrow[d, equal]\\
        \Lag_{\tilde{\JJ}}^\C(0, n_+, n_-) \arrow[r, hook] & \Lag_\JJ^\C(\vec{n}) \arrow[r, two heads, "\Re \cdot_0"] & \Gr_\JJ(n_0, 0, n-n_0)
    \end{tikzcd}
\end{equation}
was stated in Theorem B of \cite{Takeuchi} for the symplectic case and in Theorem 11.8 of \cite{Wolf1} for the general case. The map $\Re\cdot_0$ appears in Theorem A.3 of \cite{Takeuchi}, and Theorem 8.15 of \cite{Wolf1}. In Theorem 8.15.1b and Theorem 10.9.3 of \cite{Wolf1} the fibers of $\Re \cdot_0$ are identified as \emph{holomorphic arc components} of $\Lag^\C(V)$ (One implication of this result is that for two points in different fibers of $\Re\cdot_0$, there does not exist a holomorphic open disk passing through both of them).

Our symplectic approach differs from \cite{Takeuchi}, \cite{Wolf}, \cite{Wolf1} in the following ways. First, our approach is more conveniently applicable to symplectic geometry because all constructions are made from $(V, \omega)$, we do not fix a root system of the Lie algebra of $G = \Sp(V)$, and we allow the complex structure on the complex Lagrangian Grassmannian to explicitly vary according to the choice of $\omega$-compatible linear complex structure $\JJ$ on $(V, \omega)$. We provide a dictionary in Section \ref{sec:liealg}. Second, simultaneous linear symplectic reduction is manifestly $G$-equivariant, and we intentionally state it as a $G$-equivariant map. This differs from \cite{Wolf} and \cite{Wolf1}, where in Theorem 8.15 of \cite{Wolf1}, $\Re \cdot_0$ is explicitly $G$-equivariant, but in Section 11 of \cite{Wolf1} and the corollary to the holomorphic arc component theorem of \cite{Wolf}, it is intentionally restated as a $K$-equivariant map. Third, we emphasize that $\beta_\JJ$s and $\gamma_{\JJ}$s do \emph{not} depend on the choice of basepoint within the $K$-orbit, but \emph{do} depend on the choice of $\JJ$. This allows us to avoid referring to relevant spaces as orbits of a preferred basepoint, and provide a geometric, basepoint free description of the orbit fibration maps.

In Theorem \ref{thm:incidence}, we also restate and provide a different proof of the incidence of the orbit closures (Theorem A.2 of \cite{Takeuchi}, Theorem 10.6.3 \cite{Wolf1}, Figure \ref{fig:incidence}) using a Hilsum-Skandalis bibundle of action Lie groupoids (Remark \ref{rem:bibundle}).

We now discuss the correspondence between the Grassmannians $\Gr(\vec{n})$ and the $G$-orbits $\Lag^\C(\vec{n})$.

In \cite{mythesis}, it was observed that $\Gr(\vec{n})$ and $\Lag^\C(\vec{n})$, are homotopy equivalent. For example, if $V = \R^2$ with the standard symplectic form, the complex Lagrangian Grassmannian manifold is the complex projective line, $G =\SL(2; \R)$ and acts by the M\"{o}bius action, and its orbits are the upper hemisphere, equator, and the lower hemisphere. They correspond, respectively, to the point $\Gr(0, 1, 0; \R^2$), the real projective line $\Gr(1, 0, 0; \R^2)$, and another point $\Gr(0, 0, 1; \R^2)$ (Figure \ref{fig:assemblyR2}). 

In this paper, we refine the correspondence into a correspondence of fiber bundles with contractible fibers and diffeomorphic base spaces. Moreover, we explain the reason why this correspondence exists at all. The reason is that the fibers of $\gamma_\JJ$ and $\beta_\JJ$, which occur as bounded symmetric domains, are such that they have vector space manifestations that are complementary vector spaces. More precisely, we can consider the reduced case, which is fairly representative (Sections \ref{subsec:involutions}, \ref{subsec:symmetricspace}). In this case $\Ad(\JJ)$ provides a Cartan decomposition of the Lie algebra of $G$ into $\g = \k \oplus \p$, and depending on the choice of $(0, n_+, n_-)$, we can additionally choose an involution $I$ that commutes with $\JJ$, and further split $\p$ and $\g$ into $(\Ad(I), \Ad(\JJ))$-eigenspaces $\p =\p^\shortparallel \oplus \p^\times$ and $\g = \k^\shortparallel \oplus \k^\times \oplus \p^\shortparallel \oplus \p^\times$. The exponential map provides the diffeomorphisms to the fibers of $\gamma_\JJ$, $\beta_\JJ$:
\begin{eqnarray*}
    \p^\shortparallel &\cong& (\Sp(2n_+; \R)/\U(n_+)) \times (\Sp(2n_-; \R)/\U(n_-))\\
    \p^\times &\cong& \U(n_+, n_-)/(\U(n_+) \times \U(n_-)).
\end{eqnarray*}

An intermediary fiber bundle
\[ \eta_\JJ: \Lag_\oplus^\C(\vec{n}; V) \to \Lag_{\JJ, \oplus}^\C(\vec{n}; V)\]
arises, as a fiber product of $\gamma_\JJ$ and $\beta_\JJ$. The fibers of $\eta_\JJ$ are diffeomorphic to $\p$, and the base is diffeomorphic to $\Gr_{\JJ}(\vec{n}; V)$ and $\Lag_{\JJ}^\C(\vec{n}; V)$. One of its projection maps carries a corrective diffeomorphism (Corollary \ref{cor:fiberproduct}) of Baker-Campbell-Hausdorff nature which is the identity in the isotropic and coisotropic cases. In the symplectic case, the total space of $\eta_\JJ$ was introduced in \cite{LeeLeung} as the \emph{symplectic twistor Grassmannian}.

Finally, we list some observations that might interest some readers. We observe that the partial Cayley transforms introduced in \cite{KoranyiWolf} appear as elements of the binary octahedral group $BO_{48}$, or ``square roots'' of quaternionic symmetries. We note its formal similarities with the Bruhat decomposition of the complex Lagrangian Grassmannian (Section \ref{subsec:octahedron}). We also note the appearance of the Heisenberg groups of higher degree (appearing in the study of Jacobi forms \cite{EichlerZagier}, \cite{Ziegler}) in the stabilizer subgroups of each orbit (Section \ref{subsec:orbitsheisenberg}). This allows us to describe them in a coordinate free way, and without referring to central extensions (Definition \ref{def:heisenberg}). We also note the appearance of the unitary groups of every indefinite signature in the stabilizer subgroups of $\Lag^\C(V)$. We also marginally strengthen (Theorem \ref{thm:basisR}) the statement of the linear ``relative Darboux theorem'' in \cite{ArnoldGivental}.

We discuss orbits in Section \ref{sec:orbits}, orbit fibrations, complex structures, and simultaneous linear symplectic reduction in Section \ref{sec:fibrations}, and factorizations of orbit fibrations (the actual construction of $\gamma_\JJ$ and $\beta_\JJ$) in Section \ref{sec:factorization}. 

\subsection{Acknowledgments}
\label{subsec:acknowledgements}
This research was partly funded by the U.S.-Israel Binational Science Foundation (BSF) and by the Natural Sciences and Engineering Research Council of Canada. The author would like to express gratitude to Zuoqin Wang for pointing to the reference \cite{ArnoldGivental}, to Karl Neeb for pointing to the reference \cite{Loos}, and to Leonid Kovalev for clarifying some point set topology, and Yael Karshon for discussions.

\subsection{Notation}
\label{subsec:notation}
$-1_X$ denotes either the identity element of an object $X$ or the identity automorphism of an object $X$.
$Z(H)$ denotes the center of a group $H$. $\operatorname{Lie}(H)$ denotes the Lie algebra of a Lie group $H$. If $h \in H$, then $\Ad(h)(\cdot) = h (\cdot )h^{-1}$, and $\ad(\mathbf{h}) (\cdot) = [\mathbf{h}, \cdot]$ denotes left multiplication by the multiplication in $\operatorname{Lie}(H)$ by an element $\mathbf{h} \in \operatorname{Lie}(H)$.

For $R$, either the field of real or complex numbers, $\Sp(2n; R)$ denotes $\Sp(R^{2n})$ for $R^{2n}$ equipped with the standard symplectic form. $\U(a, b)$ denotes the indefinite unitary group.

$\Mat_{j \times \ell}(R)$ denotes the set of $j$ by $\ell$ matrices with coefficients in $R$, $1_j$ the $j \times j$ identity matrix, $0_{j \times \ell} \in \Mat_{j \times \ell}(R)$ the zero matrix. $R^n$ is identified with $\Mat_{n \times 1}(R)$.  $(\cdot)^t$ denotes the matrix transpose and $(\cdot)^\dagger$ denotes the conjugate matrix tranpose, $\Mat_{n \times n}(\C)^{\pm\dagger}$ denotes the set of Hermitian/skew hermitian $n \times n$ matrices and $\Mat_{n \times n}(R)^{\pm\t}$ denotes the set of symmetric/antisymmetric $n \times n$ matrices.

For a Lie group $G$ and a closed subgroup $H \le G$ and a smooth action of $H$ on a space $F$, $G\times_H F:= (G \times F)/H$ denote the orbit space of the action of $H$ on $G \times F$ given by $(g, f) \mapsto (gh, h^{-1}.f)$.

We often denote the bases given by Theorems \ref{thm:basisR}, \ref{thm:basisC}, \ref{thm:basisJ} as $\{ \ee^0, \ee^+, \ee^-, \ff^0, \ff^+, \ff^-\}$, $\{ \ee, \ff\}$, $\{\ee, \JJ\ee\}$.

\section{Orbits}\label{sec:orbits}
In this section we describe the relevant groups (including the binary octahedral group, and the Heisenberg group), spaces, group actions, orbits, and stabilizers. We also revisit some basis extension theorems (Theorems \ref{thm:basisR}, \ref{thm:basisJ}, \ref{thm:basisC}), and describe the incidence relations between the orbit closures (Proposition \ref{prop:incidenceR}, Theorem \ref{thm:incidence}).

\subsection{Orbits in ordinary real Grassmannians}
\label{subsec:orbitsreal}
Let $V$ be a $2n$-dimensional real vector space and $\omega$ be a symplectic form on $V$. Let $\Sp(V)$ be the group of real linear invertible transformations of $V$ that preserve $\omega$--i.e. $\omega(g\vv, g\ww) = \omega(\vv, \ww)$ for all $\vv, \ww \in V$ if and only if $g \in \Sp(V)$. Suppose $W \subseteq V$ is a real linear subspace of dimension $\ell$. Let $W^\omega \subseteq V$ be defined by the condition $\vv \in W^\omega$ if and only if $\omega(\vv, \ww) = 0$ for all $\ww \in W$ and call it the \emph{symplectic complement of $W$}. $W$ and $W^\omega$ have complementary dimensions, but are not complementary subspaces in general. Taking the symplectic complement twice brings any linear subspace back to itself.

Denote $W_0 := W \cap W^\omega$ and $W_0^\omega:= (W_0)^\omega = W + W^\omega$. $\omega$ descends to a nondegenerate form $\tilde{\omega}$ on $W/W_0$, so $W/W_0$ is a symplectic vector space with this form, and is in particular, even dimensional. $\omega$ also descends to a symplectic form on $W^\omega/W_0$, which is again, necessarily even dimensional, and we will again denote the descended symplectic form by $\tilde{\omega}$. We will say $W$ is of \emph{type $\vec{n}$} if $\dim_\R W_0 = n_0$ and $\frac{1}{2} \dim_\R W / W_0 = n_+$. Every linear subspace of $V$ has a type. Since $W$ and $W^\omega$ have complementary dimensions, the type of $W$ determines the dimensions of $W$ and $W^\omega$ as $\dim_\R W = n_0 + 2n_+ = \ell$ and $\dim_\R W^\omega = n_0 + 2n_- = 2n - \ell$, where $n_-:= n - n_0 - n_+$. A subspace $W$ is \emph{isotropic} if and only if $n_+ = 0$, \emph{coisotropic} if and only if $n_- = 0$, \emph{Lagrangian} if and only if $n_0 = n$, and \emph{symplectic} if and only if $n_0 = 0$.

Let $\Gr(\vec{n}; V) = \Gr(\vec{n})$ be the smooth manifold of linear subspaces of $V$ of type $\vec{n}$. Let $\Gr(\ell; V) = \Gr(\ell; 2n)$ be the ordinary Grassmannian of $\ell$-dimensional real linear subspaces of $V$. We will view $\Gr(\vec{n}; V)$ as a submanifold of $\Gr(n_0 + 2n_+; V)$. Then we have a partition
\[ \Gr(\ell; V) = \bigsqcup_{\vec{n}: n_0 + 2n_+ = \ell} \Gr(\vec{n}; V)\]
by how subspaces behave with respect to the symplectic form. See Figure \ref{fig:shelf} for an example.

The linear action of $\Sp(V)$ on $V$ induces an action on each $\Gr(\vec{n}; V)$ by $g.W := \{ g\ww : \ww \in W\}$ for $g \in \Sp(V)$ and $W \in \Gr(\vec{n};V)$, for which only $\pm 1_{\Sp(V)}$ act trivially. According to the linear ``relative Darboux theorem'' in  Section 1.2 of \cite{ArnoldGivental}), this action is transitive, so $\Gr(\vec{n}; V)$'s are precisely the $\Sp(V)$-orbits of $\Gr(\ell; V)$. We will be using the following (slightly stronger) version of the linear ``relative Darboux theorem.''

\begin{definition}[Associated splittings]
Let $W \subseteq V$ be a real linear subspace. We will call a triple $(W_+, W_-, W^0)$ of linear subspaces of $V$ a \emph{splitting of $V$ associated to $W$} if $W_+$ is complementary to $W_0$ in $W$, $W_-$ is complementary to $W_0$ in $W^\omega$, and $W^0$ is an isotropic subspace of $V$ complementary to $W_0$ in $(W_+ \oplus W_-)^\omega$. Given $W_+, W_-$, a choice of $W^0$ is equivalent to a choice of quadratic form on $W_0$.
\end{definition}

\begin{theorem}\label{thm:basisR}
Suppose $W \subseteq V$ is a linear subspace of type $\vec{n}$ and take any splitting $(W_+, W_-, W^0)$ associated to $W$. Then $(W_\pm, \omega|_{W_{\pm}})$ are symplectic subspaces of $V$, and there exists a Darboux basis \[ \{\ee^0_1, \cdots, \ee^0_{n_0}, \ee^+_1, \cdots, \ee^+_{n_+}, \ee^-_1, \cdots, \ee^-_{n_-}, \ff^0_1, \cdots, \ff^0_{n_0}, \ff^+_{1}, \cdots, \ff^+_{n_+}, \ff^-_1, \cdots, \ff^-_{n_-}\}\]
of $V$ such that $W_0 = \span_{\R}\{\ee^0_j\}_{1 \le j \le n_0}$, $W^0 = \span_{\R}\{ \ff^0_j\}_{1 \le j \le n_0}$, $W_+ = \span_{\R}\{\ee^+_\ell, \ff^+_\ell\}_{ 1 \le \ell \le n_+}$, and $W_- = \span_{\R}\{\ee^-_\ell, \ff^-_\ell\}_{1 \le \ell \le n_-}$.
\end{theorem}
\begin{proof} Consider the empty set as a Darboux basis of the zero vector space. Take a Darboux basis of $W/W_0$, and choose representatives $\{ \tilde{\ee}^+_1, \cdots, \tilde{\ee}^+_{n_+}, \tilde{\ff}^+_1, \cdots, \tilde{\ff}^+_{n_+}\}$ in $W$. Since $W_+$ is a complementary to $W_0$ in $W$, we can uniquely decompose the vectors as
\begin{eqnarray*}
    \tilde{\ee}^+_{j} &=& \ee^+_{j} + \vv_j \quad \ee^+_j \in W_+, \vv_j \in W_0\\
    \tilde{\ff}^+_j &=& \ff^+_j + \ww_j \quad \ff^+_j  \in W_+, \ww_j \in W_0
\end{eqnarray*}
Then $\{ \ee^+_1, \cdots, \ee^+_{n_+}, \ff^+_1, \cdots, \ff^+_{n_+}\}$ is a basis of $W_+$ for which $\omega|_{W_+}$ is nondegenerate, and is moreover a Darboux basis of $W_+$ with respect to $\omega|_{W_+}$. 
Similarly, $W_-$ has a Darboux basis $\{ \ee^-_1, \cdots, \ee^-_{n_-}, \ff^-_1, \cdots, \ff^-_{n_-}\}$ with respect to $\omega|_{W_-}$. $W_0 \oplus W^0$ is then a Lagrangian splitting of the symplectic vector space $(W_+ \oplus W_-)^\omega$. Applying Proposition 3.3 of \cite{Kim}, we get a Darboux basis $\{ \ee^0_1, \cdots, \ee^0_{n_0}, \ff^0_1, \cdots, \ff^0_{n_0}\}$ of $(W_+ \oplus W_-)^\omega$ such that $\{\ee^0_1, \cdots, \ee^0_{n_0}\}$ is a basis of $W_0$ and $\{\ff^0_1, \cdots, \ff^0_{n_0}\}$ is a basis of $W^0$. Putting these Darboux bases together, we obtain the desired Darboux basis of $V$.
\end{proof}
Recall that a \emph{linear complex structure} $\JJ$ on $V$ is a real linear map on $V$ such that $\JJ^2 = -1_{V}$, and it is \emph{$\omega$-compatible} if $\JJ \in \Sp(V)$ and $\omega(\vv, \JJ \vv)$ is positive for every nonzero vector $\vv$ in $V$. We state another basis theorem when given an $\omega$-compatible linear complex structure $\JJ$.
\begin{theorem}\label{thm:basisJ}
Suppose $\JJ$ is an $\omega$-compatible linear complex structure on $V$.
\begin{enumerate}
    \item For any coisotropic subspace $W \subseteq V$ of type $(n_0, n_+, 0)$, there exists a Darboux basis $\{ \ee_1, \cdots, \ee_n, \JJ \ee_1, \cdots, \JJ \ee_n\}$ of $V$ such that
\[ W^\omega = \span_{\R}\{\ee_j\}_{1 \le j \le n_0} \quad \text{and} \quad W = \span_{\R}\{\ee_j, \JJ\ee_\ell\}_{1 \le j \le n, n_0 < \ell \le n}.\]
\item If $W \subseteq V$ is a subspace such that $W \oplus \JJ W_0$ is $\JJ$-invariant, the $\omega(\cdot, \JJ\cdot)$-orthogonal complement of $W_0$ in $W$ is $\JJ$-invariant and is the unique subspace of $W$ that is both $\JJ$-invariant and complementary to $W_0$ in $W$.
\item If $W \subseteq V$ is a subspace such that $W \oplus \JJ W_0$ is $\JJ$-invariant then $W^\omega \oplus \JJ W_0$ is also $\JJ$-invariant. Moreover, there exists an $\omega(
\cdot, \JJ \cdot)$-orthogonal splitting
\[ V = (W_0 \oplus \JJ W_0) \oplus W_+^\JJ \oplus W_-^\JJ\]
into $\JJ$-invariant symplectic subspaces of $V$ such that
$W = W_0\oplus W_+^\JJ$ and $W^\omega = W_0 \oplus W_-^\JJ$.
Equivalently, there exists a Darboux basis of $V$ of the form $\{\ee_1, \cdots, \ee_n, \JJ \ee_1, \cdots, \JJ \ee_n\}$ such that
\begin{eqnarray*}
    W_0 &=& \span_{\R}\{ \ee_j\}_{1 \le j \le n_0},\\
    W &=& \span_{\R}\{\ee_j,  \ee_\ell, \JJ\ee_\ell\}_{1 \le j \le n_0, n_0 < \ell \le n_0 + n_+}\\
    W^\omega &=& \span_{\R}\{\ee_j, \ee_\ell, \JJ \ee_{\ell}\}_{1 \le j \le n_0, n_0 + n_+ < \ell \le n}.
\end{eqnarray*}
\end{enumerate}
\end{theorem}

\begin{proof} $(1)$: Take any splitting associated to $W$, and a Darboux basis given by Theorem \ref{thm:basisR}. By taking corresponding linear transformations of $\ff$s, we may assume $\{ \ee_j^0\}_{1 \le j \le n_0}$ are $\omega(\cdot, \JJ \cdot)$-orthonormal. Let $\ee_j := \ee_j^0$ for $1 \le j \le n_0$.

We can also assume $\{\ee_\ell^+\}_{1 \le \ell \le n_+}$ are $\omega(\cdot, \JJ \cdot)$-orthogonal. Then for $1 \le \ell \le n_+$, let
\[ \tilde{\ee}_{n_0 + \ell} := \ee_\ell^+ - \sum_{j = 1}^{n_0} \omega(\ee_{j}^0, \JJ \ee_\ell^+) \ee_{j}^0 \in W.\]
Then $\JJ \tilde{\ee}_{n_0 + \ell} \in (W^\omega)^\omega = W$.
Let $\ee_{n_0 + \ell}$ and $\JJ\ee_{n_0 + \ell}$ be the normalizations of $\tilde{\ee}_{n_0 + \ell}$ and $\JJ \tilde{\ee}_{n_0 + \ell}$.

For the space $\span_\R\{ \ee_j, \JJ \ee_j\}_{1 \le j \le n_0 + n_+}$, its $\omega(\cdot, \JJ \cdot)$-orthogonal complement and its symplectic complement are equal. So the symplectic complement is a symplectic subspace with an $\omega(\cdot, \JJ \cdot)$-orthonormal Darboux basis $\{ \ee_{n_0 + n_+ + 1}, \cdots, \ee_{n}, \JJ \ee_{n_0 + n_+ + 1}, \cdots, \JJ \ee_n\}$. Then $\{ \ee_1, \cdots, \ee_n, \JJ \ee_1, \cdots, \JJ \ee_n\}$ is the desired basis. $(2)$: $\omega$ restricts to a symplectic form on $W\oplus \JJ W_0$. Then $W$ is a coisotropic subspace of $W \oplus \JJ W_0$, and there exists a Darboux basis of $W\oplus \JJ W_0$ given by (1). Then $W_+^\JJ:= \span_{\R} \{ \ee_j, \JJ \ee_j\}_{n_0 < j \le n-n_-}$ is a subspace of $W$ that is $\JJ$-invariant and complementary to $W_0$ in $W$. Since $\JJ$ is compatible with $\omega$, any subspace of $W$ that is both $\JJ$-invariant and complementary to $W_0$ in $W$ coincides with both the symplectic complement and $\omega(\cdot, \JJ \cdot)$-orthogonal complement of $W_0 \oplus \JJ W_0$ in $W\oplus \JJ W_0$. Moreover, such a subspace coincides with the $\omega(\cdot, \JJ \cdot)$-orthogonal complement of $W_0$ in $W$, and is thus unique. (3): Let $W_+^\JJ$ be the subspace of $W$ given by (2). Let $W_-^\JJ$ be the $\omega(\cdot, \JJ \cdot)$-orthogonal complement of the $\JJ$-invariant subspace $W_0 \oplus \JJ W_0 \oplus W_+^\JJ$ in $V$. Since $\JJ$ is $\omega$ compatible, $W_-^\JJ$ is $\JJ$-invariant and contained in $W^\omega$. It is also transversal to $W_0 \subseteq W^\omega$, so by counting dimensions $W^\omega = W_0 \oplus W_-^\JJ$. Then $W^\omega \oplus \JJ W_0= (W_0 \oplus \JJ W_0) \oplus W_-^\JJ$ and is $\JJ$-invariant. We obtain a splitting $(W_+^\JJ, W_-^\JJ, \JJ W_0)$ of $V$ associated to $W$ such that $W_+^\JJ$, $W_-^\JJ$ are $\JJ$-invariant. Take a Darboux basis given by Theorem \ref{thm:basisR}, and we can replace the $\ff^+$s and $\ff^-$s by $\JJ\ee^+$s and $\JJ \ee^-$s to obtain the desired basis.
\end{proof}

The incidence relations of the closures of the Grassmannians $\Gr(\vec{n})$ were stated in Section 2.2 of \cite{LeeLeung}:
\begin{proposition} \label{prop:incidenceR}
Suppose $n_0 + n_+ + n_- = n = n_0' + n_+' + n_-'$ and $n_0 + 2n_+ = n_0' + 2n_+' = \ell$. Then the following are equivalent
\begin{enumerate}
    \item $\Gr(\vec{n}'; V) \cap \overline{\Gr(\vec{n}; V)}$ is nonempty.
    \item $\Gr(\vec{n}';V) \subseteq \overline{\Gr(\vec{n};V)}$.
    \item $n_+' \le n_+$.
\end{enumerate}
\end{proposition}
\begin{proof} $(1) \iff (3)$: Consider the frame space $\Fr(\Gr(\ell; V)) \subseteq V^{\ell} \times \Gr(\ell; V)$ consisting of $((\vv_1, \cdots, \vv_\ell), W)$ such that $W = \span_\R\{\vv_1, \cdots, \vv_\ell\}$. The map that sends $((\vv_1, \cdots, \vv_\ell), W)$ to the $\ell \times \ell$ skew symmetric matrix $(\omega(\vv_j, \vv_k))_{j, k}$ is smooth. The map that sends a skew symmetric matrix to its eigenvalues takes values in the metric space of configurations of $\ell$ unordered points in $i\R$ that are symmetric about the origin. This space is stratified by how many points are zero and nonzero. By the sequence lifting property (Proposition 2.4 of \cite{Siwiec}) of these maps, any convergent sequence on the space of $\ell$ unordered points lifts to a convergent sequence on $\Fr(\Gr(\ell; V))$. By the continuity of the the projection map $\Fr(\Gr(\ell; V)) \to \Gr(\ell; V)$, any convergent sequence on $\Fr(\Gr(\ell; V))$ descends to a convergence sequence on $\Gr(\ell; V)$. For the reverse direction, we note that all the maps involved are both continuous and have the sequence lifting property. $(1) \iff (2)$: This follows from the transitivity of the continuous $\Sp(V)$ action.
\end{proof}

\begin{corollary} \label{cor:compact}
    $\Gr(\vec{n}; V)$ is compact if and only if $n_+n_- = 0$.
\end{corollary}
\begin{proof}
By Proposition \ref{prop:incidenceR}, if $n_+$ and $n_-$ are both nonzero, $\Gr(\vec{n})$ is not closed. The subgroup of all $g \in \Sp(V)$ that commute with $\JJ$ is compact, and by Theorem \ref{thm:basisJ} acts transitively on $\Gr(\vec{n})$ if $n_+n_- = 0$.
\end{proof}

\subsection{Orbits of the Heisenberg group}
\label{subsec:orbitsheisenberg}

For a linear subspace $W \subseteq V$, let $\Sp_W(V)$ denote the subgroup of linear transformations $g \in \Sp(V)$ such that $g.W = W$. If $g \in \Sp_W(V)$, $g.W^\omega = W^\omega$, so $\Sp_W(V) = \Sp_{W^\omega}(V)$. Moreover, $g.W_0 = W_0$ so $g$ descends to symplectic linear transformations of $W/W_0$ and $W^\omega/W_0$, which we denote by $\left[g\right]$ and $\left[g\right]^\omega$, respectively.

\begin{definition}[Heisenberg group]
\label{def:heisenberg}
Suppose $W\subseteq V$ is a real subspace. Let $H(W) \le \Sp_W(V)$ be the subgroup of symplectic linear transformations $g$ of $V$ such that $g$ restricts to the identity transformation on $W_0$ and $g$ descends to the identity transformation on both $W/W_0$ and $W^\omega/W_0$. Denote by $Z(W) \le H(W)$ the subgroup of symplectic linear transformations that restrict to the identity transformation on both $W$ and $W^\omega$.
\end{definition}

Suppose $\{\ee, \ff\}$ is a basis given by Theorem \ref{thm:basisR} for some $W$ of type $\vec{n}$ and associated splitting $(W_+, W_-, W^0)$. In the reordered basis $\{\ee^0, \ee^+, \ff^+, \ee^-, \ff^-, \ff^0\}$, elements of $H(W)$ take the block form
\begin{equation} \label{eq:heisenbergtriangularform} \begin{pmatrix} 1_{n_0} & E_+ & F_+ & E_- & F_- & Y \\
0 & 1_{n_+} & 0 & 0 & 0 & F_+^t\\
0 & 0 & 1_{n_+} & 0 & 0 & -E_{+}^t\\
0 & 0 & 0& 1_{n_-} & 0 & F_-^t\\
0 & 0 & 0 & 0& 1_{n_-} & -E_-^t\\
0 & 0 & 0 & 0 & 0& 1_{n_0}  \end{pmatrix}, \end{equation}
where $1_k$ is the $k \times k$ identity matrix and
\begin{equation} \label{eq:heisenbergcondition}
Y - E_+ F_+^t - E_- F_-^t \in \Mat_{n_0 \times n_0}(\R)^t.
\end{equation} The matrix form shows $H(W)$ is nilpotent, contractible, and \begin{equation} \label{eq:heisenbergdimension}
\dim_\R H(W) = 2n_0 n_+ + 2n_0 n_- + \frac{1}{2}n_0(n_0 + 1).
\end{equation}
$H(W)$ is trivial if $W$ is symplectic, abelian if $W$ is Lagrangian and nonabelian otherwise. From the definition, we can also immediately see $H(W) = H(W^\omega)$, and $Z(W) = Z(W^\omega) = Z(W + W^\omega) = Z(W_0)$.

The matrix form in Equation (\ref{eq:heisenbergtriangularform}) shows that $Z(W)$ is the center of $H(W)$. In a Darboux basis obtained from any splitting associated to $W$, $Z(W)$ consists of elements such that $E_{\pm} = F_{\pm} = 0_{n_0 \times n_{\pm}}$. Since $W^0$ can be identified with the dual space of $W_0$ from the symplectic form, $Z(W)$ can be identified with the abelian group of quadratic forms on $W_0$ under addition. In particular, $\dim_\R Z(W) = n_0(n_0+ 1)/2$. $H(W)/Z(W)$ is isomorphic to the abelian group of real linear maps from $W_0$ to $W_0^\omega/W_0$ under addition $(\Hom_\R(W_0, W_0^\omega/W_0), +)$. So $H(W)$ can be viewed as a central extension of groups
\begin{equation}\label{eq:heisenbergextensionnonsplit}
    Z(W) \hookrightarrow H(W) \twoheadrightarrow \Hom_\R(W_0, W_0^\omega/W_0).
\end{equation}

By taking a Darboux basis change of $W_+$ and $W_-$ if necessary, for a choice of Lagrangian subspaces $L_+ \subseteq W_+$ and $L_- \subseteq W_-$, the abelian groups $(\Hom_\R(W_0, L_{\pm}), +)$ embed into $H(W)$, where their images are given, respectively, by the vanishing of $F_+$, $E_-$, $F_-$, and $Y$, and by the vanishing of $E_+$, $F_+$, $F_-$ and $Y$. Then we have a split group extension
\begin{equation}\label{eq:heisenbergextensionsplit}
    Z(W) \times \Hom_\R(W_0, L_+ \oplus L_-)  \hookrightarrow H(W) \twoheadrightarrow \Hom_\R(W_0, W_+/L_+ \oplus W_-/L_-).
\end{equation}

In the usual Darboux basis $\{\ee^0, \ee^+, \ee^-, \ff^0, \ff^+, \ff^-\}$, $g \in H(W)$ takes the block form
\begin{equation}\label{eq:heisenbergdarbouxform}
\begin{pmatrix} 1_{n_0} & E_+ & E_- & Y & F_+ & F_- \\
0 & 1_{n_+} & 0 & F_+^t & 0 & 0 \\
 0 & 0 & 1_{n_-} & F_-^t & 0 & 0 \\
 0 & 0 & 0 & 1_{n_0} & 0 & 0 \\
 0 & 0 & 0& -E_+^t & 1_{n_+} & 0 \\
 0 & 0 & 0 & -E_-^t & 0 & 1_{n_-} \end{pmatrix}.\end{equation}
We will denote the matrix group of $2n \times 2n$ matrices in this block form and satisfying Equation (\ref{eq:heisenbergcondition}) as $H(\vec{n})$.

\begin{remark}[Symmetries of Jacobi forms of higher degree]
\label{rem:heisenberg}
A particular matrix form of $H(W)$ has been referred to as \emph{the Heisenberg group} in \cite{Ziegler} to define Jacobi forms of higher degree. In the Darboux basis $\{ \ee^+, \ee^-, \ee^0, \ff^+, \ff^-, \ff^0\}$, $g \in H(\vec{n})$ takes the block matrix form
 \begin{equation} \label{eq:heisenbergzieglerform} \begin{pmatrix}
     1_{n_+} & 0 & 0 & 0 & 0 & F_+^t \\ 0 & 1_{n_-} & 0 & 0 & 0& F_-^t \\
     E_+ & E_- & 1_{n_0} & F_+  & F_- & Y \\
     0 & 0 & 0 & 1_{n_+} &0 & -E_+^t \\
     0 & 0 & 0 & 0 & 1_{n_-} & -E_-^t\\  0 & 0 & 0 & 0 & 0 & 1_{n_0} 
\end{pmatrix}\end{equation}
and we can identify with Ziegler's notation for $H_{\R}^{(n_+ + n_-, n_0)}$ by  $\lambda = \begin{pmatrix} E_+ & E_- \end{pmatrix}$, $\mu = \begin{pmatrix} F_+ & F_-\end{pmatrix}$, and $x = -Y$.
\end{remark}

Suppose $g \in \Sp_W(V)$, and $\{\ee, \ff\}$ is a Darboux basis provided by Theorem \ref{thm:basisR}. In the $\{ \ee^0, \ee^+, \ff^+, \ee^-, \ff^-, \ff^0\}$ basis, $g$ takes the block matrix form
 \[ \begin{pmatrix} \ast & \ast & \ast & \ast & \ast & \ast \\
 0_{n_+ \times n_0} & \ast & \ast & 0_{n_+ \times n_-} & 0_{n_+ \times n_-} & \ast\\
 0_{n_+ \times n_0} & \ast & \ast & 0_{n_+ \times n_-} & 0_{n_+ \times n_-} & \ast\\
 0_{n_- \times n_0} & 0_{n_- \times n_+} & 0_{n_- \times n_+} & \ast & \ast & \ast\\
 0_{n_- \times n_0} & 0_{n_- \times n_+} & 0_{n_- \times n_+} & \ast & \ast & \ast\\
 0_{n_0 \times n_0} & 0_{n_0 \times n_+} & 0_{n_0 \times n_+} & 0_{n_0 \times n_-} & 0_{n_0 \times n_-} & \ast \end{pmatrix}\]
 This matrix can be decomposed as a product of two matrices in block form:
\[ \begin{pmatrix} X & 0 & 0 & 0 & 0 & 0\\
0 & A_+ & B_+ & 0 & 0 & 0 \\
0 & C_+ & D_- & 0 & 0 & 0 \\
0 & 0 & 0 & A_- & B_- & 0 \\
0 & 0 & 0 & C_- & D_-& 0 \\
0 & 0 & 0 & 0 & 0 & (X^{t})^{-1} \end{pmatrix} \begin{pmatrix} 1_{n_0} & E_+ & F_+ & E_- & F_- & Y \\
0 & 1_{n_+} & 0 & 0 & 0 & F_+^t\\
0 & 0 & 1_{n_+} & 0 & 0 & -E_{+}^t\\
0 & 0 & 0& 1_{n_-} & 0 & F_-^t\\
0 & 0 & 0 & 0& 1_{n_-} & -E_-^t\\
0 & 0 & 0 & 0 & 0& 1_{n_0}  \end{pmatrix},\]
where, the assumption that $g$ is symplectic is equivalent to
\[ \left( X, \begin{pmatrix} A_{+} & B_{+} \\ C_{+} & D_{+}\end{pmatrix}, \begin{pmatrix} A_{-} & B_{-} \\ C_{-} & D_{-}\end{pmatrix}\right) \in \GL(n_0; \R) \times \Sp(2n_{+}; \R) \times \Sp(2n_-; \R),\]
and $Y - E_+ F_+^t - E_- F_-^t \in \Mat_{n_0 \times n_0}(\R)^t$. From this observation we can deduce the following:

\begin{proposition}\label{prop:homogeneousR}
If $W \subseteq V$ is a real subspace, the set of splittings of $V$ associated to $W$ is a principal homogeneous space for $H(W)$. The set of pairs of splittings $\{(W_+, W_-): W = W_0 \oplus W_+, W^\omega = W_0\oplus W_-\}$ is a principal homogeneous space for $H(W)/Z(W)$.
\end{proposition}
\begin{proof}
For two splittings $(W_+, W_-, W^0)$ and $(W_+', W_-', (W^0)')$ of $V$ associated to $W$, obtain two Darboux bases $\{ \ee, \ff\}$ and $\{ \ee', \ff'\}$ by Theorem \ref{thm:basisR}. The linear extension of the map that sends $\{ \ee, \ff\}$ to $\{ \ee', \ff'\}$ is a symplectic linear transformation that stabilizes $W$. Express it in the $\{\ee^0, \ee^+, \ff^+, \ee^-, \ff^-, \ff^0\}$ basis. Then taking the upper triangular part of the matrix decomposition, we obtain a unique element of $H(W)$ that maps $(W_+, W_-, W^0)$ to $(W_+', W_-', (W^0)')$. $H(W)$ acts on $\{(W_+, W_-)\}$ by $g.(W_+, W_-) = (g.W_+, g.W_-)$ with kernel $Z(W)$.
\end{proof}

$\Sp_W(V)$ can be written as a semidirect product in different ways.

\begin{theorem}\label{thm:splitW}
Suppose $(W_+, W_-, W^0)$ is a splitting of $V$ associated to $W$. Then the group extensions
\begin{equation} \label{eq:extensionRmax} \begin{tikzcd}[column sep = 2em] H(W) \arrow[r, hook] & \Sp_W(V) \arrow{r}[two heads, yshift=-4ex]{{\scriptscriptstyle\cdot|_{W_0} \times [\cdot] \times [\cdot]^\omega}}&\GL(W_0) \times \Sp(W/W_0) \times \Sp(W^\omega/W_0)\end{tikzcd} \end{equation}
\begin{equation} \label{eq:extensionRred} \begin{tikzcd}[column sep = small] (\GL(W_0) \times \Sp(W_-))\ltimes H(W) \arrow[r, hook] & \Sp_W(V) \arrow{r}[two heads, yshift = -3ex]{{
\scriptscriptstyle[\cdot]}} & \Sp(W/W_0)\end{tikzcd}\end{equation}
\begin{equation} \label{eq:extensionRker} \begin{tikzcd} (\Sp(W_+) \times \Sp(W_-)) \ltimes H(W) \arrow[r, hook] & \Sp_W(V) \arrow{r}[two heads, yshift = -3ex]{{\cdot|_{W_0}}} & \GL(W_0) \end{tikzcd}\end{equation}
are right split.
\end{theorem}
\begin{proof}
Using a Darboux basis given by Theorem \ref{thm:basisR}, we can identify $\Sp(W_+) \cong \Sp(W/W_0)$ and $\Sp(W_-) \cong \Sp(W^\omega/W_0)$ and identify the elements with their matrix representations. The splitting homomorphism of the extension (\ref{eq:extensionRmax}) maps 
\[ \left( X, \begin{pmatrix} A_{+} & B_{+} \\ C_{+} & D_{+}\end{pmatrix}, \begin{pmatrix} A_{-} & B_{-} \\ C_{-} & D_{-}\end{pmatrix}\right) \in \GL(W_0) \times \Sp(W_+) \times \Sp(W_-)\]
to \[\begin{pmatrix} X & 0 & 0 & 0 & 0 & 0 \\
0 & A_+ & 0 & 0 & B_+ & 0 \\ 0 & 0 & A_- & 0 & 0 & B_-\\ 0 & 0 & 0 & (X^t)^{-1} & 0 & 0 \\ 0 & C_+ & 0 & 0 & D_+ & 0 \\ 0 & 0 & C_- & 0 &0 & D_- \end{pmatrix} \in \Sp(V). \]
The other splitting homomorphisms are obtained by restricting the components of $\GL(W_0) \times \Sp(W/W_0) \times \Sp(W^\omega/W_0)$ to identity elements.
\end{proof}

\subsection{Orbits in complex Lagrangian Grassmannians}
\label{subsec:orbitscomplex}

For a real vector space $W$, denote its complexification $\C \otimes_\R W$ by $W^\C$. Write an elementary tensor as $a \otimes_\R \ww$ of $W^\C$ as  $a\ww$, and denote its complex conjugate $\overline{a} \otimes_\R \ww$ by $\overline{a\ww}$ or $\overline{a} \ww$. Extend complex conjugation linearly to an $\R$-linear involution on $W^\C$ and denote the complex conjugate of $\vv \in V^\C$ by $\overline{\vv}$. For a linear transformation $g$ of $V^\C$, denote its complex conjugate $\overline{g}\vv := \overline{g \overline{\vv}}$.

Let $\omega^\C$ be the complex bilinear extension of $\omega$ to $V^\C$. Let $\Sp(V^\C)$ be the group of complex invertible linear transformations of $V^\C$ that preserve $\omega^\C$. $\Sp(V)$ embeds into $\Sp(V^\C)$ by scalar extension of linear transformations, and identify $\Sp(V)$ as a subgroup of $\Sp(V^\C)$.

Call a complex subspace $\FF \subseteq V^\C$ \emph{a complex Lagrangian subspace} if it has complex dimension $n$ and $\omega^\C|_{\FF \otimes_\C \FF}$ vanishes identically. Denote the space of all complex Lagrangian subspaces of $V^\C$ by $\Lag^\C(V)$ and call it \emph{the complex Lagrangian Grassmannian}. If we choose an identification of $V^\C$ with $\C^{2n}$, $\Sp(V^\C)$ is identified with a linear algebraic group, so $\Lag^\C(V)$ becomes a a (complex) projective variety with complex dimension $n(n+1)/2$. If we choose an $\omega$-compatible linear complex structure $\JJ$, $\Lag^\C(V)$ becomes a K\"{a}hler manifold (Remark \ref{rem:complexstructures}), and in fact, a hermitian symmetric space.

$\Sp(V)$ acts on $\Lag^\C(V)$ by $g.\FF := \{ g \vv : \vv \in \FF\}$ for which only $\pm 1_{\Sp(V)}$ acts trivially. This action can be viewed as a generalization of the M\"{o}bius action of $\SL(2; \R)$ on the Riemann sphere (Section \ref{subsec:mobius}).

$\kappa(\vv, \ww):= -i \omega^\C(\vv, \overline{\ww})$ is a hermitian form with split signature $(n, n)$ on $V^\C$. It restricts to a hermitian form on any complex subspace of $V^\C$, in particular for any complex Lagrangian subspace $\FF$.

As before, let $\vec{n}:= (n_0, n_+, n_-)$ be an ordered triple of nonnegative integers that sum to $n$. We will call $\FF$ a (complex) Lagrangian subspace of \emph{type $\vec{n}$} if $\kappa|_\FF$ has signature (inertia) $(n_0, n_+, n_-)$. Let $\FF_0$ be the kernel of $\kappa|_{\FF}$ and we will call $(\FF, \FF_+, \FF_-)$ a \emph{splitting of $\FF$} if $\FF = \FF_0 \oplus \FF_+ \oplus \FF_-$ and both $\kappa|_{\FF_+}$ and $-\kappa|_{\FF_-}$ are positive definite. The null cone of $\FF$, $\{ \vv \in \FF: \kappa(\vv, \vv) = 0\}$ is a cartesian product of $\FF_0$ and the null cone of $\FF_+ \oplus \FF_-$, $\{\vv \in \FF_+ \oplus \FF_- : \kappa(\vv, \vv) = 0\}$.

Let $\Lag^\C(\vec{n}; V) = \Lag^\C(\vec{n})$ denote the set of complex Lagrangian subspaces of $V^\C$ of type $\vec{n}$. Since every $g \in \Sp(V)$ preserves $\kappa$, the generalized M\"{o}bius action stabilizes every $\Lag^\C(\vec{n})$. In fact, the $(n+1)(n+2)/2$ elements of $\{\Lag^\C(\vec{n})\}_{\vec{n}}$ are precisely the orbits of $\Sp(V)$ in $\Lag^\C(V)$. Since any linear map that takes a Darboux basis to a Darboux basis is symplectic, the transitivity of the generalized M\"{o}bius action is a consequence of the following basis theorem.

\begin{theorem}[\label{thm:basisC}Lemma 5.1 of \cite{BlattnerRawnsley}]
Suppose $\FF$ is a complex Lagrangian subspace of $V^\C$ of type $\vec{n}$ with splitting $\FF = \FF_0 \oplus \FF_+ \oplus \FF_-$. Then there exists a Darboux basis of $\R^{2n}$
\[ \{\ee^0_1, \cdots, \ee^0_{n_0}, \ee^+_1, \cdots, \ee^+_{n_+}, \ee^-_1, \cdots, \ee^-_{n_-}, \ff^0_1, \cdots, \ff^0_{n_0}, \ff^+_{1}, \cdots, \ff^+_{n_+}, \ff^-_1, \cdots, \ff^-_{n_-}\}\]
such that $\FF_0 = \span_\C\{ \ee^0_j\}_{1 \le j \le n_0}$, $\FF_+ = \span_\C\{ \ee^+_k - i \ff^+_k\}_{1 \le k \le n_+}$, and $\FF_- = \span_\C\{\ee^-_\ell + i \ff^-_\ell\}_{1 \le \ell \le n_-}$.
\end{theorem}

\begin{proof}
In the original reference the proof is left as an exercise for the reader. We refer to Theorem 2.3.14, Proposition 2.6.5 of \cite{mythesis} for explicit proofs.
\end{proof}

\begin{remark}[The Real Lagrangian Grassmannian]
As a corollary, $\Lag^\C(n, 0, 0; V)$ can be identified with image of the embedding $\Gr(n, 0, 0; V) \xhookrightarrow{\C \otimes_\R \cdot} \Lag^\C(V)$.
\end{remark}

The incidence relations between the orbit closures are referred to as the Orbit Structure Theorem in \cite{Wolf}, and is also stated in Theorem A.2 of \cite{Takeuchi}, Theorem 10.6.3 of \cite{Wolf1} (c.f. Boundary Orbit theorem of \cite{Wolf}). An example is shown in Figure \ref{fig:incidence}. We provide a different proof in the specific case we are considering.

\begin{theorem}[Special case of Theorem 10.6.3 of \cite{Wolf1}, Theorem A.2 of \cite{Takeuchi}] \label{thm:incidence}
Suppose $n_0 + n_+ + n_- = n_0' + n_+' + n_-' = n$. The following are equivalent
\begin{enumerate}
    \item $\Lag^\C(\vec{n}';V) \cap \overline{\Lag^\C(\vec{n};V)}$ is nonempty.
    \item $\Lag^\C(\vec{n}';V) \subseteq \overline{\Lag^\C(\vec{n};V)}$.
    \item $n_-' \le n_-$ and $n_+' \le n_+$.
\end{enumerate}
\end{theorem}

\begin{proof} $(1) \iff (3)$: Consider the frame space $\Fr(\Lag^\C(V)) \subseteq (V^\C)^n \times \Lag^\C(V)$ consisting of $((\vv_1, \cdots, \vv_n), \FF)$ such that $\FF = \span_\C\{\vv_1, \cdots, \vv_n\}$.

The map that assigns to a frame  $((\vv_1, \cdots, \vv_n), \FF)$ the $n\times n$ hermitian matrix $(\kappa(\vv_j, \vv_\ell))_{j, \ell}$ is smooth, and the map that takes an $n \times n$ hermitian matrix to its eigenvalues is also continuous. The eigenvalues take values in the metric space of $n$ unordered points on $\R$, which is stratified by how many points are positive, zero, and negative. In this space, one stratum has a nonempty intersection with the closure of another if and only if it does not have more positive points and more negative points than the other.

As in the proof of Theorem \ref{prop:incidenceR}, this stratification is ``pulled back'' to a stratification on $\Fr(\Lag^\C(V))$ via the continuity and sequence lifting property (cf. Proposition 2.4 \cite{Siwiec}) of the aforementioned maps. It is ``pushed forward'' to the desired stratification on $\Lag^\C(V)$ via the continuity and sequence lifting property of the projection map $((\vv_1, \cdots, \vv_n), \FF) \mapsto \FF$. $(1) \iff (2)$: This follows from the transitivity of the continuous action of $\Sp(V)$.
\end{proof}

\begin{remark}[Frame space as a Hilsum-Skandalis bibundle]\label{rem:bibundle}
We first notice that a smooth $G$-equivariant map $\pi: E \to B$ consists of three pieces of data that happily satisfy a compatibility condition: the smooth action of $G$ on $E$, the smooth action of $G$ on $B$, and the map $\pi$. We can view this object from a groupoid perpective by switching the order in which the data is given. We first construct an action Lie groupoid $G \times B \rightrightarrows B$ from the action of $G$ on $B$, and then observe $\pi$ is an anchor map of the action (cf. Definition 3.22 of \cite{Lerman}) of the (action) Lie groupoid $G \times B \rightrightarrows B$ on $E$. This also holds for $G$-equivariant maps for right actions of $G$ on $E$ and $B$.

The frame space $\Fr(\Lag^\C(V))$ constructed in the proof of Theorem \ref{thm:incidence} admits a right $\GL(n; \C)$ action by taking linear combinations of the $\vv_j$s and a left $\Sp(V)$ action acting diagonally on all components. The projection to $\Lag^\C(V)$ is $\Sp(V)$-equivariant and $\GL(n; \C)$-invariant and they commute. From a groupoid perspective, one can check $\Fr(\Lag^\C(V))$ is a \emph{(Hilsum-Skandalis) bibundle} (Definition 3.25 of \cite{Lerman})

\begin{equation}
\begin{tikzcd}
    \Sp(V) \times \Lag^\C(V) \arrow[d, shift left = .75ex] \arrow[d, shift right = .75ex]  & \Fr(\Lag^\C(V)) \arrow[dl, "a_L"] \arrow[l, symbol = \circlearrowright] \arrow[dr, "a_R"'] & \GL(n; \C) \times \Mat_{n\times n}(\C)^\dagger \arrow[l, symbol = \circlearrowleft] \arrow[d, shift right = 0.75ex] \arrow[d, shift left = 0.75ex] \\
    \Lag^{\C}(V) & & {\Mat}_{n\times n}(\C)^{\dagger}\end{tikzcd}
\end{equation}

The rightmost column is the action Lie groupoid of the action of $\GL(n; \C)$ on $\Mat_{n\times n}(\C)^\dagger$ by $g^\dagger A g$, and $a_R(\vv_1, \cdots, \vv_n) = (\kappa(\vv_j, \vv_k))_{j, k}$ is the anchor map of its right action on $\Fr(\Lag^\C(V))$. The leftmost column is the action Lie groupoid of the generalized M\"{o}bius action. $a_L(\vv_1, \cdots, \vv_n) = \span_\C\{\vv_1, \cdots, \vv_n\}$ is the anchor map of its left action on $\Fr(\Lag^\C(V))$. $a_L$ is also a principal bundle (cf. Definition 3.17 of \cite{Lerman}) for the groupoid in the rightmost column, and $a_R$ is invariant with respect to the left action of the groupoid in the leftmost column, making $\Fr(\Lag^\C(V))$ a (Hilsum-Skandalis) bibundle.
\end{remark}

\begin{remark}
Since $\Sp(V^\C)$ is simple, our notation and Wolf's notation in the Orbit Structure Theorem of \cite{Wolf} (also Theorem 10.6.3 of \cite{Wolf1}) are identified via $|\Psi| = n$, $|\Gamma| = n_0$, and 
\[n_+ = |\Sigma \setminus \Gamma| \le |\Sigma| \le |\Sigma \cup \Gamma| = n_0 + n_+.\]
In the Boundary Orbit Theorem of \cite{Wolf}, $|\Gamma| = n_+$. 
We refer to Appendix \ref{subsec:partialcayley}, \ref{subsec:symmetriesholarc} for details.
\end{remark}

\subsection{Real projections of complex Lagrangian subspaces}
\label{subsec:projections}
Consider the projection map $\Re(\vv):= (\vv + \overline{\vv})/2$ from $V^\C$ to $V$ along $iV$. As a consequence of Theorem \ref{thm:basisC}, for a complex Lagrangian subspace $\FF$
\begin{eqnarray}\label{eq:realprojectionimage}
    (\Re\FF_0)^\C &=& \FF_0\\
    \Re\FF_0 &=& (\Re \FF)^\omega \subseteq \Re \FF. \nonumber
\end{eqnarray}

$\Re \FF$ is coisotropic, so $\omega|_{\Re \FF}$ descends to a symplectic form on $\Re\FF / (\Re \FF)^\omega$. Moreover, $\FF + \overline{\FF}$ is the $\omega^\C$-symplectic complement of $\FF_0$, so $\omega^\C|_{\FF + \overline{\FF}}$ descends to a complex bilinear symplectic form on $(\FF + \overline{\FF})/\FF_0$. So
\[(\FF + \overline{\FF})/\FF_0 = (\Re \FF)^\C / ((\Re \FF)^\omega) ^\C \cong (\Re \FF / (\Re \FF)^\omega)^\C\]
is an identification of complexified symplectic vector spaces.

\begin{remark}[Polarizations]
When $n=1$, a point in $\Lag^\C(\R^2)$ can be viewed as a polarization of a ray of light (as an equivalence class of Jones vectors). In the literature of geometric quantization (eg Section 4.1 of \cite{Sniaticki} Section V.3 of \cite{GuilleminSternberg}, Section 5.4 of \cite{Woodhouse}), distributions of the complexified tangent bundle of a symplectic manifold fiberwise satisfying the complex Lagrangian condition has been referred to as \emph{polarizations}. When $n_- = 0$ the following construction was used on the level of spaces, but not of the elements:
\begin{equation*}
\begin{aligned}
\Re \FF_0 &=& \mathbf{D} &=&  V \cap \FF \cap \overline{\FF}\\
\Re \FF &=& \mathbf{E} &=& V \cap (\FF + \overline{\FF}).
\end{aligned}
\end{equation*}
\end{remark}

For a complex Lagrangian subspace $\FF$, let $\Sp_\FF(V)$ denote the subgroup of linear transformations $g \in \Sp(V)$ such that $g.\FF = \FF$. Since $\Re(g\vv) = g\Re(\vv)$ for all $\vv \in V^\C$, $\Sp_\FF(V)$ is a subgroup of $\Sp_{\Re \FF}(V)$. Choose a splitting $\FF = \FF_0 \oplus \FF_+ \oplus \FF_-$, and observing how $H(\Re \FF)$ acts on the Darboux basis provided by Theorem \ref{thm:basisC}, we have $H(\Re \FF) \le \Sp_\FF(V) \le \Sp_{\Re \FF}(V)$.

If $g \in \Sp_\FF(V)$, then $g.\FF_0 = \FF_0$, and $g.\Re\FF_0 = \Re \FF_0$, so $g|_{\Re\FF_0} \in \GL(\Re \FF_0)$. Moreover, $\kappa|_\FF$ descends to a nondegenerate mixed signature hermitian form on $\FF / \FF_0$. Denote by $\U(\FF/\FF_0)$ the automorphisms of this hermitian vector space. Then $g \in \Sp_\FF(V)$ descends to a unitary transformation on $\FF/\FF_0$ that we denote by $[g]_0$.

As before, $\Sp_{\FF}(V)$ can be written as a semidirect product in different ways.

\begin{theorem}\label{thm:splitF} Suppose $\FF$ is a complex Lagrangian subspace. The group extensions
\begin{equation}
    \begin{tikzcd}H(\Re \FF) \arrow[r, hook] & \Sp_\FF(V) \arrow[two heads]{r}[yshift = -4ex]{{\cdot|_{\Re \FF_0} \times {[\cdot]_0}}}& \GL(\Re \FF_0) \times \U(\FF/\FF_0) \end{tikzcd}
\end{equation}
\begin{equation}
    \begin{tikzcd} \GL(\Re \FF_0) \ltimes H(\Re \FF) \arrow[r, hook] & \Sp_\FF(V) \arrow[r, two heads, "{[\cdot]_0}"'] & \U(\FF/\FF_0) \end{tikzcd}
\end{equation}
\begin{equation}
    \begin{tikzcd} \U(\FF/\FF_0)\ltimes H(\Re \FF) \arrow[r, hook] & \Sp_\FF(V) \arrow[r, two heads, "\cdot|_{\Re \FF_0}"'] & \GL(\Re \FF_0) \end{tikzcd}
\end{equation}
are right split.
\end{theorem}
\begin{proof}
Choose a splitting $\FF = \FF_0 \oplus \FF_+ \oplus \FF_-$ and take a Darboux basis $\{\ee, \ff\}$ given by Theorem \ref{thm:basisC}. Since $\Sp_\FF(V) \le \Sp_{\Re \FF}(V)$, in the $\{ \ee^0, \ee^+, \ee^-, \ff^+, \ff^-, \ff^0\}$ basis, $g \in \Sp_\FF(V)$ can be expressed as a product $g'\cdot h$ where $(g' )_{\{ \ee^0, \ee^+, \ee^-, \ff^+, \ff^-, \ff^0\}}$ has block form
\begin{equation} \begin{pmatrix} X & 0 & 0 & 0 & 0 & 0\\
0 & A_{++} & A_{+-} & B_{++} & B_{+-} & 0 \\
0 & A_{-+} & A_{--} & B_{-+} & B_{--} & 0 \\
0 & C_{++} & C_{+-} & D_{++} & D_{+-} & 0 \\
0 & C_{+-} & C_{--} & D_{-+} & D_{--}& 0 \\
0 & 0 & 0 & 0 & 0 & (X^{t})^{-1} \end{pmatrix}\end{equation} and $(h)_{\{ \ee^0, \ee^+, \ee^-, \ff^+, \ff^-, \ff^0\}}$ has block form \begin{equation}\begin{pmatrix} 1_{n_0} & E_+ & E_- & F_+ & F_- & Y \\
0 & 1_{n_+} & 0 & 0 & 0 & F_+^t\\
0 & 0 & 1_{n_+} & 0 & 0 & F_{-}^t\\
0 & 0 & 0& 1_{n_-} & 0 & -E_+^t\\
0 & 0 & 0 & 0& 1_{n_-} & -E_-^t\\
0 & 0 & 0 & 0 & 0& 1_{n_0}  \end{pmatrix}.\end{equation}
Since $g.\FF = \FF$, in the complex Darboux basis whose first $n$ vectors span $\FF$
\[ \left \{ \ee^0, \frac{1}{\sqrt{2}}(\ee^+ - i \ff^+), -\frac{i}{\sqrt{2}}(\ee^- +i \ff^-), -\frac{i}{\sqrt{2}} (\ee^+ + i \ff^+), \frac{1}{\sqrt{2}} (\ee^- - i \ff^-), \ff^0 \right\}\]
the lower left $n \times n$ block of the matrix form of $g'$, which is equal to
\[ \frac{i}{2} \begin{pmatrix} 0 & 0 & 0 \\ 0 & A_{++} - D_{++} & A_{+-} + D_{+-} \\ 0 & -A_{-+} - D_{-+} & -A_{--} + D_{--} \end{pmatrix} + \frac{1}{2} \begin{pmatrix} 0 & 0 & 0 \\ 0 & B_{++} + C_{++} & - B_{+-} + C_{+-} \\ 0 & -B_{-+} + C_{-+} &  B_{--} + C_{--} \end{pmatrix},\]
vanishes. This implies that the upper left $n\times n$ block of the matrix form of $g'$, acting unitarily on $(\FF, \kappa|_\FF)$, is
\begin{equation} \label{eq:indefunitary}\begin{pmatrix} X & 0 & 0 \\ 0 & A_{++} - i B_{++} & A_{+-} + i B_{+-} \\ 0 & A_{-+} - i B_{-+} & A_{--} + i B_{--} \end{pmatrix} = \begin{pmatrix} X & 0 & 0 \\ 0 & D_{++} + i C_{++} & - D_{+-} + i C_{+-} \\ 0 & -D_{-+} - i C_{-+} & D_{--} + i C_{--} \end{pmatrix}.\end{equation}

The vanishing of the lower left $n\times n$ block of the matrix form of $h$ gives no additional constraints, so the upper left $n\times n$ block of the matrix form of $h$ acts unitarily on $(\FF, \kappa|_\FF)$ as
\[ \begin{pmatrix} 1_{n_0} & \frac{1}{\sqrt{2}}(E_{+} - i F_+) & \frac{1}{\sqrt{2}}(E_{-} + i F_-) \\ 0 & 1_{n_+} & 0 \\ 0 & 0 & 1_{n_-} \end{pmatrix}.\]

If we take $g \in \U(\FF/\FF_0)$, in the basis taken from the image of 
\[ \left \{ \frac{1}{\sqrt{2}}(\ee^+ - i \ff^+), -\frac{i}{\sqrt{2}}(\ee^- +i \ff^-), -\frac{i}{\sqrt{2}} (\ee^+ + i \ff^+), \frac{1}{\sqrt{2}} (\ee^- - i \ff^-) \right\}\]
under the quotient map, it has block form 
\[\begin{pmatrix} g_{++} & g_{+-} \\ g_{-+} & g_{--} \end{pmatrix}.\]
Then the splitting morphism takes
\[ \left( X, g \right) \in \GL(\Re \FF_0) \times \U(\FF/\FF_0)\] to
\begin{equation}\label{eq:embindefunitary}  \begin{pmatrix} X & 0 & 0 & 0 & 0 & 0\\ 0 & \Re g_{++} & \Re g_{+-} & 0 &  -\Im g_{++} &  \Im g_{+-} \\ 0 & \Re g_{-+} & \Re g_{--} & 0 & -\Im g_{-+} & \Im g_{--} \\
0 & 0 & 0 & (X^t)^{-1} & 0 & 0\\
0 & \Im g_{++} & \Im g_{+-} & 0 & \Re g_{++} & -\Re g_{+-} \\
0 & -\Im g_{-+} & -\Im g_{--} & 0 & -\Re g_{-+} & \Re g_{--} \end{pmatrix}\end{equation}
in the Darboux basis $\{ \ee, \ff\}$. 
The other splitting morphisms can be obtained by letting $g := 1_{\U(\FF/\FF_0)}$ and $X := 1_{\Re \FF_0}$.
\end{proof}

\begin{remark} \label{rmk:Hess}
When $n_- = 0$, $\Sp_{\FF}(V)$ has been computed in Proposition 3.3 of \cite{Hess}.
\end{remark}

\subsection{Equivalence classes of splittings}

\begin{definition}[Equivalent splittings]
Let two splittings $(\FF, \FF_+, \FF_-)$ and $(\FF, \FF'_+, \FF'_-)$ of a complex Lagrangian subspace $\FF$ be \emph{equivalent} if $\FF_0\oplus \FF_+ = \FF'_0 \oplus \FF'_+$ and $\FF_0 \oplus \FF_- = \FF'_0 \oplus \FF'_-$. Equivalently, two splittings are equivalent if they descend to the same splitting on the quotient spaces $\FF/\FF_0$ and $\FF/\FF'_0$.
Denote an equivalence class of splittings by
\begin{equation} \label{eq:lagrangianswithsplittings}
    \FF^\oplus := [(\FF, \FF_+, \FF_-)]. 
\end{equation}
Let $\FF^\oplus_{\ge 0} := \FF_0 \oplus \FF_+$, and $\FF^\oplus_{\le 0} := \FF_0 \oplus \FF_-$, so that we can express $\FF^\oplus$ as a pair $(\FF^\oplus_{\ge 0}, \FF^\oplus_{\le 0})$.
We will also denote by
\[ \FF^\oplus = [\FF_0 \oplus \FF_+ \oplus \FF_-]\]
when we assume a choice of representative of the equivalence class has been made. Let $\Lag^\C_\oplus(\vec{n})$ denote the set of equivalence classes of complex Lagrangian subspaces of type $\vec{n}$ with splittings. Let $\Lag^\C_\oplus(V)$ be the space of all equivalence classes of complex Lagrangian subspaces of $V$ with splittings. 
\end{definition}

By Theorem \ref{thm:basisC}, $\Re \FF^\oplus_{\ge 0}$ and $\Re \FF^\oplus_{\le 0}$ are symplectic complements of each other, so $H(\Re \FF^\oplus_{\ge 0}) = H(\Re\FF^{\oplus}_{\le 0})$.

Denote the linear symplectic reduction
\[ \FF^\oplus/\FF_0:= [(\FF/\FF_0, \FF^\oplus_{\ge 0} / \FF_0, \FF^\oplus_{\le 0}/\FF_0)] \in \Lag_\oplus^\C(0, n_+, n_-; \Re \FF^\oplus_{\ge 0}/\Re \FF_0).\]

\begin{remark}[Splittings as eigenspaces]
Given an $\omega$-compatible linear complex structure $\JJ$, for every complex Lagrangian subspace $\FF$ there is a linear map $\kappa|_\FF^\sharp: \FF \to \FF$ such that
\[ \kappa|_{\FF}(\uu, \vv) = \omega^\C(\uu, \JJ \overline{\kappa|_\FF^\sharp \vv}) \quad \uu, \vv \in \FF. \]
Equivalently, $\omega^\C(\cdot, \JJ \overline{\cdot})$ restricts to a positive definite inner product on $\FF$, so $\kappa|_\FF$ gets identified with $\kappa|_\FF^\sharp$ under $\FF^\vee \otimes \FF^\vee \xrightarrow{\cong} \FF \otimes \FF^\vee$ where $\FF^\vee$ is the dual space of $\FF$ as a real vector space.

For $\FF \in \Lag^\C(\vec{n})$, $\JJ$ gives a preferred choice of splitting $\FF = \FF_0 \oplus \FF_+ \oplus \FF_-$ by the zero, positive, and negative eigenspaces of $\kappa|_\FF^\sharp$, and a preferred choice of equivalence class of splittings $\iota_\JJ(\FF):=[\FF_0 \oplus \FF_+ \oplus \FF_-]$. So we obtain a section $\iota_\JJ: \Lag^\C(\vec{n}) \hookrightarrow \Lag^\C_{\oplus} (\vec{n})$ of the forgetful fibration $\varphi: \FF^\oplus \mapsto \FF$.
\end{remark}

\begin{remark}[Symplectic Twistor Grassmannians of \cite{LeeLeung}]\label{rem:symplectictwistorGrassmannians}
$\Lag^\C(0, n_+, n_-)$ can be identified with the \emph{symplectic Twistor Grassmannians} $\hat{\mathcal{T}}(2n_+, 2n)$ introduced in Definition 15.ii of \cite{LeeLeung}. If $\FF^\oplus = [0 \oplus \FF_+ \oplus \FF_-] \in \Lag^\C(0, n_+, n_-)$, then $\Re|_{\FF_{\pm}}: \FF_{\pm} \to \Re \FF_{\pm}$ are real vector space isomorphisms. Then $\JJ_{\Re \FF_{\pm}} \vv:=  \Re (\pm i \Re|_{\FF_{\pm}}^{-1} (\vv))$ are $\omega|_{\Re \FF_{\pm}}$-compatible linear complex structures on $\Re \FF_{\pm}$. So $(\Re \FF_{+}, \JJ_{\Re \FF_+}, \JJ_{\Re \FF_-})$ is an element of the symplectic twistor space $\hat{\mathcal{T}}(2n_+, 2n)$.
Conversely, if $(W, \JJ_{W}, \JJ_{W^\omega}) \in \hat{\mathcal{T}}(2n_+, 2n)$, then let $\FF_{\JJ_W}$ be the $+i$-eigenspace of $\JJ_W$ in $W^\C$ and $\FF_{\JJ_{W^\omega}}$ be the $-i$-eigenspace of $\JJ_{W^\omega}$ of $\JJ_{W^\omega}$ in $(W^\omega)^\C$. Then $\FF_{\JJ_{W}} \oplus \FF_{\JJ_{W^\omega}}$ is a splitting of a complex Lagrangian subspace of $(W\oplus W^\omega)^\C = V^\C$.
\end{remark}

\begin{theorem}
    The group extensions
    \begin{equation} \begin{tikzcd}[column sep = small]H(\Re \FF_{\ge 0}^\oplus) \arrow[r, hook] & \Sp_{\FF^{\oplus}}(V) \arrow[two heads]{r}[yshift = -4ex]{{\cdot|_{\Re\FF_0} \times [\cdot]_+ \times [\cdot]_-}} & \GL(\Re\FF_0) \times \U(\FF^\oplus_{\ge 0} / \FF_0) \times \U(\FF^\oplus_{\le 0} / \FF_0) \end{tikzcd} \end{equation}
    \begin{equation} \begin{tikzcd}[column sep = small](\GL(\Re\FF_0) \times \U(\FF^\oplus_{\le 0} / \FF_0)) \ltimes H(\Re\FF_{\ge 0}^\oplus) \arrow[r, hook] & \Sp_{\FF^\oplus}(V) \arrow[r, two heads, "{[\cdot]_+}"'] & \U(\FF^\oplus_{\ge 0} / \FF_0) \end{tikzcd} \end{equation}
    \begin{equation} \begin{tikzcd}[column sep = small] (\U(\FF^\oplus_{\ge 0}/\FF_0) \times \U(\FF^\oplus_{\le 0} / \FF_0)) \ltimes H(\Re \FF_{\ge 0}^\oplus) \arrow[r, hook]& \Sp_{\FF^\oplus}(V) \arrow[r, two heads, "\cdot|_{\Re\FF_0}"'] & \GL(\Re\FF_0) \end{tikzcd} \end{equation}
    are right split.
\end{theorem}

\begin{proof}
We repeat the proof of Theorem \ref{thm:splitF} with $A_{\pm \mp}$, $B_{\pm \mp}$, $C_{\pm \mp}$, $D_{\pm \mp}$, $g_{\pm \mp}$ set to zero.
\end{proof}

\begin{remark}[Dimensions] \label{rmk:dimensions}

If $W$ and $\FF$ are of type $\vec{n}$, by choosing Darboux bases given by Theorems \ref{thm:basisR}, \ref{thm:basisC}, there exist group isomorphisms (\cite{mythesis} Proposition 2.7.4)
\begin{eqnarray*}
\Sp_{\FF^\oplus}(V) &\cong& (\GL(n_0; \R) \times \U(n_+) \times \U(n_-)) \ltimes H(\vec{n})\\
\Sp_{\FF}(V) &\cong& (\GL(n_0; \R) \times \U(n_+, n_-)) \ltimes H(\vec{n})\\
\Sp_W(V) &\cong& (\GL(n_0; \R) \times \Sp(2n_+; \R) \times \Sp(2n_-; \R)) \ltimes H(\vec{n}).
\end{eqnarray*}
So
\begin{eqnarray*}
    \dim_{\R} \Sp_{\FF^\oplus}(V) &=& n^2 + \frac{n_0 (n_0 + 1)}{2} - 2n_+ n_- \\
    \dim_{\R} \Sp_{\FF}(V) &=& n^2 + \frac{n_0 (n_0 + 1)}{2}\\
    \dim_{\R} \Sp_W(V) &=& n^2 + \frac{n_0 (n_0 + 1)}{2} + (n_+ - n_-)^2 + n_+ + n_-.
\end{eqnarray*}
Considering $W$, $\FF$, and $\FF^\oplus$ as basepoints we obtain diffeomorphisms
\begin{eqnarray*}
    \Lag_{\oplus}^\C(\vec{n}; V) &\cong& \Sp(V) / \Sp_{\FF^\oplus}(V).\\
    \Lag^\C(\vec{n}; V) &\cong& \Sp(V) / \Sp_{\FF}(V).\\
    \Gr(\vec{n}; V) &\cong& \Sp(V)/\Sp_W(V).
\end{eqnarray*}
From this we compute
\begin{eqnarray*}
    \dim_\R \Lag^\C_\oplus(\vec{n}; V) &=& n^2 + n - \frac{n_0(n_0+1)}{2} + 2n_+ n_-  \\
    \dim_\R \Lag^\C(\vec{n}; V) &=& n^2 + n - \frac{n_0 (n_0 + 1)}{2}\\
    \dim_\R \Gr(\vec{n}; V) &=& n^2 + n - \frac{n_0 (n_0 + 1)}{2} - (n_+ - n_-)^2 - n_+ - n_-.
\end{eqnarray*}
\end{remark}

\subsection{Cayley transforms and binary octahedral symmetries}
\label{subsec:octahedron}

In this section we will closely examine the case $V = \R^2$ with standard symplectic form $\omega_{\R^2}$. We identify $V^\C$ with $\C^2$, equipped with the standard hermitian inner product $\langle \cdot, \cdot \rangle_{\C^2}$. The linear action of $\SL(2; \C)$ on $\C^2$ factors through $\{\pm 1\}$ on $\mathbb{P}^1(\C)$. Consequently, we have
\[ \begin{tikzcd} \SU(2) \arrow[r, hook] \arrow{d} &\SL(2; \C) \arrow[d, two heads] & \C^2 \setminus \{(0, 0)\}\arrow[symbol= \circlearrowright]{l} \arrow[d, two heads] \\
\SO(3) \arrow[r, hook] & \PSL(2; \C)  & \mathbb{P}^1(\C). \arrow[symbol= \circlearrowright]{l} \end{tikzcd}\]

Now consider the Pauli matrices
\begin{equation}\label{eq:pauli}
    \sigma_1 := \begin{pmatrix} 0 & 1 \\ 1 & 0 \end{pmatrix} \quad 
    \sigma_2 := \begin{pmatrix} 0 & -i \\ i & 0 \end{pmatrix} \quad
    \sigma_3 :=\begin{pmatrix} 1 & 0 \\0 & -1\end{pmatrix}.
\end{equation}
Then let $\ii := -i \sigma_1$,  $\jj := -i \sigma_2$, $\kk := -i \sigma_3$. As elements of the matrix group $\SU(2)$, $\ii$, $\jj$, $\kk$ satisfy the quaternionic relations
\[ \ii \jj = \kk \quad \jj \kk = \ii \quad \kk \ii = \jj, \quad \ii^2 = \jj^2 = \kk^2 = -1, \]
and define a hyperk\"{a}hler structure on $\C^2$, and $\jj$ is an $\omega_{\R^2}$-compatible linear complex structure on $\R^2$. The quaternion group $Q_8:= \{ \pm 1, \pm \ii, \pm \jj, \pm \kk\} \le \SU(2)$ is a nonsplit double cover of the Klein four group $K_4 \cong Q_8 /\{\pm 1\}$. As elements of $\SU(2)$, we can consider the square roots
\begin{eqnarray*}
    \sqrt{\ii} &:=& \frac{1}{\sqrt{2}} (1 + \ii) = \frac{1}{\sqrt{2}} \begin{pmatrix} 1& -i \\ -i & 1 \end{pmatrix}\\
    \sqrt{\jj} &:=& \frac{1}{\sqrt{2}} (1 + \jj) = \frac{1}{\sqrt{2}} \begin{pmatrix}  1 & -1 \\ 1 & 1 \end{pmatrix}\\
    \sqrt{\kk} &:=& \frac{1}{\sqrt{2}} (1 + \kk) = \frac{1}{\sqrt{2}}\begin{pmatrix} 1-i & 0 \\ 0 & 1+i \end{pmatrix}.
\end{eqnarray*}
$\sqrt{\ii}$, $\sqrt{\jj}$, $\sqrt{\kk}$ are elements of $\SU(2)$ and in particular, \emph{the linear fractional transformation associated to $\sqrt{\ii}$ is the Cayley transform}. For our discussion, we observe that the cyclic group of order four $\langle \sqrt{\ii}\rangle/\langle-1_{\SU(2)}\rangle$ permutes the following real lines
\begin{equation} \label{eq:cayley} \begin{tikzcd}
     & \R(\ee - i \ff) \arrow[dl, "\pm \sqrt{\ii}"'] & \\ i\R \ff \arrow[dr, "\pm \sqrt{\ii}"'] & & \R\ee \arrow[ul, "\pm \sqrt{\ii}"']\\
     & \R(\ee + i \ff) \arrow[ur, "\pm \sqrt{\ii}"'] & 
\end{tikzcd}\quad \ee = \begin{pmatrix} 1 \\ 0 \end{pmatrix}, \ff = \begin{pmatrix} 0 \\ 1\end{pmatrix} \end{equation}
and is responsible for relating Theorems \ref{thm:basisR} and \ref{thm:basisC}.

$\sqrt{\ii}$ and $\jj$ generate the binary octahedral group $BO_{48}$ which is a nonsplit double cover of the octahedral group $O_{24}$. So the octahedral and binary octahedral groups occur as ``square roots'' of the quaternionic symmetries in the following manner
\begin{equation} \begin{tikzcd} Q_8 \arrow[r, hook] \arrow[d, two heads] & BO_{48} \arrow[d, two heads] \arrow[r, hook]& \SU(2) \arrow[d, two heads]\\ K_4 \arrow[r, hook] & O_{24} \arrow[r, hook] & \SO(3). \end{tikzcd}\end{equation}
Note that complex conjugation and the antipodal map do not preserve the orientation of $\mathbb{P}^1(\C)$ and do not appear within this diagram.
The elements
\[ \mathbf{O}:= \{ [0:1], [1:0], [i: 1], [-i: 1], [1: 1], [-1:1]\}\]
are vertices of a regular octahedron in $\mathbb{P}^1(\C) \cong S^2$. $\sqrt{\ii}$, $\sqrt{\jj}$, $\sqrt{\kk}$ act on $\mathbb{P}^1(\C)$ by a quarter rotation about, respectively, the axes $[-1:1] - [1:1]$, $[i:1]-[-i:1]$, $[0:1]-[1:0]$. $\ii$, $\jj$, $\kk$ act by a half rotation about the respective axes, and the element
\[ \sqrt{\ii} \sqrt{\jj} = \frac{1}{2} (1 + \ii + \jj + \kk) = \frac{1}{2}\begin{pmatrix} 1-i & -1 - i \\ 1-i & 1+i \end{pmatrix}\]
is one of the elements that permute the three axes.

With respect to the standard Hermitian product on $\C^2$, we can identify the Pauli matrices with hermitian forms on $\C^2$ by 
\[ \langle \sigma_j \uu, \vv \rangle_{\C^2} = \overline{\vv}^t \sigma_j \uu \quad \uu, \vv \in \C^2, j = 1, 2, 3.\]
Let $\U(\sigma_j)$ be the unitary group associated to the hermitian form defined by $\sigma_j$, for $j = 1, 2, 3$. Since $\sqrt{\ii}$, $\sqrt{\jj}$, $\sqrt{\kk} \in \SU(2)$, their conjugate transpose is equal to their inverse. Moreover,
\[ \sqrt{\ii} \cdot \sigma_2 \cdot \sqrt{\ii}^{-1} = \sigma_3, \quad \sqrt{\jj}\cdot \sigma_3 \cdot \sqrt{\jj}^{-1} = \sigma_1, \quad \sqrt{\kk} \cdot \sigma_1\cdot \sqrt{\kk}^{-1} = \sigma_2,\]
so the following three real forms of $\SL(2; \C)$ are mapped to each other by
\begin{equation} \label{eq:realforms} \begin{tikzcd}[column sep= -1cm] & \SL(2; \C) \cap \U(\sigma_1) \arrow[dl, outer sep = -1pt, "\Ad_{\sqrt{\kk}}"'] &\\
\SL(2; \C) \cap \U(\sigma_2)\arrow[rr, "\Ad_{\sqrt{\ii}}"'] && \SL(2; \C) \cap \U(\sigma_3) \arrow[ul, outer sep = -4pt, "\Ad_{\sqrt{\jj}}"'].\end{tikzcd}\end{equation}
Here $\SL(2; \C) \cap \U(\sigma_2) = \SL(2; \R)$ and $\SL(2; \C) \cap \U(\sigma_3) = \SU(1, 1)$.

We can also view $\ii$, $\jj$, $\kk$ as basis elements of $\su(2)$. The real Lie algebras corresponding to the real Lie groups in Equation (\ref{eq:realforms}) are mapped to each other by
\[ \begin{tikzcd}[column sep = -0.5cm] & \R \ii \oplus i \R\jj \oplus i \R \kk \arrow[dl, outer sep = -1pt, "\Ad_{\sqrt{\kk}}"'] &\\
 i \R \ii \oplus \R \jj \oplus  i\R\kk \arrow[rr, "\Ad_{\sqrt{\ii}}"'] && i \R \ii \oplus i\R \jj \oplus  \R\kk \arrow[ul, outer sep = -4pt, "\Ad_{\sqrt{\jj}}"']\end{tikzcd}.\]

An $\sl_2$ triple can be constructed by
\begin{equation} \label{eq:sl2} \left( \frac{i\ii - \jj}{2}, i\kk, \frac{i\ii + \jj}{2}\right) = \left( \begin{pmatrix} 0 & 1 \\ 0 & 0 \end{pmatrix}, \begin{pmatrix} 1 & 0 \\ 0 & -1 \end{pmatrix}, \begin{pmatrix} 0 & 0 \\ 1 & 0 \end{pmatrix} \right), \end{equation}
Iterating the 4-periodic $\Ad_{\sqrt{\ii}}$ on this $\sl_2$ triple, we obtain mappings between the root space decompositions of $\sl(2; \R)$ and $\su(1, 1)$:
\[ \begin{tikzcd}[column sep = -2cm]
     & \R \begin{psmallmatrix} i & 1 \\ 1 & -i \end{psmallmatrix}\oplus \R \begin{psmallmatrix} 0 & i \\ -i & 0 \end{psmallmatrix}\oplus \R \begin{psmallmatrix} -i & 1 \\ 1 & i \end{psmallmatrix} \arrow[dl, "\Ad_{\sqrt{\ii}}"'] & \\ \R \begin{psmallmatrix} 0 & 0 \\ 1 & 0 \end{psmallmatrix}\oplus \R \begin{psmallmatrix} -1 & 0 \\ 0 & 1 \end{psmallmatrix}\oplus \R \begin{psmallmatrix} 0 & 1 \\ 0 & 0 \end{psmallmatrix} \arrow[dr, "\Ad_{\sqrt{\ii}}"'] & & \R \begin{psmallmatrix} 0 & 1 \\ 0 & 0 \end{psmallmatrix}\oplus \R \begin{psmallmatrix} 1 & 0 \\ 0 & -1 \end{psmallmatrix}\oplus \R \begin{psmallmatrix} 0 & 0 \\ 1 & 0 \end{psmallmatrix}\arrow[ul, "\Ad_{\sqrt{\ii}}"']\\
     & \R \begin{psmallmatrix} -i & 1 \\ 1 & i \end{psmallmatrix}\oplus \R \begin{psmallmatrix} 0 & -i \\ i & 0 \end{psmallmatrix}\oplus \R \begin{psmallmatrix} i & 1 \\ 1 & -i \end{psmallmatrix}. \arrow[ur, "\Ad_{\sqrt{\ii}}"'] & 
\end{tikzcd}\]
Here $\Ad_{\sqrt{\ii}}^2$ is the unique nontrivial Weyl group symmetry associated to the Cartan subalgebras $\C\jj$ and $\C\kk$ of $\sl(2; \C)$.

The $\R^{2n}$ case is obtained by taking $n$ direct products of these structures. We refer the reader to Section \ref{subsubsec:rootspacedecompositions} for how these structures appear from a traditional Lie theoretic perspective.

\subsection{Extension of Weyl group symmetries}
\label{subsec:basepoints}

Now fix a Darboux basis $\{\ee, \ff\}$ of $V$ and identify $V \cong (\R^2)^n$ with $\{\ee, \ff\}$. Let $H$ be a discrete subgroup of $\SU(2)$, $\mathbb{S}_n$ be the permutation group of $\{1, \cdots, n\}$ and choose $\ast:= (\ee_1, \cdots, \ee_n) \in (\R^2)^n$ and denote $L_{\ast}:= \span_\R\{\ee_1, \cdots, \ee_n\}$. Then this choice picks out a discrete subset
\[ F_{H, \ast } := \{ \span_\C\{\ww_1, \cdots, \ww_n\} : \ww_j \in H\R\ee_j, 1 \le j \le n\} \]
of $\Lag^\C(V)$. $\mathbb{S}_n \ltimes H^n$ acts on $F_{H, \ast}$ by
\[ (\sigma, (h_1, \cdots, h_n)).\span_\C\{\ww_1, \cdots, \ww_n\}:= \span_\C\{h_{\sigma(1)}.\ww_{\sigma(1)}, \cdots, h_{\sigma(n)}.\ww_{\sigma(n)}\}\]
and the subgroup $H^n$ acts transitively on $F_{H, \ast}$.
If $H = \langle \sqrt{\ii}\rangle$, $H^n$ acts on $F_{\langle \sqrt{\ii}\rangle, \ast}$ with the subgroup $\langle -1_{\SU(2)}\rangle^n$ acting trivially. From (\ref{eq:cayley}) and Theorem \ref{thm:basisC}, $F_{\langle \sqrt{\ii}\rangle, \ast } \cap \Lag^\C(\vec{n}; V)$ is nonempty for every $\vec{n}$. So the decomposition into orbits can be expressed in a notation suggestive of the Bruhat decomposition

\[ \Lag^\C((\R^2)^n) =\bigcup_{\vec{n}} \Lag^\C(\vec{n}) = \bigcup_{w \in \mathbb{S}_n \ltimes (\langle \sqrt{\ii}\rangle/\langle -1\rangle)^n} Gw\Sp_{L_\ast}((\R^2)^n).L_{\ast}^\C.\]

\begin{remark}[Partial Cayley transforms and basepoints]
Choosing a root space decomposition of $\g$ from a vectorial Cartan subalgebra is equivalent to choosing a Darboux basis up to choice of $\ee$s and $\ff$s within the splitting $V \cong \R^{\oplus 2n}$ (Section \ref{subsec:cartan}).

In the case of interest, the partial Cayley transforms
$c_\Gamma c_\Sigma^2$ of \cite{Wolf},\cite{Wolf1} are equal (cf. Equation (\ref{eq:cayleymatrix1}), (\ref{eq:cayleymatrix2})) to 
$h_{\Gamma, \Sigma} = (h_1, \cdots, h_n)$ where
\[ h_j = \begin{cases} \sqrt{\ii} & \text{ if } 2i \varepsilon_j \in \Gamma \\
\ii & \text{ if } 2i\varepsilon_j \in \Sigma.\end{cases}\]

When $\Gamma$, $\Sigma$ are given by (\ref{eq:Gamma}) (\ref{eq:Sigma}), the complex Lagrangian subspace
\[ \FF_{\vec{n}}:= h_{\Gamma, \Sigma}.\span_\C\{ \ee_1 + i \ff_1, \cdots, \ee_n + i \ff_n\}\] is equal to (in the notation of (\ref{eq:bracketspan}))
\[ \left[\begin{pmatrix} 1_{n_0} & 0 & 0 \\ 0 & 1_{n_+} & 0 \\ 0 & 0 & 1_{n_-}\\ 0 & 0 & 0 \\ 0 & -i 1_{n_+} & 0 \\ 0 & 0 & i 1_{n_-}\end{pmatrix}\right]\] and
is contained in $F_{\langle \sqrt{\ii}\rangle, \ast} \cap \Lag^\C(\vec{n})$.
\end{remark}

Conjugation by elements of $\mathbb{S}_n \ltimes \langle \ii \rangle^n$ stabilizes the maximal torus $\SO(2)^n$ of the maximal compact subgroup of elements commuting with the compatible linear complex structure $\diag(\jj, \cdots, \jj)$ in $\Sp((\R^2)^n)$, so the group $\mathbb{S}_n \ltimes (\langle \sqrt{\ii}\rangle/\langle -1\rangle)^n$ can be viewed as an extension of the Weyl group $\mathbb{S}_n \ltimes (\langle \ii \rangle/\langle -1\rangle)^n$ of $\Sp((\R^2)^n)$. Taking $H = BO_{48} = \langle \sqrt{\ii}, \jj\rangle$, we observe that this discrete group can be extended further to involve quaternionic symmetries and ``square roots of quaternionic symmetries.''

\[ \Lag^\C(V) = \bigcup_{w \in \mathbb{S}_n \ltimes (O_{24})^n} Gw.L_{\ast}^\C.\]
\section{Orbit Fibrations}
\label{sec:fibrations}

In this section we review the Borel embedding and describe three simultaneous linear symplectic reductions as fiber bundles. The fiber bundles become isomorphic when the total space is restricted to spaces that behave well with respect to a compatible linear complex structure. We show how, in the reduced space, the relevant stabilizer subgroups arise from  a choice of involution commuting with a choice of compatible linear complex structure in the reduced linear spaces. We also discuss symmetric space structures and complex structures on the fibers.

\subsection{Review of the Borel embedding}
\label{subsec:borel}

In order to discuss particular complex structures we will be using and set up some notation, we will briefly recall the Borel embedding. For matrix forms of the objects appearing in this section, we refer to Section \ref{subsec:harishchandra}.

Denote $G:= \Sp(V)$ and $G^\C:= \Sp(V^\C)$. Let $\JJ$ be an $\omega$-compatible linear complex structure on $V$. Denote again by $\JJ$ its complex linear extension to $V^\C$.

Since we assume the linear symplectic group comes with its fundamental representation, we assume $G \subseteq \End(V)$, and we can view $\JJ$ as an element of $G =\Sp(V)$. Then $\Ad(\JJ)$ is a Cartan involution of $G$. In fact, the following holds:

\begin{proposition} There is a bijection between the Cartan involutions on $\g$ and the choice of $\omega$-compatible almost complex structures $\JJ$ on $V$.
\end{proposition}

\begin{proof}
Take a Darboux basis $\{ \ee, \JJ \ee\}$ of $V$. With this basis, we identify Lie algebra of $G$ as
\[ \g = \left\{ \mathbf{X} = \begin{pmatrix} \aa & \bb \\ \cc & -\aa^t \end{pmatrix} : \bb, \cc \in \Mat_{n\times n}(\R)^t, \aa \in \Mat_{n\times n}(\R) \right\}.\]
We can compute
\begin{equation*} \Ad(\JJ) \mathbf{X} := \begin{pmatrix} 0 & -1_n\\ 1_n& 0\end{pmatrix} \begin{pmatrix} \aa & \bb \\ \cc & -\aa^t\end{pmatrix} \begin{pmatrix} 0 & 1_n \\ -1_n & 0 \end{pmatrix} = \begin{pmatrix} -\aa^t & -\cc \\ -\bb & \aa\end{pmatrix} = -\mathbf{X}^t \end{equation*}
for all $\mathbf{X} \in \g$. So $\Ad(\JJ)$ is equal to the Cartan involution $\theta(\mathbf{X}) := -\mathbf{X}^t$ on $\g$. It is well known that all Cartan involutions are conjugate, and all $\omega$-compatible linear complex structures are conjugate. The conjugates correspond to each other because
\begin{eqnarray*} \Ad(g\JJ g^{-1})(\mathbf{X}) &=& (g \JJ g^{-1})\mathbf{X} (g\JJ g^{-1})\\
&=& -g (g^{-1} \mathbf{X}g)^{t} g^{-1} = (\Ad(g) \circ \theta \circ \Ad(g^{-1}))(\mathbf{X}). \end{eqnarray*}
$g$ commutes with $\JJ$ if and only if $\Ad(g)$ commutes with $\theta$, so both spaces can be parametrized by $G/K$.
\end{proof}

Let $K^\C \le G^\C$, and $K \le G$ be the subgroups of symplectic linear transformations that commute with $\JJ$, and $G_u \le G^\C$ be the subgroup of symplectic linear transformations $g$ such that $g\JJ = \JJ \overline{g}$. Equivalently, $K^\C$, $K$, and $G_u$ consists of the symplectic linear transformations that preserve, respectively, the bilinear forms $\omega^\C(\cdot, \JJ \cdot)$, $\omega(\cdot, \JJ \cdot)$, and $\omega^\C(\cdot, \JJ \overline{\cdot})$. Let $\g$, $\g^\C$, $\g_u$, $\k^\C$, $\k$ denote the corresponding Lie algebras. The decomposition into the $\pm 1$-eigenspaces of $\Ad(\JJ)$, $\g = \k \oplus \p$ is a Cartan decomposition of $\g$.

For all nonzero $\vv, \ww \in V$
\[ \omega^\C( (1_{V^\C} \mp i \JJ) \vv, (1_{V^\C} \mp i \JJ)\ww) = 0 \]
and
\[ \pm \kappa( (1_{V^\C} \mp i \JJ) \vv, (1_{V^\C} \mp i \JJ)\vv) > 0 \]
so $(1_{V^\C} - i\JJ)V \in \Lag^\C(0, n, 0)$. We can also immediately see 
\begin{equation}
    \FF_{\pm \JJ}:= (1_{V^\C} \mp i \JJ)V
\end{equation} is the $\pm i$-eigenspace of $\JJ$.
Denote the projections to $\FF_{\pm \JJ}$ along $\FF_{\mp \JJ}$ by
\begin{equation}
    \pr_{\pm\JJ} := \frac{1}{\sqrt{2}} (1_{V^\C} \mp i \JJ)
\end{equation}

Let $P^\mp \le G^\C$ be the abelian subgroup of symplectic linear transformations (cf. Section \ref{subsec:harishchandra}) that fix all the vectors of $\FF_{\mp\JJ}$. The following is a well known fact. 
\begin{lemma}\label{lem:borelstabilizers}
$K^\C P^-$ is the stabilizer of $\FF_{-\JJ}$ in $G^\C$. Moreover,
\[ K = G\cap K^\C P^- = G_u \cap K^\C P^-.\]
\end{lemma}

\begin{proof}
$P^-$ is by definition, contained in the stabilizer of $\FF_{-\JJ}$. Since $\FF_{\JJ}$ is transverse to $V$, 
\[ \FF_{-\JJ} = \pr_{-\JJ}V = \pr_{-\JJ}V^\C. \]
If $k \in K^\C$, then
\[ k. \FF_{-\JJ} = k (1_{V^\C} +i \JJ)V^\C = (1_{V^\C} + i\JJ)kV^\C = \FF_{-\JJ}.\]
So $K^\C P^-$ is contained in the stabilizer of $\FF_{-\JJ}$.

If $g \in \Sp(V^\C)$ is in the stabilizer of $\FF_{-\JJ} \in \Lag^\C(V)$, then with respect to the splitting $V^\C \cong \FF_{-\JJ} \oplus \FF_{\JJ}$, we can write $g$ as
\[ g = \begin{pmatrix} g_{11} & g_{12} \\0 & g_{22} \end{pmatrix} = \begin{pmatrix} g_{11} & 0 \\ 0 & g_{22} \end{pmatrix} \begin{pmatrix} 1_{\FF_{-\JJ}} & g_{11}^{-1} g_{12} \\ 0 & 1_{\FF_{\JJ}} \end{pmatrix}\]
such that $g_{11}: \FF_{-\JJ} \to \FF_{-\JJ}$, $g_{12}: \FF_{\JJ} \to \FF_{-\JJ}$, and $g_{22}: \FF_{\JJ} \to \FF_{\JJ}$ are linear maps.

Then
\[ \begin{pmatrix} 1_{-\FF_{\JJ}} & g_{11}^{-1} g_{12} \\ 0 & 1_{\FF_{\JJ}} \end{pmatrix} \in P^-.\]

Moreover, for any $\vv \in V^\C$ there exist $\vv_{-\JJ} \in \FF_{-\JJ}$ and $\vv_{\JJ} \in \FF_{\JJ}$ such that  $\vv = \vv_{-\JJ} + \vv_{\JJ}$. Then
\[ \JJ (g_{11} \vv_{-\JJ} + g_{22} \vv_{\JJ}) = -i g_{11} \vv_{-\JJ} +i g_{22} \vv_{\JJ} = g_{11} \JJ \vv_{-\JJ} + g_{22} \JJ \vv_{\JJ}  \]
so 
\[ \JJ \begin{pmatrix} g_{11} & 0 \\ 0 & g_{22} \end{pmatrix}=   \begin{pmatrix} g_{11} & 0 \\ 0 & g_{22} \end{pmatrix} \JJ \implies \begin{pmatrix} g_{11} & 0 \\ 0 & g_{22} \end{pmatrix} \in K^\C.\]
Thus the stabilizer of $\FF_{-\JJ}$ is equal to $K^\C P^-$.

If $g \in G \cap K^\C P^-$, $g = \overline{g}$. Since $\FF_{\JJ} = \overline{\FF_{-\JJ}}$, $g = \overline{g}$ implies $g_{12} = \overline{g_{21}} = 0$, so that $G \cap P^-$ is trivial. So $G \cap K^\C P^- = G \cap K^\C = K$.

Since $\FF_{-\JJ}$ and $\FF_{\JJ}$ are orthogonal with respect to $\omega^\C(\cdot, \JJ \overline{\cdot})$, $G_u \cap P^-$ is trivial. Then $G_u \cap K^\C P^- = G_u \cap K^\C$ consists of those $g \in G^\C$ such that $\JJ g = g\JJ = \JJ \overline{g}$, so $G_u \cap K^\C = G \cap K^\C = K$.
\end{proof}

If we take a basis $\{ \vv_1, \cdots, \vv_n\}$ of a complex Lagrangian subspace $\FF$ that is unitary with respect to $\omega^\C(\cdot, \JJ \overline{\cdot})|_{\FF \times \FF}$, then $\{ \vv_1, \cdots, \vv_n, \JJ \overline{\vv}_1, \cdots, \JJ \overline{\vv}_n\}$ is a complex Darboux basis of $V^\C$ that is unitary with respect to $\omega^\C(\cdot, \JJ \overline{\cdot})$. Thus $G_u$ acts transitively on $\Lag^\C(V)$.

So as a consequence of Lemma \ref{lem:borelstabilizers}, the maps
$gK \mapsto g.\FF_{-\JJ}$, and $gK^\C P^- \mapsto g.\FF_{-\JJ}$ provide diffeomorphisms
\begin{eqnarray}
G/K &\xrightarrow{\cong}& \Lag^\C(0, 0, n; V) \notag\\
G_u/K &\xrightarrow{\cong}& \Lag^\C(V) \label{eq:identifcationsborel}\\
G^\C/K^\C P^- &\xrightarrow{\cong}& \Lag^\C(V). \notag
\end{eqnarray}

So in our case of discussion, the Borel embedding (Proposition 7 of \cite{Borel}, and see Proposition I.5.24 of \cite{BorelJi}, Proposition 7.14 of \cite{Helgason} for expositions)
\begin{equation} \label{eq:borel}
G/K \hookrightarrow G^\C/K^\C P^- \cong G_u /K
\end{equation}
can be expressed as just the $G$-equivariant inclusion
\begin{equation} \label{eq:borelinclusion}
    \Lag^\C(0, 0, n; V) \hookrightarrow \Lag^\C(V)
\end{equation}

\begin{remark}[Dependence on choice of $\JJ$]
The Borel embedding (\ref{eq:borel}) exists once $\JJ$ is chosen and depends equivariantly on the choice of $\JJ$. This is because all Cartan involutions are conjugate, and a choice of maximal compact subgroup $K \le G$ is equivalent to a choice of $\omega$-compatible linear complex structure $\JJ$. Moreover, since $K^\C P^-$ is a parabolic subgroup of $G^\C$, it is its own normalizer. Conjugating by elements of $G$ (and $G_u$) preserves all the structures in the embedding--i.e. all the statements above hold if we replace $\JJ$'s by $g\JJ g^{-1}$'s. However, we note that after the identifications (\ref{eq:identifcationsborel}),(\ref{eq:borelinclusion}) \emph{does not require making a choice of $\JJ$ at all} and only uses complex conjugation on $V^\C$.
\end{remark}

\begin{remark}[Complex structure on $\Lag^\C(V)$] \label{rem:complexstructures}
If we identify $V^\C \cong \C^{2n}$, $G^\C$ is a linear algebraic group. Since $K^\C P^-$ is a parabolic subgroup of $G^\C$, $\Lag^\C(V) \cong G_u /K \cong G^\C/K^\C P^-$ is a (complex) projective variety, and inherits a K\"{a}hler structure from complex projective space. In this paper, to make transparent the dependence of the choice of $
\omega$-compatible linear complex structure $\JJ$, we use the following construction of the K\"{a}hler structure on $\Lag^\C(V)$ from \cite{Helgason}:

Given an $\omega$-compatible linear complex structure $\JJ$ on $V$, let $G_u$ be the subgroup of $G^\C$ consisting of the elements fixed by $g \mapsto (g^\dagger)^{-1}$ where $(\cdot)^\dagger$ is the adjoint with respect to the hermitian inner product $\omega^\C( \cdot \JJ \overline{\cdot})$. Let $K$ be fixed locus of the involution $g \mapsto \overline{g}$ on $G_u$. The image $\Ad(K) \subseteq \GL(\g_u)$ is compact, so $(G_u, K)$ is a Riemannian symmetric pair. Let $o:= K$ be a basepoint of the left coset space $G_u/K$. A choice of $K$-invariant bilinear form on $T_o (G_u/K) \cong \g_u / \k \cong i\p $, extends to a $G_u$-invariant Riemannian metric on $G_u/K$, making it a globally Riemannian symmetric space (Proposition IV.3.4 of \cite{Helgason}).

We can choose the $K$-invariant bilinear form on $T_o (G_u/K)$ to be the restriction of the Killing form to $i\p$.

Moreover,
\[ \omega(\uu, \JJ \vv) + \omega(\JJ \uu, \vv) = 0 \quad \uu, \vv \in V,\]
so $\JJ$ can be viewed as an element of $\k \subseteq \End(V)$. Then $\Ad(\JJ)\mathbf{X} = -\mathbf{X}$ is equivalent to
\begin{equation} \label{eq:adJ}
    \frac{1}{2} \ad(\JJ)\mathbf{X} = \JJ\mathbf{X} \quad \mathbf{X}\in \p.
\end{equation}
So $\frac{1}{2}\ad(\JJ)$ is a linear complex structure on $i\p$, compatible with the restriction of the Killing form of $\g_u$ to $i\p$. By Proposition VIII.4.2 in \cite{Helgason}, $\frac{1}{2}\ad(\JJ)$ extends to a unique $G_u$-invariant genuine complex structure on $G_u/K$, making $\Lag^\C(V)$ into a hermitian symmetric space of the compact type CI in the notation of \cite{Helgason}.
\end{remark}

\subsection{Simultaneous linear symplectic reduction}

\begin{definition}[Normalizers and centralizers of fibers]\label{def:normalizercentralizer}
Suppose $E \xrightarrow{\pi} B$ is a smooth surjective map, and a real Lie group $G$ acts on $E$ smoothly. Let $N(\pi^{-1}(b)):= \{ g \in G : g.\pi^{-1}(b) = \pi^{-1}(b)\}$ and $C(\pi^{-1}(b)):= \{ g \in G : g|_{\pi^{-1}(b)} = 1_{\pi^{-1}(b)}\}$. We will denote by $N(b) := N(\pi^{-1}(b))$ and $C(b):= C(\pi^{-1}(b))$ when $\pi$ is clear from context. If $\pi(x) = b$, then denote the isotropy subgroup of $x$ in $N(\pi^{-1}(b))$ by $N(\pi^{-1}(b))_x$ (or simply by $N(b)_x$).
\end{definition}

\begin{remark}
$C(b)$ is the kernel of the action of $N(b)$ on $\pi^{-1}(b)$, and hence a normal subgroup of $N(b)$. $C(b)$ is also a subgroup of $N(b)_x$, so it is also a normal subgroup of $N(b)_x$ for every $x \in \pi^{-1}(b)$.
\end{remark}

The following statement may be well known.

\begin{lemma} \label{lem:orbitspacefibration}
Suppose $M$ is a smooth manifold, and $G$ is a real Lie group acting smoothly, freely, and properly on $M$ from the right. Suppose $H \le G$ is a closed subgroup. Then the orbit space $M/H$ is a smooth manifold and there is a smooth fibration of orbit spaces
\[ \pi: M/H \to M/G\quad xH \mapsto xG \quad x \in M\]
with diffeomorphism of fibers \[ G/H \xrightarrow{\cong} \pi^{-1}(xG) \quad gH \mapsto xgH.\]
\end{lemma}

\begin{proof}
$H$ acts on $M$ smoothly and freely. Since $H$ is closed, the inclusion map to $G$ is proper, and $H$ acts on $M$ properly. So $M/H$ is a smooth manifold such that the quotient map $M \twoheadrightarrow M/H$ is a smooth surjective submersion. $M \twoheadrightarrow M/G$ is constant on the fibers of $M \twoheadrightarrow M/H$, so we can pass smoothly to the quotient to obtain $\pi: M/H \twoheadrightarrow M/G$.
\end{proof}

The following statement may be well known.

\begin{lemma}\label{lem:fiber}
We assume the setup of Definition \ref{def:normalizercentralizer}.
\begin{enumerate}
    \item If $b\in B$ is a regular value of $\pi$, $N(b)$ acts transitively on $\pi^{-1}(b)$, and suppose $N \le C(b)$ is a normal subgroup of both $N(b)$ and $N(b)_x$. Then there is an $N(b) \twoheadrightarrow N(b)/N$-equivariant diffeomorphism
\[ \pi^{-1}(b) \cong N(b) / N(b)_x \cong (N(b)/N) / (N(b)_x / N).\]
    \item If $G$ acts on $B$ and $\pi$ is $G$-equivariant, then $g \in G$ is an element of $N(b)$ if and only if there exist $x, x' \in \pi^{-1}(b)$ such that $g.x = x'$. Moreover, $N(b)$ is equal to $G_b$, the isotropy subgroup of $b$ in $G$, and $N(b)_x$ is equal to $G_x$.
    \item If $G$ acts on $B$, $\pi$ is $G$-equivariant, and $G$ acts on $E$ transitively, then
    \begin{enumerate}
        \item $G$ acts transitively on $B$, and $N(b)$ acts transitively on $\pi^{-1}(b)$ for every $b \in B$.
        \item For every $b, b' \in B$, $N(b)$ is conjugate to $N(b')$. If $\pi(x) = b$ and $\pi(x') = b'$, then $N(b)_x$ is conjugate to $N(b')_{x'}$.
        \item $\pi$ is a fiber bundle. Moreover, for every $(x, b)$ such that $\pi(x) = b$, there are $G$-equivariant diffeomorphisms such that the following diagram commutes:
        \[ \begin{tikzcd} E  \arrow[r, "\cong"] \arrow[d, two heads, "\pi"] & G/G_x  \arrow[d, two heads, "\pi'"] \arrow[r, "\cong"] & G \times_{N(b)} (N(b) / N(b)_x) \arrow[d, two heads, "\pi''"] \\ B \arrow[r, "\cong"] & G/G_b \arrow[r, equals] & G/N(b) \end{tikzcd}.\]
    \end{enumerate}
    Here $\pi'$ is as in Lemma \ref{lem:orbitspacefibration} and $\pi''$ is an associated bundle to the principal $N(b)$-bundle $G \twoheadrightarrow G/N(b)$. $G/G_x$ and $G/G_b$ are viewed as both left coset spaces and orbit spaces.
\end{enumerate}
\end{lemma}

\begin{proof} (1): $\pi^{-1}(b)$ is closed submanifold of $E$, so $N(b)$ is closed. $N(b)$ is also a subgroup of $G$, and hence a Lie subgroup (Theorem II.2.3 of \cite{Helgason}). By assumption $\pi^{-1}(b)$ is a homogeneous space. The diffeomorphism and equivariance are well known (see for instance Theorem II.3.2, Proposition II.4.3 of \cite{Helgason}). (2): If there exist $x, x' \in \pi^{-1}(b)$ such that $g.x = x'$, $g.b  = g.\pi(x) = \pi(g.x) = \pi(x') = b$, so $g \in G_b$. If $g \in G_b$, for any $y \in \pi^{-1}(b)$, $\pi(g.y) = g.\pi(y) = g.b = b$. So $g \in N(b)$. If $g \in N(b)$, for any $x \in \pi^{-1}(b)$, $g.x \in \pi^{-1}(b)$. $N(b)_x \le G_x$ by definition. $G_x \le N(b)_x$ by what we have just shown. (3a): If $b, b' \in B$, take any $x \in \pi^{-1}(b)$ and $x' \in \pi^{-1}(b)$. By assumption, there exists $g \in G$ such that $g.x = x'$. Then since $\pi$ is $G$-equivariant, $g.b = g.\pi(x) = \pi(g.x) = \pi(x') = b'$. For any $x, x' \in \pi^{-1}(b)$, there exists $g \in G$ such that $g.x = x'$. By (2), $g \in N(b)$. (3b): From (2) $G_x = N(b)_x$, $G_x' = N(b')_{x'}$, so it follows from transitivity that the isotropy subgroups are all conjugate. Similarly, from (2) $G_b = N(b)$, $G_{b'} = N(b')$, and they are conjugate as isotropy subgroups. (3c): If $x' \in E$, by the transitivity assumption, there exists an $h \in G$ such that $hx' = x$. Then $x' \mapsto hG_x$ gives a map $E \to G/G_x$ not depending on the choice of $h$, that is well known to be a diffeomorphism. If $h' \in G$ such that $h'.\pi(x') = b$, then $h'.x' \in \pi^{-1}(b)$, so there exists a $g \in N(b) = G_b$ such that $gh' = h$ by (3a). So $h'G_b = hg^{-1} G_b = h G_b$, and the left square commutes. The right square commutes by Proposition I.5.5 of \cite{KobayashiNomizu}.
\end{proof}

\begin{theorem}[Simultaneous linear symplectic reduction] \label{thm:reduction}~\\
Consider the five $\Sp(V)$-equivariant maps given in the commuting diagram
\begin{equation}\label{eq:reduction}
    \begin{tikzcd} \Lag_\oplus^\C(\vec{n}) \arrow[r, two heads, "\varphi"] \arrow[d, two heads, "\Re\cdot_{\ge 0}"'] \arrow[dr, two heads, "\alpha"] &\Lag^\C(\vec{n}) \arrow[d, two heads, "\Re \cdot_0"]\\ \Gr(\vec{n}) \arrow[r, two heads, "\cdot \cap \cdot^\omega"'] & \Gr(n_0, 0, n-n_0) \end{tikzcd} \quad \begin{tikzcd} \FF^\oplus \arrow[r, maps to] \arrow[d, maps to] \arrow[dr, maps to] & \FF \arrow[d, maps to]\\ \Re \FF^\oplus_{\ge 0} \arrow[r, maps to] & \Re \FF_0 \end{tikzcd}.
\end{equation}
\begin{enumerate}
\item $\alpha$, $\varphi$, $(\cdot \cap \cdot^\omega)$, $\Re \cdot_{\ge 0}$, $\Re \cdot_0$ are fiber bundles.
\item If $\pi = \alpha, \Re \cdot_0, (\cdot \cap \cdot^
\omega)$ and $ b = W_0 \in \Gr(n_0, 0, n-n_0)$, or if $\pi = \Re \cdot_{\ge 0}$ and $b = W \in \Gr(\vec{n})$, or if $\pi = \varphi$ and $b = \FF \in \Lag^\C(\vec{n})$, then the normalizers of the fibers are $N(\pi^{-1}(b)) \cong \Sp_b(V)$.

\item If $\pi = \varphi, \alpha, \Re \cdot_{\ge 0}$ and $x = \FF^\oplus \in \Lag^\C_\oplus(\vec{n})$, or if $\pi = \Re \cdot_0$ and $x = \FF \in \Lag^\C(\vec{n})$, or if $\pi = \cdot \cap \cdot^\omega$ and $x = W' \in \Gr(\vec{n})$, then the isotropy subgroups of the action of $N(\pi^{-1}(\pi(x)))$ on $\pi^{-1}(\pi(x))$ are $N(\pi^{-1}(\pi(x)))_x \cong \Sp_x(V)$.

\item For all $W_0 \in \Gr(n_0, 0, n-n_0)$ there are diffeomorphisms for the fibers
\begin{eqnarray*}
    \alpha^{-1}(W_0) & \xrightarrow{\cong} & \Lag_{\oplus}^\C(0, n_+, n_-; W_0^\omega/W_0) \\
    \FF^\oplus &\mapsto &  \FF^\oplus/\FF_0 \\
    (\Re \cdot_0)^{-1}(W_0) & \xrightarrow{\cong} & \Lag^\C(0, n_+, n_-; W_0^\omega/W_0)\\
    \FF &\mapsto &\FF/\FF_0\\
    (\cdot \cap \cdot^{\omega})^{-1}(W_0) &\xrightarrow{\cong}& \Gr(0, n_+, n_-; W_0^\omega/W_0)\\
    W' &\mapsto& W'/W_0
\end{eqnarray*}
equivariant with respect to the isomorphism
\begin{equation} \label{eq:fiberaction} \Sp_{W_0}(V)/(\GL(W_0) \ltimes H(W_0)) \xrightarrow{\cong} \Sp(W_0^\omega/W_0).\end{equation}
Similarly, for all $W \in \Gr(\vec{n})$ there is a diffeomorphism
\begin{eqnarray*}
(\Re\cdot_{\ge 0 })^{-1}(W) &\xrightarrow{\cong} & \Lag^\C(0, n_+, 0; W/W_0) \times \Lag^\C(0, 0, n_+; W^\omega/W_0)\\
\FF^\oplus &\mapsto& (\FF^\oplus_{\ge 0}/\FF_0, \FF^\oplus_{\le 0}/\FF_0)
\end{eqnarray*}
equivariant with respect to the isomorphism $\Sp_W(V)/(\GL(W_0) \ltimes H(W)) \xrightarrow{\cong} \Sp(W/W_0) \times \Sp(W^\omega/W_0)$.

\item If $\pi = \alpha, \Re \cdot_0, (\cdot \cap \cdot^
\omega)$ and $ b = W_0 \in \Gr(n_0, 0, n-n_0)$, or if $\pi = \Re \cdot_{\ge 0}$ and $b = W \in \Gr(\vec{n})$, or if $\pi = \varphi$ and $b = \FF \in \Lag^\C(\vec{n})$, then there are short exact sequences of groups:
\begin{equation} \label{eq:centralizer1}
    \GL(\Re b_0) \ltimes H(\Re b) \hookrightarrow C(\pi^{-1}(b)) \twoheadrightarrow Z(\Sp((\Re b_0)^\omega/\Re b_0))
\end{equation}
\begin{equation}\label{eq:centralizer2}
    Z(\Sp(V)) \hookrightarrow C(\pi^{-1}(b)) \twoheadrightarrow (\operatorname{GL}(\Re b_0)/\{ \pm 1_{\GL(\Re b_0)}\}) \ltimes H(\Re b).
\end{equation}
\end{enumerate}
\end{theorem}

\begin{proof}

(1): By Theorems \ref{thm:basisR}, \ref{thm:basisC}, $\Sp(V)$ acts transitively on all the spaces in the diagram (\ref{eq:reduction}), and all the maps in the diagram (\ref{eq:reduction}) are $\Sp(V)$-equivariant. So Lemma \ref{lem:fiber}.3c, all the maps in the diagram (\ref{eq:reduction}) are a fiber bundles. (2), (3): The computation of $N(\pi^{-1}(b))$ and $N(\pi^{-1}(\pi(x)))_x$ follows from Lemma \ref{lem:fiber}.2, and the definitions of the relevant groups. (4): This follows from Lemma \ref{lem:fiber}.1, where we take $N$ to be $\GL(W_0) \ltimes H(W)$. (5): Suppose $\pi = \alpha, \Re \cdot_0, (\cdot \cap \cdot^\omega)$ and $b = W_0 \in \Gr(n_0, 0, n-n_0)$. Under the isomorphism (\ref{eq:fiberaction}), $C(\pi^{-1}(b))$ gets mapped to the the kernel of the action of $\Sp(b^\omega/b)$, which is $Z(\Sp(b^\omega/b))$. Since $Z(\Sp(V))$ is contained in $C(\pi^{-1}(b))$, this map is surjective, giving us (\ref{eq:centralizer1}), as $b = \Re b_0$. So $C(\pi^{-1}(b))$ is generated by $Z(\Sp(V))$ and $\GL(b) \ltimes H(b)$. Take a Darboux basis $\{\ee, \ff\}$ given by Theorem \ref{thm:basisR} associated to $b$. Then writing out the elements in their matrix forms, we can check the short exact sequence in (\ref{eq:centralizer2}). For $\pi = \Re\cdot_{\ge 0}$, $b = W\in \Gr(\vec{n})$, we can prove it in the same way. If $\pi = \varphi$ and $b = \FF \in \Lag^\C(\vec{n})$, $\U(\FF/\FF_0)$ acts transitively on $\varphi^{-1}(\FF)$ via the identification
\begin{equation}\label{eq:fiberaction3} \Sp_b(V)/(\GL(\Re b_0) \ltimes H(\Re b_0)) \xrightarrow{\cong} \U(\FF/\FF_0) \le \Sp((\Re b_0)^\omega/\Re b_0).\end{equation}
If we take a Darboux basis $\{ \ee, \ff\}$ associated to $\FF$ given by Theorem \ref{thm:basisC}, we can inspect that the kernel of this action is $Z(\Sp((\Re b_0)^\omega/\Re b_0))$, and that under the isomorphism in (\ref{eq:fiberaction3}), $C(\pi^{-1}(b))$ gets mapped to the kernel $Z(\Sp((\Re b_0)^\omega/\Re b_0))$. The rest of the proof proceeds as above.
\end{proof}

\begin{remark}[Extending actions on the fibers] The group of linear symplectomorphisms of the reduced space $\Sp(W_0^\omega/W_0)$ acts on $\Lag_\oplus^\C(0, n_+, n_-; W_0^\omega/W_0)$, $\Lag^\C(0, n_+, n_-; W_0^\omega/W_0)$, and $\Gr(0, n_+, n_-; W_0^\omega/W_0)$. For any choice of complementary subspace $W_-$ of $W_0$ in $W_0^\omega$, this action identifies with an action of $\Sp(W_-)$ on $\alpha^{-1}(W_0)$, $(\Re \cdot_0)^{-1}(W_0)$, $(\cdot \cap \cdot^\omega)^{-1}(W_0)$ via $\Sp(W_0^\omega/W_0) \cong \Sp(W_-)$, and further extends to an action on $\Lag_\oplus^\C(\vec{n})$, $\Lag^\C(\vec{n})$, $\Gr(\vec{n})$ via the splitting homomorphism in Theorem \ref{thm:splitW}:
\[ \Sp(W_0^\omega/W_0) \cong \Sp(W_-) \hookrightarrow \Sp_{W_0^\omega}(V) \le \Sp(V).\]
\end{remark}

\begin{remark}[Boundary components and compatible complex structures]
For an $\omega$-compatible linear complex structure $\JJ$, consider the map $\JJ \mapsto \FF_\JJ = (1_{V^\C} - i \JJ)V \in \Lag^\C(0, n, 0)$. This map is a diffeomorphism from the set of $\omega$-compatible linear complex structures to $\Lag^\C(0, n, 0)$. So
when $n_- = 0$, the fibers $\alpha^{-1}(W)$ can be identified with the set of $\tilde{\omega}$-compatible linear complex structures on $W^\omega/W$. They are referred to as \emph{boundary components} in the literature, for instance, \cite{WolfKoranyi}, \cite{BorelJi}, \cite{BaileyBorel}.
\end{remark}

We can also state a version of Theorem \ref{thm:reduction}, when we are given an $\omega$-compatible linear complex structure $\JJ$ on $V$. In this situation, the fibers of $\Re \cdot_{\ge 0}$ and $\varphi$ reduce to points.

\begin{definition}[Orbits of maximal compact subgroups]\label{def:orbitsJ}~\\
Suppose $\JJ$ is an $\omega$-compatible linear complex structure on $V$.
    \begin{enumerate}
    \item Let $\Gr_{\JJ}(\vec{n}) \subseteq \Gr(\vec{n})$ be the submanifold consisting of $W \subseteq V$ of type $\vec{n}$ such that $W \oplus \JJ W_0$ is $\JJ$-invariant.
    \item Let $\Lag^\C_{\JJ}(\vec{n}) \subseteq \Lag^\C(\vec{n})$ be the submanifold consisting of complex Lagrangian subspaces $\FF \subseteq V^\C$ of type $\vec{n}$ such that
    \[ \FF= \FF_0 \oplus (\FF \cap \FF_{\JJ}) \oplus (\FF \cap \FF_{-\JJ}).\]
    \item Let $\Lag^\C_{\JJ, \oplus}(\vec{n}) \subseteq \Lag^\C_{\oplus}(\vec{n})$ be the submanifold  consisting of equivalence classes of complex Lagrangian subspaces of type $\vec{n}$ with splitting $\FF^\oplus$ such that there exist $\FF_{\pm} = \FF \cap \FF_{\pm\JJ}$ such that $\FF^\oplus = [(\FF, \FF_+, \FF_-)]$.
    \end{enumerate}
\end{definition}

Suppose $W = W_0 \in \Gr_\JJ(\vec{n};V)$. Let $W_-^\JJ$ be the $\omega(\cdot, \JJ \cdot)$-orthogonal complement of $W$ in $W^\omega$. By Theorem \ref{thm:basisJ}, there exists a unique $\JJ$-invariant subspace $W_-^\JJ$ of $W^\omega$ complementary to $W$ in $W^\omega$. Moreover, $W_-^\JJ = (W \oplus \JJ W)^\omega$ as $\JJ$-invariant symplectic subspaces of $V$. Then $\omega$ descends to a symplectic form on both $\tilde{W} = W^\omega/W$ and $V/(W \oplus \JJ W)$, and we will denote both by $\tilde{\omega}$. Moreover, there exist $\tilde{\omega}$-compatible linear complex structures on both $\tilde{W} = W^\omega/W$ and $V/(W \oplus \JJ W)$. We will again denote both by $\tilde{\JJ}(W)$, so that
\[ (W_-^\JJ, \omega|_{W_-^\JJ}, \JJ|_{W_-^\JJ}) \cong (W^\omega/W, \tilde{\omega}, \tilde{\JJ}(W)) \cong (V/(W \oplus \JJ W), \tilde{\omega}, \tilde{\JJ}(W))\]
as symplectic vector spaces with compatible linear complex structure. In particular, $\JJ$ descends to every $W^\omega/W$ simultaneously.

We can observe that
\[\iota_\JJ|_{\Lag_{\JJ}^\C(\vec{n})}(\FF):= [\FF_0 \oplus (\FF \cap \FF_{\JJ}) \oplus (\FF \cap \FF_{-\JJ})]\]
is a diffeomorphism from $\Lag_\JJ^\C(\vec{n})$ to $\Lag_{\JJ, \oplus}^\C(\vec{n})$ with inverse $\varphi|_{\Lag_{\JJ, \oplus}^\C (\vec{n})}$.

Similarly,
\[ \psi_{\JJ}(W) := [ W_0 \oplus (1_{V^\C} - i \JJ)W \oplus (1_{V^\C}+i\JJ)W^\omega]\]
is a diffeomorphism from $\Gr_{\JJ}(\vec{n})$ to $\Lag_{\JJ, \oplus}^\C(\vec{n})$ with inverse $\Re \cdot_{\ge 0}|_{\Lag_{\JJ, \oplus}^\C (\vec{n})}$.

\begin{proposition}[Compact orbit fibrations]\label{prop:compactfibrations}~\\
Suppose $\JJ$ be an $\omega$-compatible linear complex structure on $V$, and $K \le G = \Sp(V)$ is the maximal compact subgroup of linear symplectomorphisms of $V$ that commute with $\JJ$.
    \begin{enumerate}
        \item $K$ acts transitively on $\Gr_{\JJ}(\vec{n})$, $\Lag_{\JJ}^\C(\vec{n})$, and $\Lag_{\JJ, \oplus}^\C(\vec{n})$. For each choice of $W \in \Gr_{\JJ}(\vec{n})$, $\FF \in \Lag_{\JJ}^\C(\vec{n})$, and $\FF^\oplus \in \Lag_{\JJ, \oplus}^\C(\vec{n})$ there are diffeomorphisms
        \begin{equation*}
        \begin{aligned}
            \Gr_{\JJ}(\vec{n}) &\cong& K/(K\cap \Sp_W(V)) &\cong& \U(n)/(O(n_0) \times \U(n_+) \times \U(n_-))&\\
            \Lag_{\JJ}^\C(\vec{n}) &\cong& K/(K \cap \Sp_{\FF}(V)) &\cong& \U(n)/(O(n_0) \times \U(n_+) \times \U(n_-))&\\
            \Lag_{\JJ, \oplus}^\C(\vec{n}) &\cong& K/(K \cap \Sp_{\FF^\oplus}(V)) &\cong& \U(n)/(O(n_0) \times \U(n_+) \times \U(n_-))&.
        \end{aligned}
        \end{equation*}
        \item The maps
        \begin{equation*}
        \begin{aligned}
        &(\cdot \cap \cdot^\omega)|_{\Gr_{\JJ}(\vec{n})} : &\Gr_{\JJ}(\vec{n}) &\to& \Gr(n_0, 0, n-n_0)\\
        &(\Re \cdot_0)|_{\Lag_{\JJ}^\C(\vec{n})} : &\Lag_{\JJ}^\C(\vec{n}) &\to& \Gr(n_0, 0, n-n_0)\\
        &\alpha|_{\Lag_{\JJ, \oplus}^\C(\vec{n})} : &\Lag_{\JJ, \oplus}^\C(\vec{n}) &\to& \Gr(n_0, 0, n-n_0)
        \end{aligned}
        \end{equation*}
        are fiber bundles with fibers diffeomorphic to the ordinary complex Grassmannian of complex $n_+$ dimensional subspaces in complex $n-n_0$ dimensional space.
    \end{enumerate}
\end{proposition}
\begin{proof}
(1): If $\FF^\oplus \in \Lag_{\JJ, \oplus}^\C(\vec{n})$, take a Darboux basis $\{\ee, \ff\}$ given by Theorem \ref{thm:basisC}. Since $\ee^+_k - i \ff^+_k$ is in the $+i$-eigenspace of $\JJ$, $\ff^+_k = \JJ \ee^+_k$ for $1 \le k \le n_+$. Similarly, $\ff^-_k = \JJ \ee^-_k$ for $1 \le k \le n_-$. Then we have constructed a Darboux basis $\{ \ee^0, \ee^+, \ee^-, \JJ \ee^0, \JJ \ee^+, \JJ \ee^-\}$ such that $\FF_0 = \span_\C\{\ee^0\}$, $\FF_+ = \span_\C\{ \ee^+ - i \JJ\ee^+\}$, and $\FF_- = \span_\C\{\ee^- + i \JJ \ee^-\}$. Then for any two choices of $\FF^\oplus, (\FF^\oplus)' \in \Lag_{\JJ, \oplus}^\C(\vec{n})$, the linear extension of the map assigning one such Darboux basis obtained as above to the other is a symplectic linear transformation and commutes with $\JJ$. So $K$ acts transitively on $\Lag_{\JJ, \oplus}^\C(\vec{n})$. Since $\Re \cdot_{\ge 0}|_{\Lag_{\JJ, \oplus}^\C (\vec{n})}$ and $\varphi|_{\Lag_{\JJ, \oplus}^\C (\vec{n})}$ are $K$-equivariant diffeomorphisms, $K$ acts transitively on $\Gr_{\JJ}(\vec{n})$ and $\Lag_{\JJ}^\C(\vec{n})$ as well (Lemma \ref{lem:fiber}.3a).
(2): $\alpha|_{\Lag_{\JJ, \oplus}^\C (\vec{n})}$, $\Re \cdot_0|_{\Lag_\JJ^\C (\vec{n})}$, and $(\cdot \cap \cdot^\omega)|_{\Gr_{\JJ}(\vec{n})}$ are $K$-equivariant. Since $K$ acts transitively on the total spaces, we can apply Lemma \ref{lem:fiber}. In particular, these maps are fiber bundles by Lemma \ref{lem:fiber}.3c.

We can also apply Lemma \ref{lem:fiber}.1 to compute the diffeomorphism of the fibers as follows. If $W' \in (\cdot \cap \cdot^\omega)|_{\Gr_{\JJ}(\vec{n})}^{-1}(W_0)$, then $W'/W_0$ is a complex subspace of $(W_0^\omega/W_0, \tilde{\JJ}(W_0))$. Similarly, if $\FF \in (\Re \cdot_0)|_{\Lag_\JJ^\C(\vec{n})}^{-1}(W_0)$, then $\Re((\FF_0 + \FF \cap \FF_{\JJ}))/\Re\FF_0$ is a complex linear subspace of $(W_0^\omega/W_0, \tilde{\JJ}(W_0))$, and if $\FF^\oplus \in \alpha|_{\Lag_{\JJ, \oplus}^\C(\vec{n})}$, then $\Re \FF^\oplus_{\ge 0}$ is a complex subspace of $(W_0^\omega/W_0, \tilde{\JJ}(W_0))$. These maps to the ordinary complex Grassmannian of complex linear subspaces of $(W_0^\omega/W_0, \tilde{\JJ}(W_0))$ are equivariant with respect to $N(\pi^{-1}(b))/N = (K\cap \Sp_{b}(V))/(K \cap (\GL(b) \ltimes H(b)))\xrightarrow{\cong} \Sp(b^\omega/b)^{\Ad(\tilde{\JJ}(W_0))}$, where $b = W_0$, and $\pi = (\cdot \cap \cdot^\omega)|_{\Gr_{\JJ}(\vec{n})}$, $\Re \cdot_0|_{\Lag_\JJ^\C (\vec{n})}$, or $\alpha|_{\Lag_{\JJ, \oplus}^\C (\vec{n})}$. Since these groups act transitively, the maps are diffeomorphisms. 
\end{proof}

\begin{remark}
    In the isotropic and coisotropic case ($n_+ n_- = 0$), $\Gr_{\JJ}(\vec{n}) = \Gr(\vec{n})$ and Proposition \ref{prop:compactfibrations}.1 for $\Gr(\vec{n})$ was obtained in \cite{Arnold} for the Lagrangian case and \cite{OhPark} for the coisotropic case. The statement about the fibers $(\Re \cdot_0|_{\Lag_\JJ^\C(\vec{n})})^{-1}(W_0)$ in Proposition \ref{prop:compactfibrations}.2 appears in Lemma 11.4 of \cite{Wolf1} and Lemma 9.8 of \cite{Wolf}.
\end{remark}


\subsection{Involutions and recompositions}
\label{subsec:involutions}

Suppose $W = W_0$ is an isotropic subspace of $V$ and denote by $\tilde{\omega}$ the descended symplectic form on $\tilde{W}:= W^\omega/W$. For convenience in notation we will suppress the dependence on the choice of $W$ in the following. Denote $\tilde{G}:=\Sp(\tilde{W})$. Let $\tilde{\JJ}$ be a linear complex structure on $\tilde{W}$ compatible with $\tilde{\omega}$. Let $\tilde{I}$ be an involution on $\tilde{W}$ in $\tilde{G}$ such that it commutes with $\tilde{\JJ}$ and its $\pm 1$-eigenspaces have real dimensions $2n_{\pm}$. Since $\tilde{I}$ commutes with $\tilde{\JJ}$, its $\pm 1$-eigenspaces are $\tilde{\JJ}$-invariant, symplectic, and are symplectic complements (and $\tilde{\omega}(\cdot, \tilde{\JJ}\cdot)$-orthogonal complements) of each other. Conversely, a splitting of $\tilde{W}$ into symplectic, $\tilde{\JJ}$-invariant subspaces with the appropriate dimensions determines an $\tilde{I}$ up to sign.

Denote the subgroups of elements of $\tilde{G}$ that, respectively, commute with $\tilde{I}$, $\tilde{\JJ}$, and $\tilde{I}\tilde{\JJ}$ by
\begin{equation*}
\begin{aligned}
\tilde{G}^\shortparallel &= &\tilde{G}(\tilde{I}) &:=& \tilde{G}^{\Ad (\tilde{I})} &\le& \tilde{G}\\
\tilde{K} &= &\tilde{G}(\tilde{\JJ}) &:=& \tilde{G}^{\Ad (\tilde{\JJ})} &\le& \tilde{G}\\
 &&\tilde{G}(\tilde{I}\tilde{\JJ}) &:=& \tilde{G}^{\Ad(\tilde{I}\tilde{\JJ})} &\le& \tilde{G}.
 \end{aligned}
\end{equation*}
Let
\[\tilde{K}^\shortparallel = \tilde{G}(\tilde{I}, \tilde{\JJ}) := \tilde{G}(\tilde{I}) \cap \tilde{G}(\tilde{\JJ}) = \tilde{G}(\tilde{I}) \cap \tilde{G}(\tilde{I}\tilde{\JJ}) = \tilde{G}(\tilde{\JJ}) \cap \tilde{G}(\tilde{I}\tilde{\JJ}).\]

Let $\tilde{\g}$ be the Lie algebra of $\tilde{G}$. Let $\tilde{\k}$ be the $+1$-eigenspace of $\Ad(\tilde{\JJ})$ and $\tilde{\p}$ be the $-1$-eigenspace of $\Ad(\tilde{\JJ})$ in $\tilde{\g}$. Let $\tilde{\g}^\shortparallel$ be the $+1$-eigenspace, and $\tilde{\g}^\times$ be the $-1$-eigenspace of $\Ad(\tilde{I})$ in $\tilde{\g}$. Since $\Ad(\tilde{I})$ commutes with $\Ad(\tilde{\JJ})$, $\tilde{\g}$ has two splittings
\[ \tilde{\g} = \tilde{\g}^\shortparallel \oplus \tilde{\g}^\times = \tilde{\k} \oplus \tilde{\p},\]
orthogonal with respect to the Killing form, such that
\[ \tilde{\g} = (\tilde{\g}^\shortparallel \cap \tilde{\k}) \oplus (\tilde{\g}^\times \cap \tilde{\k}) \oplus (\tilde{\g}^\shortparallel \cap \tilde{\p}) \oplus (\tilde{\g}^\times \cap \tilde{\p})\]
is also an orthogonal splitting. Denote $\tilde{\k}^\shortparallel := \tilde{\g}^\shortparallel \cap \tilde{\k}$, $\k^\times := \tilde{\g}^\times \cap \tilde{\k}$, $\tilde{\p}^\shortparallel := \tilde{\g}^\shortparallel \cap \tilde{\p}$, $\tilde{\p}^\times := \tilde{\g}^\times \cap \tilde{\p}$.

From the descriptions as eigenspaces, we have
\begin{equation}\begin{aligned}
        [\tilde{\k}, \tilde{\k}] &\subseteq \tilde{\k}, &[\tilde{\k}, \tilde{\p}] &\subseteq \tilde{\p}, & [\tilde{\p}, \tilde{\p}] &\subseteq \tilde{\k}\\
        [\tilde{\g}^\shortparallel, \tilde{\g}^\shortparallel] &\subseteq \tilde{\g}^\shortparallel, &[\tilde{\g}^\shortparallel, \tilde{\g}^\times] &\subseteq \tilde{\g}^\times, & [\tilde{\g}^\times, \tilde{\g}^\times] &\subseteq \tilde{\g}^\shortparallel.
\end{aligned} \end{equation}
From these conditions we obtain
\begin{equation}\label{eq:brackets1}[\tilde{\k}^\shortparallel, \tilde{\k}^\shortparallel] \subseteq \tilde{\k}^\shortparallel, \quad [\tilde{\k}^\times, \tilde{\k}^\times] \subseteq \tilde{\k}^\shortparallel, \quad [\tilde{\p}^\shortparallel, \tilde{\p}^\shortparallel] \subseteq \tilde{\k}^\shortparallel, \quad [\tilde{\p}^\times, \tilde{\p}^\times] \subseteq \tilde{\k}^\shortparallel,\end{equation}
\begin{equation}\label{eq:brackets2}[ \tilde{\k}^\shortparallel, \tilde{\k}^\times] \subseteq \tilde{\k}^\times,\quad [\tilde{\k}^\shortparallel, \tilde{\p}^\shortparallel] \subseteq \tilde{\p}^\shortparallel, \quad [ \tilde{\k}^\shortparallel, \tilde{\p}^\times] \subseteq \tilde{\p}^\times.\end{equation}
In particular $\tilde{\k}^\shortparallel$, $\tilde{\g}^\shortparallel = \tilde{\k}^\shortparallel \oplus \tilde{\p}^\shortparallel$, $\tilde{\k} = \tilde{\k}^\shortparallel \oplus \tilde{\k}^\times$, and $\tilde{\k}^\shortparallel \oplus \tilde{\p}^\times$ are Lie subalgebras of $\tilde{\g}$, and by (\ref{eq:brackets2}), $\tilde{K}^\shortparallel$ acts on $\tilde{\p}$, $\tilde{\p}^\shortparallel$, $\tilde{\p}^\times$, $\tilde{\p}/\tilde{\p}^\shortparallel$, and $\tilde{\p}/\tilde{\p}^\times$ by $\Ad$. Moreover,
since $\tilde{I}$ commutes with $\tilde{\JJ}$, $\frac{1}{2} \ad(\tilde{\JJ})$ is a linear complex structure (\ref{eq:adJ}) on $\tilde{\p}^\shortparallel$, $\tilde{\p}^\times$, $\tilde{\p}$, $\tilde{\p}/\tilde{\p}^\shortparallel$, and $\tilde{\p}/\tilde{\p}^\times$ compatible with the restrictions and reductions of the Killing form of $\tilde{\g}$.

By the correspondence between Lie groups and Lie algebras,
\begin{equation*}
\begin{aligned}
    \tilde{\g}^\shortparallel &=&\operatorname{Lie} \tilde{G}(\tilde{I}) &=& \tilde{\k}^\shortparallel \oplus \tilde{\p}^\shortparallel\\
    \tilde{\k} &=&\operatorname{Lie} \tilde{G}(\tilde{\JJ}) &=& \tilde{\k}^\shortparallel \oplus \tilde{\k}^\times\\
    && \operatorname{Lie} \tilde{G}(\tilde{I} \tilde{\JJ}) &=& \tilde{\k}^\shortparallel \oplus \tilde{\p}^\times\\
    && \operatorname{Lie}\tilde{G}(\tilde{I}, \tilde{\JJ}) &=& \tilde{\k}^\shortparallel.
\end{aligned}
\end{equation*}

By Theorem \ref{thm:basisJ} there exists a Darboux basis $\{\tilde{\ee}, \tilde{\JJ} \tilde{\ee}\}$ of $\tilde{W}$ such that $\tilde{I}$ and $\tilde{\JJ}$ are expressed in this basis as
\begin{equation}\label{eq:IJ}  (\tilde{I})_{\{\tilde{\ee}, \tilde{\JJ} \tilde{\ee}\}} = \begin{pmatrix} 1_{n_+} & 0 & 0 & 0\\ 0 & -1_{n_-} & 0 & 0\\ 0 & 0 & 1_{n_+} & 0 \\ 0 & 0 & 0 & -1_{n_-}\end{pmatrix}, \quad (\tilde{\JJ})_{\{\tilde{\ee}, \tilde{\JJ} \tilde{\ee}\}} = \begin{pmatrix} 0 & -1_{n-n_0} \\ 1_{n-n_0} & 0 \end{pmatrix}.\end{equation}

Then we can write the elements of $\tilde{\g}$ in this basis as block matrices
\[ \begin{pmatrix} \aa_{++} & \aa_{+-} & \bb_{++} & \bb_{+-}\\
\aa_{-+} & \aa_{--} & \bb_{-+} & \bb_{--} \\ \cc_{++} & \cc_{+-} & \dd_{++} & \dd_{+-} \\ \cc_{-+} & \cc_{--} & \dd_{-+} & \dd_{--} \end{pmatrix}\]
such that
\begin{equation}\label{eq:matrix1}
    \aa_{\pm \pm} = -\dd_{\pm \pm}^t \in \Mat_{n_{\pm} \times n_{\pm}}(\R),\quad \bb_{\pm \pm}, \cc_{\pm, \pm} \in \Mat_{n_{\pm} \times n_{\pm}}(\R)^{t}, \end{equation}
\begin{equation}\label{eq:matrix2}
    \aa_{-+} = - \dd_{+-}^t, \aa_{+-} = -\dd_{-+}^t, \bb_{-+} = \bb_{+-}^t, \cc_{-+} = \cc_{+-}^t,  \in \Mat_{n_- \times n_+}(\R).\end{equation}

In matrix form the eigenspaces are
\begin{eqnarray}
    \tilde{\k}^\shortparallel &=& \left\{ \begin{pmatrix} \aa_{++}&  0 & \bb_{++} & 0 \\ 0 & \aa_{--} & 0 & \bb_{--}\\
    -\bb_{++} & 0 & \aa_{++} & 0 \\\
    0 & -\bb_{--} & 0 & \aa_{--} \end{pmatrix}:  \begin{aligned} \aa_{\pm \pm} &\in \Mat_{n_{\pm} \times n_{\pm}}(\R)^{-t}\\ \bb_{\pm \pm} &\in \Mat_{n_{\pm} \times n_{\pm}}(\R)^{t} \end{aligned} \right\}\\
    \tilde{\p}^\shortparallel &=& \left\{ \begin{pmatrix}\aa_{++} & 0 & \bb_{++} & 0 \\ 0 & \aa_{--} & 0 & \bb_{--}\\
    \bb_{++} & 0 & -\aa_{++} & 0 \\
    0 & \bb_{--} & 0 & -\aa_{--}\end{pmatrix} : \begin{aligned} \aa_{\pm \pm} &\in \Mat_{n_{\pm} \times n_{\pm}}(\R)^{t}\\ \bb_{\pm \pm} &\in \Mat_{n_{\pm} \times n_{\pm}}(\R)^{t} \end{aligned} \right\} \\
    \tilde{\k}^\times &=& \left\{ \begin{pmatrix} 0 & -\aa_{-+}^t & 0 & \bb_{-+}^t \\
    \aa_{-+} & 0 & \bb_{-+}& 0\\
    0 & -\bb_{-+}^t & 0 &-\aa_{-+}^t\\
    -\bb_{-+} & 0 & \aa_{-+} & 0 \end{pmatrix}: \aa_{-+}, \bb_{-+} \in \Mat_{n_{-} \times n_{+}}(\R)\right\}\\
    \tilde{\p}^\times &=& \left\{ \begin{pmatrix} 0 & \aa_{-+}^t & 0 & \bb_{-+}^t\\
    \aa_{-+} & 0 & \bb_{-+} & 0\\
    0 & \bb_{-+}^t & 0 & -\aa_{-+}^t\\
    \bb_{-+} & 0 & -\aa_{-+} & 0 \end{pmatrix} : \aa_{-+}, \bb_{-+} \in \Mat_{n_{-} \times n_{+}}(\R)\right\}.
\end{eqnarray}

Conditions (\ref{eq:matrix1}), (\ref{eq:matrix2}) also imply
\[ \begin{pmatrix} \aa_{\pm \pm} & \bb_{\pm\pm} \\ \cc_{\pm\pm} & \dd_{\pm\pm} \end{pmatrix} \in \sp(2n_{\pm}; \R)\]
and if $\aa_{\pm \pm} \in \Mat_{n_{\pm} \times n_{\pm}}(\R)^{-t}$ then $\aa_{\pm\pm} + i \bb_{\pm\pm} \in \u(n_{\pm})$, 
\[ \begin{pmatrix} \aa_{++} + i \bb_{++} & - \aa_{-+}^t +i \bb_{-+}^t\\ \aa_{-+} + i \bb_{-+} & \aa_{--} + i \bb_{--} \end{pmatrix} \in \u(n), \quad \begin{pmatrix} \aa_{++} + i \bb_{++} & \aa_{-+}^t + i \bb_{-+}^t \\ \aa_{-+} - i \bb_{-+} & \aa_{--} - i \bb_{--} \end{pmatrix} \in \u(n_+, n_-).\]
So we have isomorphisms of vector spaces
\begin{eqnarray*}
\tilde{\k}^\shortparallel \oplus \tilde{\p}^\shortparallel & \cong& \sp(2n_+; \R) \times \sp(2n_-; \R)\\
\tilde{\k}^\shortparallel \oplus \tilde{\k}^\times &\cong& \u(n)\\
\tilde{\k}^\shortparallel \oplus \tilde{\p}^\times &\cong& \u(n_+, n_-)\\
\tilde{\k}^\shortparallel &\cong& \u(n_+) \times \u(n_-)
\end{eqnarray*}
that are also isomorphism of Lie algebras. In the matrix forms, we can also verify ((\ref{eq:indefunitary}), (\ref{eq:embindefunitary}))
\begin{eqnarray*}
    \tilde{G}(\tilde{I}) &\cong& \Sp(2n_+; \R) \times \Sp(2n_-; \R)\\
    \tilde{G}(\tilde{\JJ}) &\cong& \U(n_+ + n_-)\\
    \tilde{G}(\tilde{I}\tilde{\JJ}) &\cong& \U(n_+, n_-)\\
    \tilde{G}(\tilde{I}, \tilde{\JJ}) &\cong& \U(n_+) \times \U(n_-).
\end{eqnarray*}

\begin{definition}[Compatible triples]
Suppose $W_0 \in \Gr(n_0, 0, n-n_0; V)$, and suppose $\tilde{I}$, $\tilde{\JJ}$ are as above.
\begin{enumerate}
    \item If $\FF^\oplus \in \alpha^{-1}(W_0)$, a triple $(\tilde{I}, \tilde{\JJ}, \FF^\oplus)$ is a \emph{compatible triple} if
    \[ \tilde{G}(\tilde{I}, \tilde{\JJ}) = \Sp_{\FF^\oplus/\FF_0}(\tilde{W}) \cap \tilde{G}(\tilde{\JJ}).\]
    \item If $\FF \in (\Re \cdot_0)^{-1}(W_0)$, a triple $(\tilde{I}, \tilde{\JJ}, \FF)$ is a \emph{compatible triple}  if 
    \[ \tilde{G}(\tilde{I}, \tilde{\JJ}) = \Sp_{\FF/\FF_0} (\tilde{W}) \cap \tilde{G}(\tilde{\JJ}).\]
    \item If $W' \in (\cdot \cap \cdot^\omega)^{-1}(W_0)$, a triple $(\tilde{I}, \tilde{\JJ}, W')$ is a \emph{compatible triple} if
    \[ \tilde{G}(\tilde{I}, \tilde{\JJ}) = \Sp_{W'/W_0}(\tilde{W}) \cap \tilde{G}(\JJ).\]
    \end{enumerate}
\end{definition}

By Theorems \ref{thm:basisR}, \ref{thm:basisC}, and (\ref{eq:IJ}), for every pair $(\tilde{I}, \tilde{\JJ})$, there exists $\FF^\oplus \in \alpha^{-1}(W_0)$, $\FF\in (\Re \cdot_0)^{-1}(W_0)$, $W
\in (\cdot \cap \cdot^\omega)^{-1}(W_0)$ such that $(\tilde{I}, \tilde{\JJ}, \FF^\oplus)$, $(\tilde{I}, \tilde{\JJ}, \FF)$, $(\tilde{I}, \tilde{\JJ}, W')$ are compatible triples. Conversely, for every $\FF^\oplus \in \alpha^{-1}(W_0)$, $\FF\in (\Re \cdot_0)^{-1}(W_0)$, $W' \in (\cdot \cap \cdot^\omega)^{-1}(W_0)$, there exist $(\tilde{I}, \tilde{\JJ})$ such that $(\tilde{I}, \tilde{\JJ}, \FF^\oplus)$, $(\tilde{I}, \tilde{\JJ}, \FF)$, $(\tilde{I}, \tilde{\JJ}, W')$ are compatible triples.

\subsection{Fibers of $\Re \cdot_\ge 0$ and $\varphi$ as symmetric spaces}
\label{subsec:symmetricspace}
We have the tower of real Lie groups
\begin{equation} \label{eq:tower}
\begin{tikzcd}
    & \tilde{G} = \Sp(\tilde{W})\arrow[dl, dash] \arrow[d, dash] \arrow[dr, dash]&\\ \tilde{G}^\shortparallel = \tilde{G}(\tilde{I}) \arrow[dr, dash] & \tilde{K} = \tilde{G}(\tilde{\JJ}) \arrow[d, dash] & \tilde{G}(\tilde{I}\tilde{\JJ}) \arrow[dl, dash]\\ & \tilde{K}^\shortparallel = \tilde{G}(\tilde{I}, \tilde{\JJ}). &
\end{tikzcd}
\end{equation}

Then from Lemma \ref{lem:orbitspacefibration} and Lemma \ref{lem:fiber}.3.c we obtain the following fiber bundles:
\begin{eqnarray}
    \tilde{G}(\tilde{I})/\tilde{G}(\tilde{I}, \tilde{\JJ}) \hookrightarrow & \tilde{G}/\tilde{G}(\tilde{I}, \tilde{\JJ}) & \twoheadrightarrow \tilde{G}/\tilde{G}(\tilde{I})\label{eq:grreducedbundle}\\
    \tilde{G}(\tilde{\JJ})/\tilde{G}(\tilde{I}, \tilde{\JJ}) \hookrightarrow & \tilde{G}/\tilde{G}(\tilde{I}, \tilde{\JJ}) & \twoheadrightarrow \tilde{G}/\tilde{G}(\tilde{\JJ})\label{eq:leeleung}\\
    \tilde{G}(\tilde{I}\tilde{\JJ})/\tilde{G}(\tilde{I}, \tilde{\JJ}) \hookrightarrow & \tilde{G}/\tilde{G}(\tilde{I}, \tilde{\JJ}) & \twoheadrightarrow \tilde{G}/\tilde{G}(\tilde{I}\tilde{\JJ}). \label{eq:lagreducedbundle}
\end{eqnarray}

Now choose an $\FF^\oplus \in \Lag_\oplus^\C(\vec{n}; V)$ such that $(\tilde{I}, \tilde{\JJ}, \FF^\oplus)$ is a compatible triple. Taking $\FF^\oplus$, $\FF^\oplus$, and $\Re \FF^\oplus_{\ge 0}$ as basepoints, we can identify the fiber bundle in (\ref{eq:grreducedbundle}) as
\begin{equation} \begin{tikzcd} \tilde{G}(\tilde{I})/\tilde{G}(\tilde{I}, \tilde{\JJ}) \arrow[r, hook] \arrow[d, "\cong"] & \tilde{G}/\tilde{G}(\tilde{I}, \tilde{\JJ}) \arrow[r, two heads] \arrow[d, "\cong"] & \tilde{G}/\tilde{G}(\tilde{I}) \arrow[d, "\cong"]\\
(\Re \cdot_{\ge 0})^{-1}(\Re \FF^\oplus_{\ge 0 }) \arrow[r, hook] & \alpha^{-1}(\Re \FF_0) \arrow[r, two heads, "\Re \cdot_{\ge 0}"'] & (\cdot \cap \cdot^\omega)^{-1}(\Re \FF_0). \end{tikzcd}\end{equation}

Similarly, choose an $\FF^\oplus \in \Lag_\oplus^\C(\vec{n}; V)$ such that $(\tilde{I}, \tilde{\JJ}, \FF^\oplus)$ is a compatible triple. Taking $\FF^\oplus$, $\FF^\oplus$, and $\varphi(\FF^\oplus)$ as basepoints, we can identify the fiber bundle in (\ref{eq:lagreducedbundle}) as

\begin{equation} \begin{tikzcd} \tilde{G}(\tilde{I} \tilde{\JJ})/\tilde{G}(\tilde{I}, \tilde{\JJ}) \arrow[r, hook] \arrow[d, "\cong"] & \tilde{G}/\tilde{G}(\tilde{I}, \tilde{\JJ}) \arrow[r, two heads] \arrow[d, "\cong"] & \tilde{G}/\tilde{G}(\tilde{I}\tilde{\JJ}) \arrow[d, "\cong"]\\
\varphi^{-1}(\varphi(\FF^\oplus)) \arrow[r, hook] & \alpha^{-1}(\Re \FF_0) \arrow[r, two heads, "\varphi"'] & (\Re \cdot_0)^{-1}(\Re \FF_0). \end{tikzcd}\end{equation}

The fiber bundle in (\ref{eq:leeleung}) can be identified with the the fiber bundle $\beta$ of the symplectic twistor Grassmannian in Proposition 17 of \cite{LeeLeung}.

Out of the seven left coset spaces that can be obtained from (\ref{eq:tower}), the following six of them
\[ \tilde{G}(\tilde{I})/\tilde{G}(\tilde{I}, \tilde{\JJ}), \tilde{G}(\tilde{\JJ})/\tilde{G}(\tilde{I}, \tilde{\JJ}), \tilde{G}(\tilde{I}\tilde{\JJ})/\tilde{G}(\tilde{I}, \tilde{\JJ}), \tilde{G}/\tilde{G}(\tilde{I}), \tilde{G}/\tilde{G}(\tilde{\JJ}), \tilde{G}/\tilde{G}(\tilde{I}\tilde{\JJ}) \]
are symmetric spaces, as the subgroup is the subgroup of elements fixed by an involution (Theorem II.1.3a of \cite{Loos}).

Out of these six symmetric spaces, the following four 
\[ \tilde{G}(\tilde{I})/\tilde{G}(\tilde{I}, \tilde{\JJ}), \tilde{G}(\tilde{I}\tilde{\JJ})/\tilde{G}(\tilde{I}, \tilde{\JJ}), \tilde{G}(\tilde{\JJ}) / \tilde{G}(\tilde{I}, \tilde{\JJ}), \tilde{G}/\tilde{G}(\tilde{\JJ})\]
are hermitian symmetric spaces, since $\tilde{G}(\tilde{I}, \tilde{\JJ})$ and $\tilde{G}(\tilde{\JJ})$ are compact. In the same way that we chose the K\"{a}hler structure on $\Lag^\C(V)$ in Remark \ref{rem:complexstructures}, we can construct invariant K\"{a}hler structures on $\tilde{G}(\tilde{I})/\tilde{G}(\tilde{I}, \tilde{\JJ})$, $\tilde{G}(\tilde{I}\tilde{\JJ})/\tilde{G}(\tilde{I}, \tilde{\JJ})$, $\tilde{G}/\tilde{G}(\tilde{\JJ})$, from the fact that $\frac{1}{2}\ad(\tilde{\JJ})$ is a linear complex structure on $\tilde{\p}^\shortparallel$, $\tilde{\p}^\times$, $\tilde{\p}$.

Since $\tilde{G}(\tilde{\JJ})$ is compact and $\tilde{G}$, $\tilde{G}(\tilde{I})$, $\tilde{G}(\tilde{I}\tilde{\JJ})$, are noncompact, $\tilde{G}(\tilde{\JJ}) / \tilde{G}(\tilde{I}, \tilde{\JJ})$ is a hermitian symmetric space of compact type, and $\tilde{G}(\tilde{I})/\tilde{G}(\tilde{I}, \tilde{\JJ}), \tilde{G}(\tilde{I}\tilde{\JJ})/\tilde{G}(\tilde{I}, \tilde{\JJ}), \tilde{G}/\tilde{G}(\tilde{\JJ})$
are hermitian symmetric spaces of noncompact type. By choosing basepoints and by Theorem VI.1.1.(iii) of \cite{Helgason}, we have
\begin{equation}\label{eq:vectorspace}\begin{aligned}
(\Re \cdot_{\ge 0})^{-1}(W) &\cong& \tilde{G}(\tilde{I})/\tilde{G}(\tilde{I}, \tilde{\JJ}) &\cong \tilde{\p}^\shortparallel&\\
\varphi^{-1}(\FF)  &\cong& \tilde{G}(\tilde{I}\tilde{\JJ})/\tilde{G}(\tilde{I}, \tilde{\JJ}) &\cong \tilde{\p}^\times&\\
&&\tilde{G}/\tilde{G}(\tilde{\JJ}) &\cong \tilde{\p}&\end{aligned}
\end{equation}
for $W \in \Gr(\vec{n}; V)$ and $\FF \in \Lag^\C(\vec{n}; V)$. In particular, $(\Re \cdot_{\ge 0})^{-1}(W)$ and $\varphi^{-1}(\FF)$ obtain complex structures from the invariant complex structures on $\tilde{G}(\tilde{I})/\tilde{G}(\tilde{I}, \tilde{\JJ}))$ and $\tilde{G}(\tilde{I}\tilde{\JJ})/\tilde{G}(\tilde{I}, 
\tilde{\JJ})$. By the Harish-Chandra embedding, these spaces are biholomorphic to bounded symmetric domains.

\section{Factorizations of orbit fibrations}
\label{sec:factorization}

In this section, we construct the fiber bundles of interest $\gamma_\JJ$, $\beta_\JJ$, $\eta_\JJ$ as a factor of the fiber bundles of simultaneous linear symplectic reduction, first in the reduced case, and then in the nonreduced case. We follow the proof of Theorem 2 of \cite{Takeuchi}. We also discuss how the symplectic twistor Grassmannians are as fiber products of $\gamma_\JJ$ and $\beta_{\JJ}$, with a corrected diffeomorphism.

\subsection{Fibers of $\alpha$, $\Re \cdot_0$, and $(\cdot \cap \cdot^\omega)$ as vector bundles}

Since $\tilde{G}(\tilde{I})$ and $\tilde{G}(\tilde{I}\tilde{\JJ})$ are noncompact, the two symmetric spaces $\tilde{G}/\tilde{G}(\tilde{I})$, $\tilde{G}/\tilde{G}(\tilde{I}\tilde{\JJ})$ do not admit $\tilde{G}$-invariant Riemannian metrics (Theorem IV.2.5b of \cite{Helgason}). Identifying by choosing appropriate basepoints we get
\begin{equation}
\begin{aligned}
(\cdot \cap \cdot^\omega)^{-1}(W_0) &\cong & \Gr(0, n_+, n_-; \tilde{W}) &\cong & \tilde{G}/\tilde{G}(\tilde{I})\\
(\Re \cdot_0)^{-1}(W_0) & \cong & \Lag^\C(0, n_+. n_-; \tilde{W}) &\cong & \tilde{G}/\tilde{G}(\tilde{I}\tilde{\JJ}),
\end{aligned} \end{equation}
and applying Theorem IV.3.5 of \cite{Loos}, we obtain diffeomorphisms to total spaces of vector bundles:
\begin{equation}
\begin{aligned}
  (\cdot \cap \cdot^\omega)^{-1}(W_0) &\cong&  \tilde{G}/\tilde{G}(\tilde{I}) &\cong& \tilde{K} \times_{\tilde{K}^\shortparallel} \tilde{\p}^\times&\\
      (\Re \cdot_0)^{-1}(W_0) &\cong & \tilde{G}/\tilde{G}(\tilde{I}\tilde{\JJ}) &\cong& \tilde{K} \times_{\tilde{K}^\shortparallel} \tilde{\p}^\shortparallel.&
     \end{aligned}
\end{equation}

For the left coset space, $\tilde{G}/\tilde{G}(\tilde{I}, \tilde{\JJ})$
we show an analogous diffeomorphism directly, and reshow it for the symmetric spaces $\tilde{G}/\tilde{G}(\tilde{I})$ and $\tilde{G}/\tilde{G}(\tilde{I}\tilde{\JJ})$. From the Cartan decomposition of Lie groups
\[ \tilde{G}(\tilde{I}) \cong \tilde{K}^\shortparallel \times \tilde{\p}^\shortparallel, \quad \tilde{G}(\tilde{I}\tilde{\JJ}) \cong \tilde{K}^\shortparallel \times \tilde{\p}^\times,\]
we use the fact that every $\tilde{g} \in \tilde{G}(\tilde{I})$ can be expressed uniquely as $\tilde{g} = \tilde{k}^\shortparallel \exp \tilde{p}^\shortparallel$ for some $\tilde{k}^\shortparallel \in \tilde{K}^\shortparallel$ and $\tilde{p}^\shortparallel \in \tilde{\p}^\shortparallel$, and every $\tilde{g} \in \tilde{G}(\tilde{I}\tilde{\JJ})$ can be expressed uniquely as $\tilde{g} = \tilde{k}^\shortparallel \exp \tilde{p}^\times$ for some $\tilde{k}^\shortparallel \in \tilde{K}^\shortparallel$ and $\tilde{p}^\times \in \tilde{\p}^\times$.

\begin{lemma}\label{lem:fibers} ~
Suppose $W_0 \in \Gr(n_0, 0, n-n_0; V)$, and $(\tilde{I}, \tilde{\JJ}, \FF^\oplus)$, $(\tilde{I}, \tilde{\JJ}, \FF)$, $(\tilde{I}, \tilde{\JJ}, W')$ are compatible triples. 
\begin{enumerate}
    \item There are diffeomorphisms
    \begin{eqnarray*}
    \tilde{K} \times_{\tilde{K}^\shortparallel} \tilde{\p} &\cong&\alpha^{-1}(W_0)\\
    {[(\tilde{k}, \tilde{p})]} & \mapsto& \tilde{k} \exp \tilde{p}.\FF^\oplus\\
   \tilde{K}\times_{\tilde{K}^\shortparallel} (\tilde{\p}/\tilde{\p}^\times)&\cong&(\Re \cdot_0)^{-1}(W_0)\\
    {[(\tilde{k}, \tilde{p}^\shortparallel + \tilde{\p}^\times)]} &\mapsto& \tilde{k} \exp \tilde{p}^\shortparallel.\FF\\
    \tilde{K}\times_{\tilde{K}^\shortparallel} (\tilde{\p}/\tilde{\p}^\shortparallel) &\cong &(\cdot \cap \cdot^\omega)^{-1}(W_0)\\
    {[(\tilde{k}, \tilde{p}^\times + \tilde{\p}^\shortparallel)]} &\mapsto& \tilde{k} \exp \tilde{p}^\times.W'.
\end{eqnarray*}

$\tilde{G}$ acts on the fiber spaces by $\tilde{G} = \Sp(W_0^\omega/W_0) \cong \Sp_{W_0}(V)/(\GL(W_0) \ltimes H(W_0))$, and the diffeomorphisms are $\tilde{K} = \tilde{G}(\tilde{\JJ})$-equivariant.
\item The diagrams
\begin{equation}\label{eq:alphasquar}
\begin{tikzcd} \tilde{K} \times_{\tilde{K}^\shortparallel} \tilde{\p} \arrow[d, two heads] & \alpha^{-1}(W_0) \arrow[d, "\pi_{(\tilde{I}, \tilde{\JJ}, \FF^\oplus)}"] \arrow[l, "\cong"] \\ \tilde{K}/\tilde{K}^\shortparallel \arrow[r, "\cong"]& \tilde{K}.\FF^\oplus\end{tikzcd}\end{equation}
\begin{equation}\label{eq:Re0square}
\begin{tikzcd} \tilde{K}\times_{\tilde{K}^\shortparallel} (\tilde{\p}/\tilde{\p}^\times)  \arrow[d, two heads] & (\Re \cdot_0)^{-1}(W_0) \arrow[d, "\pi_{(\tilde{I}, \tilde{\JJ}, \FF)}"] \arrow[l, "\cong"] \\ \tilde{K}/\tilde{K}^\shortparallel \arrow[r, "\cong"]& \tilde{K}.\FF\end{tikzcd}\end{equation}
\begin{equation}\label{eq:kernelsquare}
\begin{tikzcd} \tilde{K} \times_{\tilde{K}^\shortparallel} (\tilde{\p}/\tilde{\p}^\shortparallel) \arrow[d, two heads] & (\cdot \cap \cdot^\omega)^{-1}(W_0) \arrow[l, "\cong"] \arrow[d, "
\pi_{(\tilde{I}, \tilde{\JJ}, W')}"] \\ \tilde{K}/\tilde{K}^\shortparallel \arrow[r, "\cong"]& \tilde{K}.W'\end{tikzcd}\end{equation}
commute. Here the upper horizontal arrows are defined from the previous statement, and
\begin{eqnarray} 
    \pi_{(\tilde{I}, \tilde{\JJ}, \FF^\oplus)} (\tilde{k} \exp \tilde{p}.\FF^\oplus) &:=& \tilde{k}.\FF^\oplus \label{eq:vectorbundles1}\\
    \pi_{(\tilde{I}, \tilde{\JJ}, \FF)}(\tilde{k} \exp(\tilde{p}^\shortparallel).\FF) &:=& \tilde{k}.\FF\label{eq:vectorbundles2}\\
    \pi_{(\tilde{I}, \tilde{\JJ}, W')}(\tilde{k}\exp(\tilde{p}^\times).W') &:=& \tilde{k}.W'\label{eq:vectorbundles3}
\end{eqnarray}
are $\tilde{K}$-equivariant.
\item The maps (\ref{eq:vectorbundles1}), (\ref{eq:vectorbundles2}), (\ref{eq:vectorbundles3}) depend on the choices involved in defining them as follows. If $g \in \tilde{G}$
\begin{eqnarray*}
    g^{-1} \cdot \pi_{(\Ad(g)\tilde{I}, \Ad(g)\tilde{\JJ}, g.\FF^\oplus)} (g \cdot) &=& \pi_{(\tilde{I}, \tilde{\JJ}, \FF^\oplus)} (\cdot)\\
    g^{-1} \cdot \pi_{(\Ad(g)\tilde{I}, \Ad(g)\tilde{\JJ}, g.\FF)}(g \cdot) &=& \pi_{(\tilde{I}, \tilde{\JJ}, \FF)}(\cdot)\\
    g^{-1} \cdot\pi_{(\Ad(g)\tilde{I}, \Ad(g)\tilde{\JJ}, g.W')}(g\cdot) &=& \pi_{(\tilde{I}, \tilde{\JJ}, W')} (\cdot).
\end{eqnarray*}
The maps (\ref{eq:vectorbundles1}), (\ref{eq:vectorbundles2}), (\ref{eq:vectorbundles3}) do not depend on the choice of basepoint in each $\tilde{K}$-orbit--i.e. for any $\tilde{k} \in \tilde{K}$,
\begin{eqnarray*}
    \pi_{(\Ad(\tilde{k})\tilde{I}, \Ad(\tilde{k})\tilde{\JJ}, \tilde{k}.\FF^\oplus)} &=& \pi_{(\tilde{I}, \tilde{\JJ}, \FF^\oplus)}\\
    \pi_{(\Ad(\tilde{k})\tilde{I}, \Ad(\tilde{k})\tilde{\JJ}, \tilde{k}.\FF)} &=& \pi_{(\tilde{I}, \tilde{\JJ}, \FF)}\\
    \pi_{(\Ad(\tilde{k})\tilde{I}, \Ad(\tilde{k})\tilde{\JJ}, \tilde{k}.W')} &=& \pi_{(\tilde{I}, \tilde{\JJ}, W')}.
\end{eqnarray*}
\end{enumerate}
\end{lemma}

\begin{proof} (1): The map $\tilde{K} \times \tilde{\p} \to \tilde{G}$ given by $(\tilde{k}, \tilde{p}) \mapsto \tilde{k}\cdot\exp \tilde{p}$ is a diffeomorphism (Theorem VI.1.1 (iii) of \cite{Helgason}), and descends to a diffeomorphism on $\tilde{K} \times_{\tilde{K}^\shortparallel} \tilde{\p} \to \tilde{G}/\tilde{K}^\shortparallel$, as
\[ \tilde{k} \cdot \exp(p)\cdot \tilde{k}^\shortparallel = \tilde{k}\tilde{k}^\shortparallel \cdot \exp (\Ad(\tilde{k}^\shortparallel)^{-1} \tilde{p}) \quad \tilde{k}^\shortparallel \in \tilde{K}^\shortparallel, \tilde{p} \in \tilde{\p}.\]
Since $\tilde{K}^\shortparallel$ is identified with the stabilizer of $\FF^\oplus$ in $\Sp_W(V)/(\GL(V) \ltimes H(W))$, $\tilde{G}/\tilde{K}^\shortparallel$ is diffeomorphic to $\alpha^{-1}(W)$ by the map $g\tilde{K}^\shortparallel \mapsto g\tilde{K}^\shortparallel.\FF^\oplus = g.\FF^\oplus$. 

Consider the map $\tilde{K} \times \tilde{\p}^\shortparallel \times \tilde{\p}^\times \to \tilde{G}$ given by $(\tilde{k}, \tilde{p}^\shortparallel, \tilde{p}^\times) \mapsto \tilde{k} \cdot \exp \tilde{p}^\shortparallel \cdot \exp \tilde{p}^\times$. This is a diffeomorphism by Theorem IV.3.2 of \cite{Loos}, Lemma 6 of \cite{Takeuchi}, Lemma 9.10 of \cite{Wolf}, Lemma 11.6 of \cite{Wolf1}. For any $\tilde{k}_1^\shortparallel \exp \tilde{p}_1^\times \in \tilde{G}(\tilde{I}\tilde{\JJ})$, there exists $\tilde{k}_2^\shortparallel \exp \tilde{p}_2^\times \in \tilde{G}(\tilde{I}\tilde{\JJ})$ such that
\begin{eqnarray*}
    \tilde{k}\exp \tilde{p}^\shortparallel\exp \tilde{p}^\times \tilde{k}_1^\shortparallel \exp \tilde{p}_1^\times &=& \tilde{k} \tilde{k}_1^\shortparallel \exp (\Ad(\tilde{k}_1^\shortparallel)^{-1} \tilde{p}^\shortparallel) \exp (\Ad(\tilde{k}_1^\shortparallel)^{-1} \tilde{p}^\times) \exp \tilde{p}_1^\times\\&=&
    \tilde{k} \tilde{k}_1^\shortparallel \exp (\Ad(\tilde{k}_1^\shortparallel)^{-1} \tilde{p}^\shortparallel) \tilde{k}_2^\shortparallel \exp \tilde{p}_2^\times,
\end{eqnarray*}
because $\exp (\Ad(\tilde{k}_1^\shortparallel)^{-1} \tilde{p}^\times) \exp \tilde{p}_1^\times \in \tilde{G}(\tilde{I}\tilde{\JJ})$, and every element of $\tilde{G}(\tilde{I}\tilde{\JJ})$ can be expressed in the form $\tilde{k}_2^\shortparallel \exp \tilde{p}_2^\times$.
So the diffeomorphism $\tilde{K} \times \tilde{\p}^\shortparallel \times \tilde{\p}^\times \to \tilde{G}$ descends to a diffeomorphism
\[ \tilde{K} \times_{\tilde{K}^\shortparallel} (\tilde{\p}/\tilde{\p}^\times) \to \tilde{G}/\tilde{G}(\tilde{I}\tilde{\JJ}).\]
Since $(\tilde{I}, \tilde{\JJ}, \FF)$ is a compatible triple, $\tilde{G}(\tilde{I}\tilde{\JJ})$ is the stabilizer of $\FF$ in $\tilde{G}$, so $\tilde{G}/ \tilde{G}(\tilde{I}\tilde{\JJ})$ is diffeomorphic to $(\Re \cdot_0)^{-1}$. The diffeomorphism for $(\cdot \cap \cdot^\omega)^{-1}(W)$ is proven similarly. (2) is immediate. (3):

Suppose $\tilde{\FF}^\oplus = \tilde{k}.\exp(p).\FF^\oplus = \tilde{k}g^{-1} \exp(gpg^{-1}).(g.\FF^\oplus)$. Then
\begin{eqnarray*}
    g^{-1} \cdot \pi_{(\Ad(g)\tilde{I}, \Ad(g)\tilde{\JJ}, g.\FF^\oplus)}(g.\tilde{\FF}^\oplus)
    &=& g^{-1}\cdot g\tilde{k}g^{-1}.(g.\FF^\oplus)\\
    &=& \pi_{(\tilde{I}, \tilde{\JJ}, \FF^\oplus)}(\tilde{\FF}^\oplus).
\end{eqnarray*}
The group action for (\ref{eq:vectorbundles2}), (\ref{eq:vectorbundles3}) can be verified similarly. The independence property follows by $\tilde{K}$-equivariance of the maps (\ref{eq:vectorbundles1}), (\ref{eq:vectorbundles2}), (\ref{eq:vectorbundles3}).
\end{proof}

\begin{remark}
If $\tilde{\JJ}$ is descended from an $\omega$-compatible linear complex structure $\JJ$ on $V$, and $\FF^\oplus \in \Lag_{\JJ, \oplus}^\C(\vec{n})$, $\FF \in \Lag_{\JJ}^\C(\vec{n})$, or $W' \in \Gr_{\JJ}(\vec{n})$, we can take
\[ \begin{aligned} \tilde{K}.\FF^\oplus &=& (\alpha|_{\Lag_{\JJ, \oplus}^\C(\vec{n})})^{-1}(W_0)\\
\tilde{K}.\FF &=& (\Re \cdot_{0}|_{\Lag_{\JJ}^\C(\vec{n})})^{-1}(W_0)\\
\tilde{K}.W' &=& ((\cdot \cap \cdot^{\omega})|_{\Gr_{\JJ}(\vec{n})})^{-1}(W_0)\end{aligned}\]
in the diagrams (\ref{eq:alphasquar}), (\ref{eq:Re0square}), (\ref{eq:kernelsquare}).
\end{remark}

\begin{remark}
When $W_0$ is the zero vector space, the diffeomorphism  for $(\Re \cdot_0)^{-1}(W_0) \cong \Lag^\C(0, n_+,n_-; V)$ of Lemma \ref{lem:fibers}.1 appears in Theorem 2 of \cite{Takeuchi}, and the associated bundle in (\ref{eq:Re0square}) is \emph{isomorphic as vector bundles} to the normal bundle of $\Lag^\C_\JJ(0, n_+, n_-; V)$ in $\Lag^\C(0, n_+, n_-; V)$. Our proof follows the idea in the proof of Theorem 2 of \cite{Takeuchi}. Compare with the holomorphic version in Theorem 11.8.2 of \cite{Wolf1} and the second statement of the Orbit Fibration Theorem of \cite{Wolf}, where $\Lag^\C(0, n_+, n_-; V)$ is identified with a \emph{relatively compact} tubular neighborhood of the holomorphic normal bundle of $\Lag^\C_\JJ(0, n_+, n_-; V)$ in $\Lag^\C(0, n_+, n_-; V)$.\end{remark}

\subsection{Factorization of orbit fibrations}
In this section, we will assume the $\tilde{\omega}$-compatible linear complex structure $\tilde{\JJ} = \tilde{\JJ}(W)$ on $\tilde{W} = W^\omega/W$ is induced from $\JJ$ as a function of every isotropic subspace $W=W_0 \subseteq V$.

\begin{proposition}[Uniqueness of compatible triples]~
\begin{enumerate}
\item If $W' \in (\cdot \cap \cdot^\omega)|_{\Gr_{\JJ}(\vec{n})}^{-1}(W_0)$, then there exists a unique involution 
\[ \tilde{I}(W') \in \Sp_{W'/W_0}(W_0^\omega/W_0)\]
such that $(\tilde{I}(W'), \tilde{\JJ}(W_0), W')$ is a compatible triple and $W'/W_0$ is the linear subspace fixed by $\tilde{I}(W')$.
\item If $\FF \in (\Re \cdot_0)|_{\Lag_{\JJ}^\C(\vec{n})}^{-1}(W_0)$, then there exists a unique involution
\[\tilde{I}(\FF)\in\Sp_{\FF/\FF_0}(W_0^\omega/W_0)\]
such that $(\tilde{I}(\FF), \tilde{\JJ}(W_0), \FF)$ is a compatible triple and $(\FF/\FF_0) \cap \FF_{\tilde{\JJ}(W_0)}$ is the linear subspace fixed by $\tilde{I}(\FF)|_{\FF/\FF_0}$.
\item If $\FF^\oplus \in \alpha|_{\Lag^\C_{\JJ, \oplus}(\vec{n})}^{-1}(W_0)$, then there exists a unique involution
\[ \tilde{I}(\FF^\oplus) \in \Sp_{\FF^\oplus/\FF_0}(W_0^\omega/W_0)\]
such that $(\tilde{I}(\FF^\oplus), \tilde{\JJ}(W_0), \FF^\oplus)$ is a compatible triple and $\FF^\oplus_{\ge 0} / \FF_0$ is the linear subspace fixed by $\tilde{I}(\FF^\oplus)|_{\FF/\FF_0}$.
\end{enumerate}
\end{proposition}

\begin{proof}
(1): Consider the splitting 
\[ W_0^\omega/W_0 = (W'/W_0) \oplus ((W')^\omega/W_0)\]
into $\tilde{\JJ}(W_0)$-invariant symplectic subspaces. Then let $\tilde{I}(W')(\ww + \ww'):= \ww' - \ww''$ for $\ww \in W'/W_0$, $\ww'' \in (W')^\omega/W_0$. $\tilde{I}(W')$ is an involution, its $+1$-eigenspace is $W'/W_0$, and commutes with $\tilde{\JJ}(W_0)$. If $g \in \Sp_{W'/W_0}(W_0^\omega/W_0)^{\Ad(\tilde{\JJ}(W_0))}$, then since $(W')^\omega/W_0$ is the $\tilde{\omega}(\cdot, \tilde{\JJ}(W_0)\cdot)$-orthogonal complement of $W'/W_0$ and $g$ preserves orthogonal complements, $g = g|_{W'/W_0} + g|_{(W')^\omega/W_0}$. Then $g$ commutes with $\tilde{I}(W')$. Conversely, if $g$ commutes with both $\tilde{\JJ}(W_0)$ and $\tilde{I}(W')$, $g$ preserves the eigenspaces of $\tilde{I}(W')$ and $\tilde{\omega}(\cdot, \tilde{\JJ}(W_0)\cdot)$ orthogonal complements, so $g = g|_{W'/W_0} + g|_{(W')^\omega/W_0}$, so $g \in \Sp_{W'/W_0}(W_0^\omega/W_0)^{\Ad(\tilde{\JJ(W_0)})}$. So $(\tilde{I}(W'), \tilde{\JJ}(W_0), W')$ is a compatible triple. If $\tilde{I}' \in \Sp_{W'/W_0}(W_0^\omega/W_0)$ is an involution that commutes $\tilde{\JJ}(W_0)$, and whose $+1$-eigenspace is $W'/W_0$, then its $-1$-eigenspace has to be equal to the $\tilde{\omega}(\cdot, \tilde{\JJ}(W_0) \cdot)$-orthogonal complement of $W'/W_0$. So $\tilde{I}' = \tilde{I}(W')$.

(2): $(W_0^\omega/W_0)^\C$ splits into $\tilde{\omega}^\C(\cdot, \tilde{\JJ}(W_0) \overline{\cdot})$-orthogonal subspaces
\begin{equation*}((\FF/\FF_0) \cap \FF_{\tilde{\JJ}(W_0)}) \oplus (\overline{(\FF/\FF_0) \cap \FF_{\tilde{\JJ}(W_0)}}) \oplus ((\FF/\FF_0) \cap \FF_{-\tilde{\JJ}(W_0)}) \oplus (\overline{(\FF/\FF_0) \cap \FF_{-\tilde{\JJ}(W_0)}}).\end{equation*}
Then let $\tilde{I}(\FF)(\ww + \overline{\ww} + \ww' + \overline{\ww}'):= \ww + \overline{\ww} - \ww' -\overline{\ww}'$ for $\ww \in (\FF/\FF_0) \cap \FF_{\tilde{\JJ}(W_0)})$, $\overline{\ww} \in \overline{(\FF/\FF_0) \cap \FF_{\tilde{\JJ}(W_0)})}$, $\ww' \in (\FF/\FF_0) \cap \FF_{-\tilde{\JJ}(W_0)})$, $\overline{\ww}' \in \overline{(\FF/\FF_0) \cap \FF_{-\tilde{\JJ}(W_0)})}$. The rest of the proof can be verified similarly as above, considering additionally the fact that $\tilde{I}(\FF)$ commutes with complex conjugation.
(3): We can proceed similarly to (2), from the splitting of $(W_0^\omega/W_0)^\C$ into $\tilde{\omega}^\C(\cdot, \tilde{\JJ}(W_0) \overline{\cdot})$-orthogonal subspaces
\[ (\FF^\oplus_{\ge 0}/\FF_0) \oplus (\overline{\FF^\oplus_{\ge 0}/\FF_0}) \oplus (\FF^\oplus_{\le 0}/\FF_0) \oplus (\overline{\FF^\oplus_{\le 0}/\FF_0}).\]
\end{proof}

So when an $\omega$-compatible linear complex structure on $V$ is given, a choice of compatible triple is determined by a choice of basepoint.

We will denote any choice of $W' \in (\cdot \cap \cdot^\omega)|_{\Gr_{\JJ}(\vec{n})}^{-1}(W_0)$, $\FF \in (\Re \cdot_0)|_{\Lag_{\JJ}^\C(\vec{n})}^{-1}(W_0)$, and $\FF^\oplus \in \alpha|_{\Lag^\C_{\JJ, \oplus}(\vec{n})}^{-1}(W_0)$ by $b = b_{\JJ, W_0}$, so we can express how the choice of involution depends on other choices as
\[ \tilde{I} = \tilde{I}(b) = \tilde{I}(b, \JJ, W_0).\]

Before we provide the proof of the main theorem, we briefly recall how different objects depend on others. The objects that depend only on $\JJ$ are $K$, $\k$, and $\p$, where $K = G(\JJ) = G^{\Ad(\JJ)} \le G = \Sp(V)$,
and $\operatorname{Lie}(G) = \g = \k \oplus \p$. The objects that depend only on the choice of an isotropic subspace $W = W_0 \subseteqq V$ are $\tilde{G} = \Sp (W_0^\omega/W_0)$ and $\tilde{\g} = \operatorname{Lie}(\tilde{G})$. The objects that depend on the choice of both $\JJ$ and $W=W_0$ are $\tilde{K} = \tilde{G}(\tilde{\JJ}(W_0))$, $\tilde{\k} = \tilde{\k}(\JJ, W_0)$, $\tilde{\p} = \tilde{\p}(\JJ, W_0)$, and $b = b_{\JJ, W_0}$. The objects that depend on the choice of $b$, $\JJ$, and $W_0$ are $\tilde{I} = \tilde{I}(b) = \tilde{I}(b, \JJ, W_0)$,
\[\tilde{G}^\shortparallel = \tilde{G}(\tilde{I}) = \tilde{G}(\tilde{I}(b, \JJ, W_0)), \quad  \tilde{G}(\tilde{I}\tilde{\JJ}) = \tilde{G}(\tilde{I}(b)\tilde{\JJ}(W_0)),\quad \tilde{K}^\shortparallel = \tilde{G}(\tilde{I}(b), \tilde{\JJ}(W_0)),\] \begin{equation*}
    \begin{array}{cccccc}
    \tilde{\g}^\shortparallel&=
    &\tilde{\g}^\shortparallel(\tilde{I}(b, \JJ, W_0))), & \tilde{\g}^\times
    &=& \tilde{\g}^\times(\tilde{I}(b, \JJ, W_0)),\\
    \tilde{\k}^\shortparallel&=& \tilde{\k}^\shortparallel(\tilde{I}(b), \tilde{\JJ}(W_0)),&
    \tilde{\k}^\times&=&\tilde{\k}^\times(\tilde{I}(b), \tilde{\JJ}(W_0)),\\ \tilde{\p}^\shortparallel&=&\tilde{\p}^\shortparallel(\tilde{I}(b), \tilde{\JJ}(W_0)),& \tilde{\p}^\times&=&\tilde{\p}^\times(\tilde{I}(b), \tilde{\JJ}(W_0)).
    \end{array}
\end{equation*}

Let $K^\shortparallel(b_\ast) := K \cap \Sp_{b_\ast}(V)$, for $b_\ast = \FF^\oplus \in \Lag_{\JJ, \oplus}^\C(\vec{n})$, $b_\ast = \FF \in \Lag_{\JJ}^\C(\vec{n})$, or $b_\ast = W' \in \Gr_{\JJ}(\vec{n})$.
\[ K^\shortparallel(b_\ast) \cong O(n_0) \times \U(n_+) \times \U(n_-)\] for any choice of $b_\ast$.

Now consider 
\begin{eqnarray*} \gamma_\JJ: \Gr(\vec{n}) &\to& \Gr_{\JJ}(\vec{n})\\
W' & \mapsto & \pi_{(\tilde{I}(W_\ast '), \tilde{\JJ}((W')_0), W_\ast')}(W')
\end{eqnarray*}
where $W_\ast'$ is any point in $(\cdot \cap \cdot^\omega)|_{\Gr_{\JJ}(\vec{n})}^{-1}((W')_0)$. By Lemma \ref{lem:fibers}.3, $\gamma_\JJ$ only depends on $\JJ$, and not on the choice of auxiliary basepoints $W_\ast'$'s used to define it.

Similarly, let
\begin{eqnarray*} \beta_\JJ: \Lag^\C(\vec{n}) &\to& \Lag_{\JJ}^\C(\vec{n})\\
\FF' & \mapsto & \pi_{(\tilde{I}(\FF_\ast'), \tilde{\JJ}(\Re \FF_0), \FF_\ast')}(\FF')
\end{eqnarray*}
where $\FF_\ast' \in (\Re \cdot 0)|_{\Lag_{\JJ}^\C(\vec{n})}^{-1}(\Re \FF'_0)$, and let
\begin{eqnarray*} \eta_\JJ: \Lag_{\oplus}^\C(\vec{n}) &\to& \Lag_{\JJ, \oplus}^\C(\vec{n})\\
(\FF^\oplus)' & \mapsto & \pi_{(\tilde{I}((\FF^\oplus)_\ast'), \tilde{\JJ}(\Re \FF'_0), (\FF^\oplus)_\ast')}((\FF^\oplus)')
\end{eqnarray*}
where $(\FF^\oplus)_\ast' \in (\alpha|_{\Lag_{\JJ, \oplus}^\C(\vec{n})})^{-1}(\alpha((\FF^\oplus)')$.

Note that the existence of the splitting $V = W_0 \oplus \JJ W_0\oplus W_-^\JJ$ simultaneously for every isotropic $W_0$ provides a preferred choice of identifications
\begin{equation} \label{eq:AdJ}
\begin{tikzcd} \tilde{G}(\tilde{\JJ}(W_0))\arrow[r, hook, "\JJ"] \arrow[d, "\cong"', "\Ad(k)"] & \Sp(W_0^\omega/W_0) \arrow[d,  "\cong"', "\Ad(k)"] \arrow[r, hook, "\JJ"] & \Sp_{W_0}(V) \arrow[r, hook] \arrow[d, "\Ad(k)", "\cong"'] & \Sp(V) \arrow[d, "\Ad(k)"]\\ \tilde{G}(\tilde{\JJ}(kW_0)) \arrow[r, hook, "\JJ"] & \Sp(kW_0^\omega/kW_0) \arrow[r, hook, "\JJ"] & \Sp_{kW_0}(V) \arrow[r, hook] & \Sp(V) \end{tikzcd} \end{equation}
for any $k \in K$. We have labelled with $\JJ$, maps that depend on the choice of $\JJ$. The three vertical arrows on the left are isomorphisms and the rightmost vertical arrow on the right is an automorphism.

Similarly, if $b_\ast = \FF^\oplus \in \Lag_{\JJ, \oplus}^\C(\vec{n})$, $b_\ast = \FF \in \Lag_{\JJ}^\C(\vec{n})$, or $b_\ast = W' \in \Gr_{\JJ}(\vec{n})$, $\JJ$ allows us to identify $\tilde{\p}$, $\tilde{\p}^\shortparallel$, $\tilde{\p}^\times$ as subspaces of $\g$, and $\Ad(k)$ induces isomorphisms

\begin{equation*}
    \begin{array}{ccc}
        \tilde{K}(\Re (b_\ast)_0) &\xrightarrow{\cong}& \tilde{K}(k \Re (b_\ast)_0),\\
        \tilde{\p}(\Re (b_\ast)_0) &\xrightarrow{\cong}& \tilde{\p}(k \Re (b_\ast)_0),\\
        \tilde{\p}^\shortparallel(\tilde{I}(b_\ast), \tilde{\JJ}((b_\ast)_0)) &\xrightarrow{\cong}& \tilde{\p}^\shortparallel(\tilde{I}(kb_\ast), \tilde{\JJ}(kb_\ast)_0)),\\
        \tilde{\p}^\times(\tilde{I}(b_\ast), \tilde{\JJ}((b_\ast)_0)) &\xrightarrow{\cong}& \tilde{\p}^\times(\tilde{I}(kb_\ast), \tilde{\JJ}((kb_\ast)_0)). 
    \end{array}
\end{equation*}
If $k \in K^\shortparallel$, these isomorphisms are automorphisms, and $K^\shortparallel$ acts on $\tilde{\p}(\Re (b_\ast)_0)$, $\tilde{\p}^\shortparallel(\tilde{I}(b_\ast))$, and $\tilde{\p}^\times(\tilde{I}(b_\ast))$ where $\GL(\Re (b_\ast)_0) \le K^\shortparallel$ acts trivially.

Now we are ready to reiterate the proof of Lemma \ref{lem:fibers}.
\begin{theorem}\label{thm:vectorbundle}
Suppose $\FF_\ast^\oplus \in \Lag_{\JJ, \oplus}^\C(\vec{n})$, $\FF_\ast \in \Lag_{\JJ}^\C(\vec{n})$, and $W_\ast \in \Gr_{\JJ}(\vec{n})$.
\begin{enumerate}
\item $\eta_\JJ$, $\beta_\JJ$, and $\gamma_\JJ$ are $K$-equivariant.
\item There are diffeomorphisms
    \begin{eqnarray*}
    K \times_{K^\shortparallel(\FF_\ast^\oplus)} \tilde{\p}(\Re (\FF_\ast)_0) &\cong&\Lag_\oplus^\C(\vec{n})\\
    {[(k, \tilde{p}_\ast)]} & \mapsto& k \exp \tilde{p}_\ast .\FF_\ast^\oplus\\
   K\times_{K^\shortparallel(\FF_\ast)} (\tilde{\p}(\Re (\FF_\ast)_0)/\tilde{\p}^\times(\tilde{I}(\FF_\ast), \tilde{\JJ}(\Re (\FF_\ast)_0))&\cong& \Lag^\C(\vec{n})\\
    {[(k, \tilde{p}_\ast^\shortparallel + \tilde{\p}^\times)]} &\mapsto& k \exp \tilde{p}_\ast^\shortparallel.\FF_\ast\\
    K\times_{K^\shortparallel(W_\ast)} (\tilde{\p}((W_\ast)_0)/\tilde{\p}^\shortparallel(\tilde{I}(W_\ast), \tilde{\JJ}((W_\ast)_0)) &\cong& \Gr(\vec{n})\\
    {[(k, \tilde{p}_\ast^\times + \tilde{\p}^\shortparallel)]} &\mapsto& k \exp \tilde{p}_\ast^\times.W_\ast.
\end{eqnarray*}
    \item The diagrams
    \begin{equation}
    \begin{tikzcd} K \times_{K^\shortparallel(\FF_\ast^\oplus)} \tilde{\p}(\FF_\ast^\oplus)\arrow[d, two heads] & \Lag_\oplus^\C(\vec{n}) \arrow[d, "\eta_\JJ"] \arrow[l, "\cong"'] \\ K/K^\shortparallel(\FF_\ast^\oplus) \arrow[r, "\cong"]& \Lag_{\JJ, \oplus}^\C(\vec{n})
    \end{tikzcd}
    \end{equation}
    
    \begin{equation}
    \begin{tikzcd} K\times_{K^\shortparallel(\FF_\ast)} (\tilde{\p}/\tilde{\p}^\times)(\FF_\ast)  \arrow[d, two heads] & \Lag^\C(\vec{n}) \arrow[d, "\beta_\JJ"] \arrow[l, "\cong"'] \\ K/K^\shortparallel(\FF_\ast) \arrow[r, "\cong"]& \Lag_\JJ^\C(\vec{n})
    \end{tikzcd}
    \end{equation}
    
    \begin{equation}
    \begin{tikzcd} K \times_{K^\shortparallel(W_\ast)} (\tilde{\p}/\tilde{\p}^\shortparallel)(W_\ast) \arrow[d, two heads] & \Gr(\vec{n}) \arrow[l, "\cong"'] \arrow[d, "
\gamma_\JJ"] \\ K/K^\shortparallel(W_\ast) \arrow[r, "\cong"]& \Gr_\JJ(\vec{n})\end{tikzcd}\end{equation}
commute. The upper horizontal arrows are defined from the previous statement. 
\end{enumerate}
\end{theorem}

\begin{proof}

(1): Suppose $W' \in \Gr(\vec{n})$ and $W_\ast' \in (\cdot \cap \cdot^\omega)|_{\Gr_{\JJ}(\vec{n})}^{-1}((W')_0)$. By Lemma \ref{lem:fibers}, there exist $\tilde{k} \in \tilde{K}(\JJ, (W')_0)$ and $\tilde{p}^\times \in \tilde{\p}^\times(\tilde{I}(W_\ast'))$ such that $W' = \tilde{k}e^{\tilde{p}^\times}.W_\ast'$. Then for all $k \in K$
\begin{eqnarray*}
\gamma_\JJ(k.W') &=& \pi_{(\tilde{I}(k.W_\ast'), \tilde{\JJ}(k.(W')_0), k.W_\ast')}(k.W')\\
&=& \pi_{(\tilde{I}(k.W_\ast'), \tilde{\JJ}(k.(W')_0), k.W_\ast')}(k \tilde{k} e^{\tilde{p}^\times}.W_\ast')\\
&=& \pi_{(\tilde{I}(k.W_\ast'), \tilde{\JJ}(k.(W')_0), k.W_\ast')}(k \tilde{k}k^{-1} \cdot k e^{\tilde{p}^\times}k^{-1}.kW_\ast')\\
&=& k \tilde{k} k^{-1}. kW_\ast'\\
&=& k.\gamma_{\JJ}(W'),
\end{eqnarray*}
because $\gamma_\JJ(W') = \gamma_\JJ( \tilde{k} e^{\tilde{p}^\times}.W_\ast') = \tilde{k}.W_{\ast}'$.
So $\gamma_\JJ$ is $K$-equivariant. Similarly, $\eta_\JJ$, $\beta_\JJ$ are $K$-equivariant. (2): Suppose $(\FF^\oplus)' \in \Lag^\C(\vec{n})$ and $(\FF^\oplus_\ast)' \in \alpha|_{\Lag_{\JJ, \oplus}^\C(\vec{n})}^{-1}(\alpha((\FF^\oplus)'))$. By Lemma \ref{lem:fibers}.1, there exist $\tilde{k} \in \tilde{K}(\Re(\FF_\ast^\oplus)'_{0})$ and $\tilde{p} \in \tilde{\p}(\Re(\FF_\ast^\oplus)'_{\ge 0})$ such that $(\FF^\oplus)' = \tilde{k} e^{\tilde{p}} (\FF_\ast^\oplus)'$. Since $K$ acts transitively on $\Lag_{\JJ, \oplus}^\C(\vec{n})$, there exists a $k \in K$ such that
\[ (\FF^\oplus)' = \tilde{k} e^{\tilde{p}} k.\FF_\ast^\oplus = \tilde{k}k. e^{\Ad(k)^{-1} \tilde{p}}.\FF_\ast^\oplus.\]
So the map $K \times \tilde{\p}(\Re (\FF_\ast^\oplus)_{\ge 0}) \to \Lag_{\oplus}^\C(\vec{n})$, $(k, \tilde{p}_\ast) \mapsto k e^{\tilde{p}_\ast}.\FF_\ast^\oplus$ is surjective.

If $ke^{\tilde{p}_\ast}.\FF_\ast^\oplus = k' e^{\tilde{p}_\ast'}.\FF_\ast^\oplus$, taking $\eta_\JJ$ on both sides, we obtain $k^{-1}k' \in \Sp_{\FF_\ast^\oplus}(V)$. So there exists a $k^\shortparallel\in K^\shortparallel(\FF_\ast^\oplus)$ such that $k
' = k k^\shortparallel$. Putting this back in, we get
\[ e^{\tilde{p}_\ast}.\FF_\ast^\oplus = k^{-1} k' e^{\tilde{p}_\ast'}.\FF_\ast^\oplus = e^{\Ad(k^\shortparallel)^{-1}\tilde{p}_\ast'}.\FF_\ast^\oplus.\]
By Lemma \ref{lem:fibers}.1, $\tilde{p_\ast} = \Ad(k^\shortparallel)^{-1}\tilde{p}_\ast'$. So the map $K \times \tilde{\p}(\Re (\FF_\ast^\oplus)_{\ge 0}) \to \Lag_{\oplus}^\C(\vec{n})$ descends to a diffeomorphism $K \times_{K^\shortparallel(\FF_\ast^\oplus)} \tilde{\p}(\Re (\FF_\ast^\oplus)_{\ge 0}) \to \Lag_{\oplus}^\C(\vec{n})$. Similarly, if $W' \in \Gr(\vec{n})$, choose $W_\ast' \in  (\cdot \cap \cdot^\omega|_{\Gr_\JJ(\vec{n})})^{-1}(W_0')$. By Lemma \ref{lem:fibers}.1, there exist $
\tilde{k} \in \tilde{K}(W_0')$ and $\tilde{p} = \tilde{p}^\times + \tilde{p}^\shortparallel \in \tilde{\p}(W_0')$ such that $W' = \tilde{k} e^{\tilde{p}^\times} e^{\tilde{p}^\shortparallel}.W_\ast' = \tilde{k} e^{\tilde{p}^\times}.W_\ast'$. Since $K$ acts transitively on $\Gr_\JJ(\vec{n})$, there exists a $k \in K$ such that $W_\ast' = k.W_\ast$. So
\[ W' = \tilde{k} e^{\tilde{p}^\times}k.W_\ast' = \tilde{k}k. e^{\Ad(k)^{-1}\tilde{p}^\times}.W_\ast.\]
So map  $K \times \tilde{\p}((W_\ast))_0\to \Gr(\vec{n})$, $(\tilde{k}, \tilde{p}_\ast) \mapsto ke^{\tilde{p}_\ast^\times} e^{\tilde{p}_\ast^\shortparallel}.W_\ast$ is surjective. Proceeding as before, it descends to a diffeomorphism $K \times_{K^\shortparallel(W_\ast)} (\tilde{\p}/\tilde{\p}^\shortparallel)(W_\ast)\xrightarrow{\cong} \Gr(\vec{n})$. Similarly, we obtain the diffeomorphism for $K \times_{K^\shortparallel(\FF_\ast)} (\tilde{\p}/\tilde{\p}^\shortparallel)(\FF_\ast)\xrightarrow{\cong} \Gr(\vec{n})$. (3): This follows from the previous statement and the construction of $\eta_\JJ$, $\beta_\JJ$, and $\gamma_\JJ$.
\end{proof}

\begin{remark}
    For $\Lag^\C(\vec{n}; V)$ compare with Theorem 2 of \cite{Takeuchi} in the $n_0 = 0$ case, and the holomorphic version in Theorem 11.8.4 of \cite{Wolf1} and the third statement of the Orbit fibration theorem of \cite{Wolf}.
\end{remark}

These maps fit into the commuting diagram
\begin{equation} \begin{tikzcd} \Lag^\C(\vec{n}) \arrow[d, two heads, "\beta_\JJ"'] \arrow[r, hook, "\iota_\JJ", shift left] & \Lag^\C_{\oplus}(\vec{n}) \arrow[r, two heads, "\Re\cdot_{\ge 0}"] \arrow[l, two heads, "\varphi", shift left] \arrow[d, two heads, "\eta_\JJ"] & \Gr(\vec{n}) \arrow[d, two heads, "\gamma_{\JJ}"]  \\ \Lag^\C_\JJ(\vec{n}) \arrow[ddr, "\Re \cdot_0|_{\Lag_{\JJ}^\C(\vec{n})}"'] & \Lag^\C_{\oplus, \JJ}(\vec{n}) \arrow[l, "\varphi", "\cong"'] \arrow[r, "\Re \cdot_{\ge 0}"', "\cong"] \arrow{dd}[description]{\alpha|_{\Lag_{\oplus, \JJ}^\C(\vec{n})}} & \Gr_{\JJ}(\vec{n})\arrow[ddl, "\cdot \cap \cdot^\omega|_{\Gr_\JJ(\vec{n})}"]\\ &&\\
& \Gr(n_0, 0, n-n_0) & \end{tikzcd} \end{equation}
where every arrow is $K$-equivariant.

Since Theorem \ref{thm:vectorbundle} shows $\eta_\JJ$, $\beta_\JJ$, and $\gamma_\JJ$ are realized as fiber bundles with contractible fibers, we can conclude the following from \cite{Dold}:
\begin{corollary}\label{cor:homotopytype}~
For any $\vec{n}$, $\Gr(\vec{n})$, $\Lag^\C(\vec{n})$, and $\Lag^\C_{\oplus}(\vec{n})$ are homotopy equivalent to the compact left coset space $U(n) / (O(n_0) \times U(n_+) \times U(n_-))$.
\end{corollary}

\begin{remark}[Oriented Grassmannians]
From the long exact sequence of homotopy groups, we can compute
\begin{equation} \label{eq:fundamentalgroups} \pi_1(\Gr(\vec{n})) \cong \begin{cases} \{1\} & \text{ if } n_0 = 0\\
\Z/2\Z & \text{ if } 0 < n_0 < n\\
\Z & \text{ if } n_0 = n.\end{cases} \end{equation}

So if $n_0= 0$, the oriented Grassmannian $\widetilde{Gr}(\vec{n})$ is a disconnected double cover of $\Gr(\vec{n})$. If $n_0 > 0$, take a $W \in \Gr(\vec{n})$, and a Darboux basis 
\[ \{\ee^0_1, \ee^0_2, \cdots, \ee^0_{n_0}, \ee^+, \ee^-, \ff^0_{1}, \ff^0_{2}, \cdots, \ff^0_{n_0}, \ff^+, \ff^-\}\]
associated to $W$ by Theorem \ref{thm:basisR}. It provides an orientation on $W$, and the Darboux basis
\[ \{-\ee^0_1, \ee^0_2, \cdots, \ee^0_{n_0}, \ee^+, \ee^-, -\ff^0_{1}, \ff^0_{2}, \cdots, \ff^0_{n_0}, \ff^+, \ff^-\}\]
provides an opposite orientation on $W$. So $\Sp(V)$ acts transitively on $\widetilde{\Gr}(\vec{n})$ and $\widetilde{\Gr}(\vec{n})$ is a connected double cover of $\Gr(\vec{n})$.
\end{remark}

We observe from (\ref{eq:vectorspace}) and Theorem \ref{thm:vectorbundle} that 
\begin{equation*}
\begin{array}{ccccccl}
\gamma_\JJ^{-1} (W') &\cong& \varphi^{-1}(\FF) &\cong& \tilde{\p}^\times &\quad& W' \in \Gr_{\JJ}(\vec{n}), \FF \in \Lag^\C(\vec{n})\\
\beta_\JJ^{-1}(\FF) &\cong& \Re \cdot_{\ge 0}^{-1}(W') &\cong& \tilde{\p}^\shortparallel &\quad& W' \in \Gr(\vec{n}), \FF \in \Lag_\JJ^\C(\vec{n}).
\end{array}
\end{equation*}
The invariant complex structures on the fibers of $\gamma_\JJ$, $\beta_\JJ$, $\eta_\JJ$ can be obtained by the same argument as Remark \ref{rem:complexstructures} and the end of Section \ref{subsec:symmetricspace}. Since $\frac{1}{2} \ad(\tilde{\JJ})$ is a linear complex structure on $\tilde{\p}^\shortparallel$, $\tilde{\p}^\times$, $\tilde{\p}$ (\ref{eq:adJ}) it extends to an invariant complex structure on the fibers by Proposition IV.3.4 of \cite{Helgason}

We observe from the proof of Theorem \ref{thm:vectorbundle} there is a correction term, depending on the choice of basepoint $\FF_\ast^\oplus \in \Lag_{\JJ, \oplus}^\C(\vec{n})$, and arising from the Baker-Campbell-Hausdorff noncommutativity of the exponentials:
\begin{equation} \label{eq:correction}
\begin{array}{ccccc}
 \Lag_\oplus^\C(\vec{n})&\xrightarrow{\cong}& K \times_{K^\shortparallel(\FF_\ast^\oplus)} \tilde{\p}(\Re (\FF_\ast^\oplus)_0) &\xrightarrow{\cong}& \Lag_\oplus^\C(\vec{n})\\
ke^{\tilde{p}^\shortparallel_\ast}e^{\tilde{p}^\times_\ast}.\FF_\ast^\oplus &\mapsto& [(k, \tilde{p}_\ast)] &\mapsto& ke^{\tilde{p}^\times_\ast}e^{\tilde{p}^\shortparallel_\ast}.\FF_\ast^\oplus
\end{array}\end{equation}
We will let $\chi_{\JJ, \FF_\ast^\oplus}(ke^{\tilde{p}^\shortparallel_\ast}e^{\tilde{p}^\times_\ast}.\FF_\ast^\oplus) := ke^{\tilde{p}^\times_\ast}e^{\tilde{p}^\shortparallel_\ast}.\FF_\ast^\oplus$.
$\chi_{\JJ, \FF_\ast^\oplus}$ is the identity map when $n_+n_- = 0$ or when restricted to $\Lag_{\JJ, \oplus}^\C(\vec{n})$.

\begin{corollary}\label{cor:fiberproduct}
The commuting diagram
\begin{equation}
\begin{tikzcd}
    \Lag_{\oplus}^\C(\vec{n}) \arrow[rr, "\varphi"] \arrow[dd, "\Re \cdot_{\ge 0} \circ \chi_{\JJ, \FF_\ast^\oplus}"'] && \Lag^\C(\vec{n}) \arrow[d, "\beta_\JJ"] \\ && \Lag_\JJ^\C(\vec{n}) \arrow[d, "\iota_\JJ"]\\ \Gr(\vec{n}) \arrow[r, "\gamma_\JJ"'] & \Gr_\JJ(\vec{n}) \arrow[r, "\psi_\JJ"']& \Lag_{\JJ, \oplus}(\vec{n})
\end{tikzcd}
\end{equation}
is a pullback diagram.
\end{corollary}

\begin{proof}

\begin{equation*}
\begin{tikzcd}
    K\times_{K^\shortparallel(\FF_\ast^\oplus)} \tilde{\p}(\FF_\ast^\oplus) \arrow[r] \arrow[d] & K\times_{K^\shortparallel(\varphi(\FF_\ast^\oplus))} (\tilde{\p}/\tilde{\p}^\times)(\varphi(\FF_\ast^\oplus)) \arrow[d]\\K\times_{K^\shortparallel(\Re(\FF_\ast^\oplus)_{\ge 0})} (\tilde{\p}/\tilde{\p}^\shortparallel)(\Re(\FF_\ast^\oplus)_{\ge 0}) \arrow[r] & K/K^\shortparallel(\FF_\ast^\oplus)
\end{tikzcd}
\end{equation*}
is a pullback diagram. This can be read off from the elementwise description of the maps
\begin{equation*}
\begin{tikzcd}
    {[(k, \tilde{p}_\ast^\shortparallel + \tilde{p}_\ast^\times)]} \arrow[d, maps to] \arrow[r, maps to] & {[(k, \tilde{p}_\ast^\shortparallel + \tilde{\p}^\times)] }\arrow[d, maps to]\\
    {[(k, \tilde{p}_\ast^\times + \tilde{\p}^\shortparallel)]} \arrow[r, maps to] & {[k]}.
\end{tikzcd}
\end{equation*}
The pullback diagram is identifed with the given diagram with the diffeomorphisms in Theorem \ref{thm:vectorbundle} and (\ref{eq:correction}).
\end{proof}

\appendix
\counterwithin{figure}{section}
\section{Appendix}
\label{sec:matrix}

\subsection{Darboux bases and Vectorial Cartan subalgebras}
\label{subsec:cartan}

\begin{proposition}
Suppose $(V, \omega)$ is a real, finite dimensional symplectic vector space, $G = \Sp(V)$ is the group of linear symplectomorphisms of $V$, and $\g$ is the Lie algebra of $G$. Given a Darboux basis $\{\ee_1, \cdots, \ee_n, \ff_1, \cdots, \ff_n\}$ of $V$ there exists a unique Cartan subalgebra $\a_{\{\ee, \ff\}}$ (\emph{vectorial} Cartan subalgebra in the sense of \cite{Sugiura}) of $\g \subseteq \End(V)$ such that the weight space decomposition
\[ V \cong \bigoplus_{\alpha \in \a_{\{\ee, \ff\}}^\vee} V_{\alpha}\]
agrees with
\[ V \cong \R \ee_1 \oplus \cdots \oplus \R \ee_n \oplus \R \ff_1 \oplus \cdots  \oplus \R \ff_n.\]

Conversely, given a Cartan subalgebra $\a$ of $\g$ contained in $\p$ for some Cartan decomposition $\g = \k \oplus \p$ of $\g$, there exists a Darboux basis $\{\ee_1, \cdots, \ee_n, \ff_1, \cdots, \ff_n\}$ of $V$ such that the weight space decomposition
\[ V \cong \bigoplus_{\alpha \in \a^\vee} V_{\alpha}\]
agrees with
\[ V \cong \R \ee_1 \oplus \cdots \oplus \R \ee_n \oplus \R \ff_1 \oplus \cdots  \oplus \R \ff_n.\]
This Darboux basis is unique up to the choice and labelling of vectors within the weight spaces.
\end{proposition}

\begin{proof}
Suppose we are given a Darboux basis $\{\ee, \ff\}$ of $V$. The linear extension of the map assigning $\ee_j \mapsto \ff_j$, $\ff_\jj \mapsto -\ee_j$ for all $1 \le j \le n$ is an $\omega$-compatible linear complex structure $\JJ$ on $V$. It defines an inner product $\omega(\cdot, \JJ \cdot)$ on $V$, an involution $\cdot^t$ on $\End(V)$ by $\omega(g\cdot, \JJ \cdot) = \omega(\cdot, \JJ g^t \cdot)$, and a Cartan decomposition of $\g \subseteq \End(V)$ by $\k = \g \cap \End(V)^{-t}$ and $\p = \g  \cap \End{V}^{t}$.

Let $\{\ee^t, \ff^t\}$ be the dual basis of $\{\ee, \ff\}$ in $V^\vee := \Hom_\R(V; \R)$, and consider the abelian Lie subalgebra
\[ \a_{\{\ee, \ff\}} := \span_\R\{ \ee_j \otimes_\R \ee_j^t - \ff_j \otimes_\R \ff_j^t: 1 \le j \le n\} \subseteq \g.\]
$\a_{\{\ee, \ff\}}$ is a linear subspace of $\p$ and is identified with the diagonal matrices in the image of $\g$ in $\Mat_{2n \times 2n}(\R)$. Computing commutators in $\Mat_{2n \times 2n}(\R)$, we can verify $\a_{\{\ee, \ff\}}$ is a Cartan subalgebra of $\g$.

The weights $\{\alpha_j: 1 \le j \le n\}$ defined by
\[ \alpha_j (\ee_k \otimes_\R \ee_k^t - \ff_k \otimes_\R \ff_k^t):= \begin{cases} 1 \text{ if } j = k \\ 0 \text{ otherwise.} \end{cases} \]
span $\a_{\{\ee, \ff\}}^\vee$, and the weight space of $\alpha_j$ is $\R\ee_j$ and the weight space of $-\alpha_j$ is $\R \ff_j$ for all $1 \le j \le n$.

From Corollary 2, Section 1 of \cite{Sugiura}, any two Cartan decompositions are conjugate under the action of $\Sp(V)$, and any two maximal abelian subalgebras $\a$ of $\g$ contained in a $\p$ for some Cartan decomposition $\g = \k \oplus \p$ are conjugate under the action of the maximal compact subgroup of $\Sp(V)$ corresponding to the Cartan decomposition $\g = \k \oplus \p$.

So for Cartan subalgebra $\a'$ of $\g$ contained in some $\p'$ for some Cartan decomposition $\g = \k' \oplus \p'$, there exists a $g \in \Sp(V)$ such that $\a' = g \a_{\{\ee, \ff\}} g^{-1}$. The weight space decomposition of $\a' = g \a_{\{\ee, \ff\}} g^{-1}$ is
\[ V = \bigoplus_{\alpha \in \a_{\{\ee, \ff\}}^\vee} gV_{\alpha}.\]
Therefore, if $g$ stabilizes the decomposition
\[ V = \R \ee_1 \oplus \cdots \oplus \R \ee_n \oplus \R \ff_1 \oplus \cdots  \oplus \R \ff_n,\]
$\a' = g \a_{\{\ee, \ff\}} g^{-1} = \a_{\{\ee, \ff\}}$. So $\a'$ is uniquely determined by the weight space decomposition. This argument also shows that a choice of Cartan subalgebra $\a \subseteq \g$ contained in $\p$ for some Cartan decomposition uniquely determines a weight space decomposition of $V$ into real one-dimensional subspaces. Then the choice of Darboux basis is determined up to the choice of vectors within each weight space, and the choices differ only by rescaling and relabelling of the vectors.
\end{proof}

\subsection{The generalized M\"{o}bius action}
\label{subsec:mobius}
In this section we fix a Darboux basis of $V$ and assume the identification $V \cong \R^{2n}$, $V^\C \cong \C^{2n}$. 

The symplectic form takes the form 
\[ \omega^\C(\vv, \ww) = \ww^t \begin{pmatrix} 0& -1_{n} \\ 1_{n} & 0 \end{pmatrix} \vv \quad \vv, \ww \in V^\C.\]

We will denote the condition $\FF = \span_\C\{\vv_1, \cdots, \vv_n\}$
as
\begin{equation} \label{eq:bracketspan} \FF = \left[ \begin{pmatrix} | &  & | \\ \vv_1 & \cdots & \vv_n \\ | &  & | \end{pmatrix} \right]=  \left[ \begin{pmatrix} Q \\ P \end{pmatrix} \right] \quad Q, P \in \Mat_{n \times n}(\C)\end{equation}

$\FF$ is a complex Lagrangian subspace if and only if it is $n$ dimensional and $Q^t P = P^t Q$, and is of type $\vec{n}$ if and only if, in addition, the $n \times n$ hermitian matrix $-i (\overline{P}^t Q - \overline{Q}^t P)$ has signature $\vec{n}$.

If $X$ is an invertible $n \times n$ complex matrix, then
\[ \left[ \begin{pmatrix} Q \\ P \end{pmatrix} X \right] = \left[ \begin{pmatrix} Q\\ P \end{pmatrix} \right].\]

\begin{example}[The complex projective line]\label{ex:R2}
When $n = 1$, every complex one dimensional subspace of $\C^2$ is complex Lagrangian so $\Lag^\C(\R^2)$ can be identified with the complex projective line $\mathbb{P}^1(\C)$ via
\[ \left[ \begin{pmatrix} q \\ p \end{pmatrix}\right] \mapsto [q:p].\]

$\Sp(2; \C) = \SL(2; \C)$ and $\Sp(2; \R) = \SL(2; \R)$, so the generalized M\"{o}bius action is then the honest M\"{o}bius action given by linear fractional transformations.

\[\kappa|_{[q:p]} = \begin{pmatrix} \overline{q} & \overline{p} \end{pmatrix} \begin{pmatrix} 0 & -1\\ 1 & 0 \end{pmatrix} \begin{pmatrix} q \\ p \end{pmatrix} = 2\Im(q \overline{p})\]

So $\Lag^\C(0, 1, 0)$, $\Lag^\C(1, 0, 0)$, $\Lag^\C(0, 0, 1)$ can, respectively, be identified with the upper hemisphere, equator, and lower hemisphere when the north and south poles are identified with $[\pm i: 1]$.
\end{example}

\begin{example}[Siegel upper half plane]
The \emph{Siegel upper half plane} is given by
\[ \{ Z \in \Mat_{n\times n}(\C)^t : \Im Z > 0\},\]
and can be identified with $\Lag^\C(0, n, 0)$ by
\[ Z \mapsto \left[ \begin{pmatrix} Z \\ 1_n \end{pmatrix} \right].\]
The generalized M\"{o}bius action of $\Sp(2n; \R)$ on the Siegel upper half plane defined by
\[ \begin{pmatrix} A & B \\ C & D \end{pmatrix} . Z := (AZ+B)(CZ+D)^{-1}\]
is then just a consequence of
\[ \left[ \begin{pmatrix} A & B \\ C & D \end{pmatrix} \begin{pmatrix} Z \\ 1_n\end{pmatrix}\right] = \left[ \begin{pmatrix} AZ + B\\ CZ+D\end{pmatrix}\right] = \left[ \begin{pmatrix} (AZ + B)(CZ+D)^{-1} \\ 1_n \end{pmatrix}\right].\]
\end{example}

\subsection{The Harish-Chandra embedding in matrix form}
\label{subsec:harishchandra}

In this section we fix a Darboux basis of $V$ and assume the identification $V \cong \R^{2n}$, $V^\C \cong \C^{2n}$. We choose an $\omega$-compatible complex structure
\[ \JJ = \begin{pmatrix} 0 & -1_n \\ 1_n & 0 \end{pmatrix}.\]
Then the subgroups defined in Section \ref{subsec:borel} are
\begin{eqnarray*}
     K &=& \left\{ \begin{pmatrix} \xx & -\yy\\ \yy & \xx \end{pmatrix} : \xx + i \yy \in U(n)\right\} \\
     K^\C &=& \left\{ \begin{pmatrix} \xx & - \yy \\ \yy & \xx \end{pmatrix} : \xx^t \xx + \yy^t \yy = 1_n, \xx^t \yy = \yy^t \xx \right\}\\
     G_u &=& \left\{ \begin{pmatrix} \xx & -\overline{\yy}\\ \yy & \overline{\xx} \end{pmatrix} : \xx + i \yy \in U(n)\right\}.
\end{eqnarray*}
Moreover,
\[\FF_{\pm \JJ} = \left[ \begin{pmatrix} \pm i 1_n \\ 1_n \end{pmatrix}\right]\]
and
\begin{eqnarray*}
     P^+ &:=& \left\{ \begin{pmatrix} 1_n -\zz & i\zz \\ i \zz & 1_n + \zz\end{pmatrix} : \zz \in \Mat_{n \times n}(\C)^t\right\}\\
     P^- &:=& \left\{ \begin{pmatrix} 1_n + \zz & i\zz \\ i \zz & 1_n - \zz\end{pmatrix} : \zz \in \Mat_{n \times n}(\C)^t\right\}.
\end{eqnarray*}

The Harish-Chandra embedding is defined by taking a $gK \in G/K$ to the unique element $\pp^+$ of the Lie algebra of $P^+$ such that
\[ gK^\C P^- = \exp(\pp^+) K^\C P^- \in G^\C/K^\C P^-.\]
Composing the Harish-Chandra embedding with the map $\pp^+ \mapsto \exp(\pp^+) K^\C P^-$ gives the Borel embedding.

In our case, the matrix form the Harish-Chandra embedding can be expressed as
\[ \begin{pmatrix} 1_n & \xx \\ 0 & 1_n \end{pmatrix} \begin{pmatrix} \yy & 0 \\ 0 & \yy^{-1}\end{pmatrix}K \mapsto \begin{pmatrix} -\zz & i\zz \\ i \zz & \zz \end{pmatrix} \]
where $\xx$ and $\yy$ are real symmetric $n \times n$ matrices, $\yy$ is strictly positive definite, and
\begin{equation} \label{eq:hc} \zz = \frac{1}{2} (1_n + i (\xx - i \yy^2))(1_n - i (\xx - i \yy^2))^{-1}.\end{equation}
(c.f. the Cayley transform $\zz \mapsto (1_n + i \zz)(1_n - i\zz)^{-1}$ mapping the Siegel upper half plane to the Siegel disk consisting of all $n\times n$ complex symmetric matrices $\zz$ such that $1_n - \overline{\zz}\zz$ is positive definite.)

(\ref{eq:hc}) can be computed by equating the coefficients
\[ \begin{pmatrix} 1_n & \xx \\ 0 & 1_n \end{pmatrix} \begin{pmatrix} \yy & 0 \\ 0 & \yy^{-1} \end{pmatrix} \left[ \begin{pmatrix} -i 1_n \\ 1_n \end{pmatrix}\right] = \left[ \begin{pmatrix} - i(\yy^2 + i \xx) \\ 1_n \end{pmatrix} \right]\]
\[ \begin{pmatrix} 1_n - \zz & i\zz \\ i\zz & 1+ \zz \end{pmatrix} \left[\begin{pmatrix} -i 1_n \\ 1_n \end{pmatrix} \right] = \left[ \begin{pmatrix} -i (1-2\zz) (1+2\zz)^{-1}\\1_n \end{pmatrix} \right]. \]

In particular, the image of the Harish-Chandra embedding is a bounded domain biholomorphic to 
\[ \{ \zz \in \Mat_{n\times n}(\C)^t : 4 \cdot 1_n - \zz \zz^\dagger > 0 \}\]  so the closure of $\Lag^\C(0, 0, n)$ in $\Lag^\C(\R^{2n})$ is homeomorphic to a closed $n(n+1)$-ball.

\subsection{Matrix description of Lie Algebra}
\label{sec:liealg}

In this section we fix a Darboux basis of $V$ and assume the identification $V \cong \R^{2n}$, $V^\C \cong \C^{2n}$. 

\subsubsection{Cartan decomposition and linear complex structure}
Let $\g$ be the lie algebra of $\Sp(2n; \R)$. The Cartan decomposition $\g = \k \oplus \p$ with respect to the involution $\theta(\cdot):= -\cdot^t$ is given by
\begin{eqnarray*}
\g &=& \left\{ \begin{pmatrix} \aa & \bb \\ \cc & -\aa^t \end{pmatrix} : \bb, \cc \in \Mat_{n\times n}(\R)^t, \aa \in \Mat_{n\times n}(\R) \right\}\\
\k &=& \left\{\begin{pmatrix} \aa & -\bb \\ \bb & \aa \end{pmatrix} : \aa \in \Mat_{n\times n}(\R)^{-t}, \bb \in \Mat_{n\times n}(\R)^t \right\}\\
&=& \g \cap \Mat_{2n \times 2n}(\R)^{-t}\\
\p &=& \left\{\begin{pmatrix} \aa & \bb \\ \bb & -\aa \end{pmatrix} : \aa , \bb \in \Mat_{n\times n}(\R)^{t} \right\}\\
&=&  \g \cap \Mat_{2n \times 2n}(\R)^{t}.
\end{eqnarray*}

The element
\begin{equation} \label{eq:Jcenter} \JJ:=\begin{pmatrix} 0 & -1_n \\ 1_n & 0 \end{pmatrix}\end{equation}
in the center of $\k$ is such that
\begin{equation} \label{eq:complexstructure} \frac{1}{2}\ad(\JJ) \left(\begin{pmatrix} \aa & \bb \\ \bb & -\aa\end{pmatrix}\right) = \JJ \left(\begin{pmatrix} \aa & \bb \\ \bb & -\aa\end{pmatrix}\right) = \begin{pmatrix} -\bb & \aa \\ \aa & \bb\end{pmatrix}.\end{equation}
So $(\ad(\JJ)|_\p)^2 = -1_{\p}$, and is a linear complex structure on $\p$, with $\pm i$-eigenspaces
\begin{eqnarray*}
\p^+ &:=& \left\{ \begin{pmatrix} -\zz & i \zz \\ i \zz & \zz\end{pmatrix} : \zz \in \Mat_{n\times n}(\C)^t\right\} \subseteq \p^\C\\
\p^- &:=& \left\{ \begin{pmatrix} \zz & i \zz \\ i \zz & -\zz\end{pmatrix} : \zz \in \Mat_{n\times n}(\C)^t\right\} \subseteq \p^\C.
\end{eqnarray*}
Unlike $\p$, $\p^+$ and $\p^-$ are abelian Lie subalgebras of $\g^\C$.

\subsubsection{Root space decompositions} \label{subsubsec:rootspacedecompositions}

Maximal abelian subalgebras $\a, \h$ of $\g$ contained in $\p$ and $\t \subseteq \k$ are given by
\begin{eqnarray*}
\a &=& \left\{\begin{pmatrix} \dd & 0 \\ 0 & -\dd \end{pmatrix}: \dd = \diag(d_1, \cdots, d_n), d_j \in \R, 1 \le j \le n\right\}\\
\h &=& \left\{ \begin{pmatrix} 0 & \dd \\ \dd & 0 \end{pmatrix}: \dd = \diag(d_1, \cdots, d_n), d_j \in \R, 1 \le j \le n\right\}\\
\t &=& \left\{\begin{pmatrix} 0 & -\dd \\ \dd & 0 \end{pmatrix}: \dd = \diag(d_1, \cdots, d_n), d_j \in \R, 1 \le j \le n\right\}
\end{eqnarray*}

The weights of the fundamental representation are
\begin{eqnarray*}
\mu_j \left(\begin{pmatrix} \dd & 0 \\ 0 & -\dd \end{pmatrix}\right) &:=& d_j  \quad 1 \le j \le n\\
\delta_j \left(\begin{pmatrix} 0 & \dd \\ \dd & 0 \end{pmatrix} \right) &:=& d_j \quad 1 \le j \le n\\
\varepsilon_j \left(\begin{pmatrix} 0 & -\dd \\ \dd & 0 \end{pmatrix}\right) &:=& d_j \quad 1 \le j \le n.
\end{eqnarray*}

The roots are, with uncorrelated choice of signs, elements of 
\begin{eqnarray*}
    \Phi(\g; \a) &:=& \{ \pm \mu_j \pm \mu_\ell : 1 \le j, \ell \le n\}\setminus \{0\}\\
    \Phi(\g; \h) &:=& \{ \pm \delta_j \pm \delta_\ell : 1 \le j, \ell \le n\}\setminus \{0\}\\
    \Phi(\g^\C; \t^\C) &:=& \{ \pm i \varepsilon_j \pm i\varepsilon_{\ell} : 1 \le j, \ell \le n\} \setminus \{0\}.
\end{eqnarray*}

The positive and simple roots are chosen as
\begin{eqnarray*}
    \Phi^+ (\g; \a) &:=& \{ \mu_j - \mu_\ell : 1 \le j < \ell \le n\} \sqcup \{ \mu_j + \mu_\ell : 1 \le j \le \ell \le n\}\\
    \Delta(\g; \a) &:=& \{\mu_1 - \mu_2, \cdots, \mu_{n-1} - \mu_n, 2\mu_n\}\\
    \Phi^+ (\g; \h) &:=& \{ \delta_j - \delta_\ell : 1 \le j < \ell \le n\} \sqcup \{ \delta_j + \delta_\ell : 1 \le j \le \ell \le n\}\\
    \Delta(\g; \h) &:=& \{\delta_1 - \delta_2, \cdots, \delta_{n-1} - \delta_n, 2\delta_n\}\\
    \Phi^+ (\g^\C; \t^\C) &:=& \{ i\varepsilon_j - i\varepsilon_\ell : 1 \le j < \ell \le n\} \sqcup \{ i\varepsilon_j + i\varepsilon_\ell : 1 \le j \le \ell \le n\}\\
    \Delta(\g^\C; \t^\C) &:=& \{i\varepsilon_1 - i\varepsilon_2, \cdots, i\varepsilon_{n-1} - i\varepsilon_n, 2i \varepsilon_n\}
\end{eqnarray*}

With respect to this choice, the fundamental representation has highest weights $\mu_1$, $\delta_1$, $\varepsilon_1$.

The Weyl group $\mathbb{S}_n \ltimes (\Z/2\Z)^n$ acts on $\Phi(\g; \a)$ by permuting the indices and flipping signs
\begin{eqnarray*}
\sigma.(\pm \mu_j \pm \mu_\ell) &=&\pm \mu_{\sigma(j)} \pm \mu_{\sigma(\ell)} \quad \sigma \in \Sigma_n\\
(\epsilon_1, \cdots, \epsilon_n).(\pm \mu_j \pm \mu_\ell) &=&\pm \epsilon_j \mu_j \pm \epsilon_\ell \mu_{\ell} \quad \epsilon_j \in \{ \pm 1\}, 1 \le j \le n
\end{eqnarray*}
and similarly on $\Phi(\g; \h)$ and $\Phi(\g^\C; \t^\C)$.

The $n\times n$ matrix with $1$ on the entry in the $j$th row and $\ell$th column, and rest of entries zero is denoted as $E_{j\ell}$.

The root spaces for $\Phi(\g; \a)$ are
\begin{eqnarray*}
    \g_{\mu_j - \mu_\ell} &=&  \span_\R \left\{ \begin{pmatrix} E_{j\ell} & 0 \\ 0 & -E_{\ell j}\end{pmatrix} \right\} \quad j \neq \ell\\
    \g_{\mu_j + \mu_\ell} &=& \span_\R \left\{ \begin{pmatrix} 0 & E_{j \ell} + E_{\ell j} \\ 0 & 0 \end{pmatrix}\right\}\\
    \g_{-\mu_j - \mu_\ell} &=& \span_\R \left\{ \begin{pmatrix} 0 & 0 \\ E_{j\ell} + E_{\ell j} & 0 \end{pmatrix} \right\}
\end{eqnarray*}
where $1 \le j, \ell \le n$.

The root spaces for $\Phi(\g; \h)$ are
\begin{eqnarray*}
    \g_{\delta_j - \delta_\ell} &:=& \span_\R \left\{ \begin{pmatrix} E_{j\ell} - E_{\ell j} & E_{j \ell} + E_{\ell j} \\ E_{j\ell} + E_{\ell j} & E_{j\ell} - E_{\ell j}\end{pmatrix} \right\}   \quad j \neq \ell \\
    \g_{\delta_j + \delta_\ell} &:=& \span_\R \left\{ \begin{pmatrix} -E_{j\ell} - E_{\ell j} & E_{j \ell}+E_{\ell j} \\ -E_{j\ell} - E_{\ell j} & E_{j\ell} + E_{\ell j}\end{pmatrix} \right\}\\
    \g_{-\delta_j -\delta_\ell} &:=& \span_\R \left\{ \begin{pmatrix} -E_{j\ell} - E_{\ell j} & -E_{j \ell} -E_{\ell j} \\ E_{j\ell} + E_{\ell j} & E_{j\ell} + E_{\ell j}\end{pmatrix} \right\}
\end{eqnarray*}
where $1 \le j, \ell \le n$.

The root spaces for $\Phi(\g^\C; \t^\C)$ are
\begin{eqnarray*}
\g_{i\varepsilon_j - i\varepsilon_\ell}^\C &=&  \span_\C \left\{ \begin{pmatrix} E_{j\ell} - E_{\ell j} & i (E_{j \ell} + E_{\ell j}) \\ -i (E_{j\ell} + E_{\ell j}) & E_{j\ell} - E_{\ell j}\end{pmatrix} \right\} \subseteq \k^\C  \quad j \neq \ell \\
\g_{i\varepsilon_j + i\varepsilon_\ell}^\C &=& \span_\C \left\{ \begin{pmatrix} -E_{j \ell} - E_{\ell j} & i (E_{j\ell} + E_{\ell j}) \\ i (E_{j \ell} + E_{\ell j}) & E_{j \ell} + E_{\ell j} \end{pmatrix} \right\}\subseteq \p^\C\\
\g_{-i\varepsilon_j - i \varepsilon_\ell}^\C &=& \span_\C \left\{\begin{pmatrix} E_{j \ell} + E_{\ell j} & i (E_{j\ell} + E_{\ell j}) \\ i (E_{j \ell} + E_{\ell j}) & -E_{j \ell} - E_{\ell j} \end{pmatrix} \right\}\subseteq \p^\C\\
\end{eqnarray*}
where $1 \le j, \ell \le n$.

$\Phi(\g^\C; \t^\C)$ partitions into the compact roots $\{\pm(i\varepsilon_j - i \varepsilon_\ell): 1 \le j < \ell \le n\}$ and the noncompact roots $\{\pm(i \varepsilon_j + i \varepsilon_\ell) : 1 \le j, \ell \le n\}$.

\subsubsection{Choice of root vectors}

The Killing form on $\g$ is
\[ \left\langle \begin{pmatrix} \aa & \bb \\ \cc & -\aa^t \end{pmatrix}, \begin{pmatrix} \aa' & \bb' \\ \cc' & -(\aa')^t \end{pmatrix} \right\rangle = 2(n+1)\Tr(\aa\aa' + \aa' \aa + \bb \cc' + \cc' \bb).\]

The coroots of $\pm i \varepsilon_j \pm i \varepsilon_\ell \in \Phi(\g^\C; \t^\C)$ are, with correlated choice of signs,
\[ h_{\pm i (\varepsilon_j - i \varepsilon_\ell)} := \begin{pmatrix} 0 & \pm i E_{jj} \mp i E_{\ell \ell} \\ \mp i E_{jj} \pm i E_{jj}& 0 \end{pmatrix}\]
\[ h_{\pm i (\varepsilon_j + i \varepsilon_\ell)} := \begin{pmatrix} 0 & \pm i E_{jj} \pm i E_{\ell \ell} \\ \mp i E_{jj} \mp i E_{jj}& 0 \end{pmatrix}\]

Choose root vectors, for $1 \le j, \ell \le n$,
\begin{eqnarray*}
    e_{i\varepsilon_j - i \varepsilon_\ell} &:=& \frac{1}{2} \begin{pmatrix} E_{j\ell} - E_{\ell j} & i (E_{j \ell} + E_{\ell j}) \\ -i (E_{j\ell} + E_{\ell j}) & E_{j\ell} - E_{\ell j}\end{pmatrix} \in \g_{i \varepsilon_j - i \varepsilon_\ell}^\C \quad j \neq \ell\\
    e_{ i\varepsilon_j + i\varepsilon_\ell} &:=& \frac{1}{2}\begin{pmatrix} -E_{j \ell} - E_{\ell j} & i (E_{j\ell} + E_{\ell j}) \\ i (E_{j \ell} + E_{\ell j}) & E_{j \ell} + E_{\ell j} \end{pmatrix} \in \g_{ i\varepsilon_j + i\varepsilon_\ell}^\C\\
     e_{ -i\varepsilon_j - i\varepsilon_\ell} &:=& \frac{1}{2}\begin{pmatrix} -E_{j \ell} - E_{\ell j} & -i (E_{j\ell} + E_{\ell j}) \\ -i (E_{j \ell} + E_{\ell j}) & E_{j \ell} + E_{\ell j} \end{pmatrix} \in \g_{-i\varepsilon_j - i\varepsilon_\ell}^\C.
\end{eqnarray*}

Then $[e_{\varphi}, e_{-\varphi}] = h_\varphi$ for all $\varphi \in \Phi(\g^\C; \t^\C)$.

For each positive noncompact root $\varphi = i\varepsilon_j + i \varepsilon_\ell$, let
\begin{eqnarray*}
    x_\varphi &:=& e_{\varphi} + e_{-\varphi} = \begin{pmatrix} -E_{j\ell} - E_{\ell j} & 0 \\ 0 & E_{j\ell} + E_{\ell j} \end{pmatrix}\\
    y_{\varphi} &:=& i (e_{\varphi} - e_{-\varphi}) = \begin{pmatrix} 0 & - E_{j\ell} - E_{\ell j} \\ - E_{j\ell} - E_{\ell j} & 0 \end{pmatrix}.
\end{eqnarray*}
Then $\{ x_\varphi, y_\varphi\}_{\{i\varepsilon_j + i \varepsilon_\ell : 1 \le j, \ell \le n\}}$ is a basis of $\p$ and
\[ \frac{1}{2}\ad(\JJ) x_\varphi = y_\varphi, \quad \frac{1}{2}\ad(\JJ) y_\varphi = -x_{\varphi}, \quad [x_\varphi, y_\varphi] = -2i h_\varphi.\]

\subsubsection{Partial Cayley Transforms}
\label{subsec:partialcayley}

Two roots $\varphi, \psi \in \Phi(\g^\C; \t^\C)$ are \emph{strongly orthogonal} if neither $\varphi \pm \psi$ is a root. We choose a maximal strongly orthogonal collection of positive noncompact roots
\begin{equation} \label{eq:stronglyorthogonal} \Psi := \{ 2i \varepsilon_1, \cdots, 2i\varepsilon_n\}.\end{equation}

We choose
\begin{eqnarray} 
    \Gamma &:=& \{ 2i \varepsilon_1, \cdots, 2i\varepsilon_{n_0}\} \subseteq \Psi \label{eq:Gamma}\\
    \Sigma &:=& \{ 2i \varepsilon_{n_0 + 1}, \cdots, 2i \varepsilon_{n_0 + n_+}\} \subseteq \Psi, \label{eq:Sigma}
\end{eqnarray}
and suppress denoting the dependence of $n_0$, $n_-$, because if $|\Gamma| = |\Gamma'|$, for another subset $\Gamma' \subseteq \Psi$ then $\Gamma$ and $\Gamma'$ (respectively $\Sigma$ and $\Sigma'$) differ by a Weyl group transformation.

The partial Cayley transforms are 
\[ c_\Gamma:= \prod_{\gamma \in \Gamma} c_\gamma, \quad c_{2i \varepsilon_j} := \exp\left( \frac{i\pi}{4} y_{2i\varepsilon_j}\right). \]
So
\begin{equation}\label{eq:cayleymatrix1}
    c_\Gamma = \begin{pmatrix} \frac{1}{\sqrt{2}}1_{n_0} & 0 & 0 & -\frac{i}{\sqrt{2}} 1_{n_0} & 0 & 0  \\
    0 & 1_{n_+} & 0 & 0 & 0 & 0 \\
    0 & 0 &  1_{n_-} & 0 & 0 & 0 \\
    -\frac{i}{\sqrt{2}} 1_{n_0} & 0 & 0 & \frac{1}{\sqrt{2}}1_{n_0} & 0 & 0 \\
    0 & 0 & 0 & 0 &  1_{n_+} & 0\\
    0 & 0 & 0 & 0 & 0 & 1_{n_-}\end{pmatrix}
\end{equation}

\begin{equation}\label{eq:cayleymatrix2}
    c_\Sigma = \begin{pmatrix}  1_{n_0} & 0 & 0 & 0 & 0 & 0  \\
    0 & \frac{1}{\sqrt{2}} 1_{n_+} & 0 & 0 & -\frac{i}{\sqrt{2}} 1_{n_+} & 0 \\
    0 & 0 &  1_{n_-} & 0 & 0 & 0 \\
    0 & 0 & 0 & 1_{n_0} & 0 & 0 \\
    0 & -\frac{i}{\sqrt{2}} 1_{n_+} & 0 & 0 & \frac{1}{\sqrt{2}} 1_{n_+} & 0\\
    0 & 0 & 0 & 0 & 0 & 1_{n_-}\end{pmatrix}.
\end{equation}

\subsubsection{Symmetries of holomorphic arc components}
\label{subsec:symmetriesholarc}

The set of all roots orthogonal to $\Psi \setminus \Gamma$ is, with uncorrelated signs,
\[ (\Psi \setminus \Gamma)^{\perp} = \{ \pm i \varepsilon_j \pm i\varepsilon_\ell : n_0 < j, \ell \le n\} \setminus \{0\},\]
The Lie algebra $\t^\C + \sum_{\varphi \in (\Psi \setminus \Gamma)^{\perp}} \g_{\varphi}^\C$ consists of all complex $2n \times 2n$ matrices
\[ \begin{pmatrix} 0 & 0 & -\dd & 0 \\ 0 & \aa & 0 & \bb\\ \dd & 0 & 0 & 0\\ 0 & \cc & 0 & -\aa \end{pmatrix}\]
such that
\[\dd = \diag(d_1, \cdots, d_{n_0}) \quad d_j \in \C, 1 \le j \le n_0\]
and
\[\aa \in \Mat_{(n-n_0) \times (n-n_0)}(\C), \quad \bb, \cc \in \Mat_{(n-n_0) \times (n- n_0)}(\C)^{t} .\]

Its derived algebra is
\[ \g_{\Psi \setminus\Gamma}^\C = \left\{ \begin{pmatrix} 0 & 0 & 0 & 0 \\ 0 & \aa & 0 & \bb\\ 0 & 0 & 0 & 0\\ 0 & \cc & 0 & -\aa \end{pmatrix} : \bb, \cc \in \Mat_{(n-n_0) \times (n- n_0)}(\C)^{t} \right\} \subseteq \g^\C.\]
Let  
\begin{eqnarray*}
\g_{\Psi \setminus \Gamma} &:=& \g_{\Psi \setminus \Gamma}^\C \cap \g\\
\k_{\Psi \setminus \Gamma} &:=& \g_{\Psi \setminus \Gamma} \cap \k\\
\p_{\Psi \setminus \Gamma} &:=& \g_{\Psi \setminus \Gamma} \cap \p
\end{eqnarray*}

so that $\g_{\Psi \setminus \Gamma} = \k_{\Psi \setminus \Gamma} \oplus \p_{\Psi \setminus \Gamma}$.

We can compute
\[ c_\Sigma^4 =\begin{pmatrix} 1_{n_0} & 0 & 0 & 0 & 0\\
0 & -1_{n_+} & 0 & 0 & 0& 0 \\ 0 & 0 & 1_{n_-} & 0 & 0& 0\\ 0 & 0 & 0 &  1_{n_0} & 0  & 0\\ 0 & 0 & 0 & 0 & -1_{n_+} & 0 \\ 0 & 0 & 0 & 0 & 0 &  1_{n_-} \end{pmatrix}. \]

Let, with correlating signs, $\g^{\pm \Sigma} := \{ v \in \g : \Ad(c_{\Sigma}^4)v = \pm v\}$.

Then let
\begin{eqnarray*}
\g_{\Psi \setminus \Gamma}^\Sigma &:=& \g_{\Psi \setminus \Gamma} \cap \g^{+\Sigma}\\
\k_{\Psi \setminus \Gamma}^\Sigma &:=& \k_{\Psi \setminus \Gamma} \cap \g^{+\Sigma}\\
\p_{\Psi \setminus \Gamma}^\Sigma &:=& \p_{\Psi \setminus \Gamma} \cap \g^{+\Sigma}\\
\q_{\Psi \setminus \Gamma} &:=& \k_{\Psi \setminus \Gamma} \cap \g^{-\Sigma}\\
\r_{\Psi \setminus \Gamma} &:=& \p_{\Psi \setminus \Gamma} \cap \g^{-\Sigma}.
\end{eqnarray*}

Then
\[ \g_{\Psi \setminus \Gamma} = \k_{\Psi \setminus \Gamma}^\Sigma \oplus \p_{\Psi \setminus \Gamma}^\Sigma \oplus \q_{\Psi \setminus \Gamma}^{\Sigma} \oplus \r_{\Psi \setminus \Gamma}^{\Sigma}\]
where
\begin{eqnarray}
    \k_{\Psi \setminus \Gamma}^\Sigma &=& \left\{ \begin{psmallmatrix} 0 & 0 & 0 & 0 & 0 & 0 \\
    0 & \aa_{++} & 0 & 0 & -\bb_{++} & 0 \\
    0 & 0 & \aa_{--} & 0 & 0 & -\bb_{--} \\
    0 & 0 & 0 & 0 & 0 & 0 \\
    0 & \bb_{++} & 0 & 0 & \aa_{++} & 0\\
    0 & 0 & \bb_{--} & 0 & 0 & \aa_{--} \end{psmallmatrix}: \begin{aligned}\aa_{\pm \pm} &\in \Mat_{n_{\pm} \times n_{\pm}}(\R)^{-t},\\ \bb_{\pm \pm} &\in \Mat_{n_{\pm} \times n_{\pm}}(\R)^{t}\end{aligned}\right\}\\
     \p_{\Psi \setminus \Gamma}^\Sigma &=& \left\{ \begin{psmallmatrix} 0 & 0 & 0 & 0 & 0 & 0 \\
    0 & \aa_{++} & 0 & 0 & \bb_{++} & 0 \\
    0 & 0 & \aa_{--} & 0 & 0 & \bb_{--} \\
    0 & 0 & 0 & 0 & 0 & 0 \\
    0 & \bb_{++} & 0 & 0 & -\aa_{++} & 0\\
    0 & 0 & \bb_{--} & 0 & 0 & -\aa_{--} \end{psmallmatrix}: \begin{aligned}\aa_{\pm \pm} &\in \Mat_{n_{\pm} \times n_{\pm}}(\R)^{t},\\ \bb_{\pm \pm} &\in \Mat_{n_{\pm} \times n_{\pm}}(\R)^{t}\end{aligned}\right\} \\
     \q_{\Psi \setminus \Gamma}^\Sigma &=& \left\{ \begin{psmallmatrix} 0 & 0 & 0 & 0 & 0 & 0 \\
    0 & 0 & -\aa_{-+}^t & 0 & 0 & -\bb_{-+}^t \\
    0 & \aa_{-+} & 0 & 0 & -\bb_{-+} & 0 \\
    0 & 0 & 0 & 0 & 0 & 0 \\
    0 & 0 & \bb_{-+}^t & 0 & 0 & -\aa_{-+}^t\\
    0 & \bb_{-+} & 0 & 0 & \aa_{-+} & 0 \end{psmallmatrix}: \aa_{-+}, \bb_{-+} \in \Mat_{n_- \times n_+}(\R) \right\}\\
     \r_{\Psi \setminus \Gamma}^\Sigma &=& \left\{ \begin{psmallmatrix} 0 & 0 & 0 & 0 & 0 & 0 \\
    0 & 0 & \aa_{-+}^t & 0 & 0 & \bb_{-+}^t \\
    0 & \aa_{-+} & 0 & 0 & \bb_{-+} & 0 \\
    0 & 0 & 0 & 0 & 0 & 0 \\
    0 & 0 & \bb_{-+}^t & 0 & 0 & \aa_{-+}^t\\
    0 & \bb_{-+} & 0 & 0 & \aa_{-+} & 0 \end{psmallmatrix}: \aa_{-+}, \bb_{-+} \in \Mat_{n_- \times n_+}(\R) \right\}.
\end{eqnarray}

Let $G_{\Psi \setminus \Gamma}^\C$, $G_{\Psi \setminus \Gamma}$, $G_{\Psi \setminus \Gamma}^\Sigma$, $K_{\Psi \setminus \Gamma}$, $K_{\Psi \setminus \Gamma}^\Sigma$, be the Lie groups corresponding to the Lie algebras $\g_{\Psi \setminus \Gamma}^\C$, $\g_{\Psi \setminus \Gamma}$, $\g_{\Psi \setminus \Gamma}^\Sigma$, $\k_{\Psi \setminus \Gamma}$, $\k_{\Psi \setminus \Gamma}^\Sigma$. Then
\begin{eqnarray*}
    G_{\Psi \setminus \Gamma}^\C &\cong& \Sp(2(n-n_0); \C)\\
    G_{\Psi \setminus \Gamma}&\cong& \Sp(2(n-n_0); \R)\\
    G_{\Psi \setminus \Gamma}^\Sigma & \cong & \Sp(2n_+; \R) \times \Sp(2n_-; \R)\\
    K_{\Psi \setminus \Gamma} &\cong & U(n-n_0)\\
    K_{\Psi \setminus \Gamma}^\Sigma &\cong& U(n_+) \times U(n_-).
\end{eqnarray*}

\subsection{Figures}

\begin{figure}[H]
    \centering
    \includegraphics[width = \textwidth]{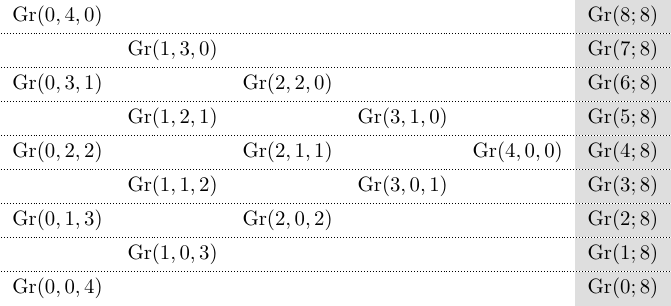}
    \caption{The Grassmannians $\Gr(\vec{n})$ are tabulated with vertical coordinate $n_+ - n_-$ and horizontal coordinate $n_0$. The rows describe the partition of the ordinary Grassmannians by the Grassmannians $\Gr(\vec{n})$. Grassmannians of, respectively, symplectic, isotropic, coisotropic subspaces are located on, respectively, left side, lower right, and upper right side of the triangle. Only the isotropic and coisotropic Grassmannians are compact. Taking symplectic complements identifies orbits symmetric about the horizontal line $n_+ = n_-$.}
    \label{fig:shelf}
\end{figure}

\begin{figure}[H] 
\centering
\includegraphics[width = \textwidth]{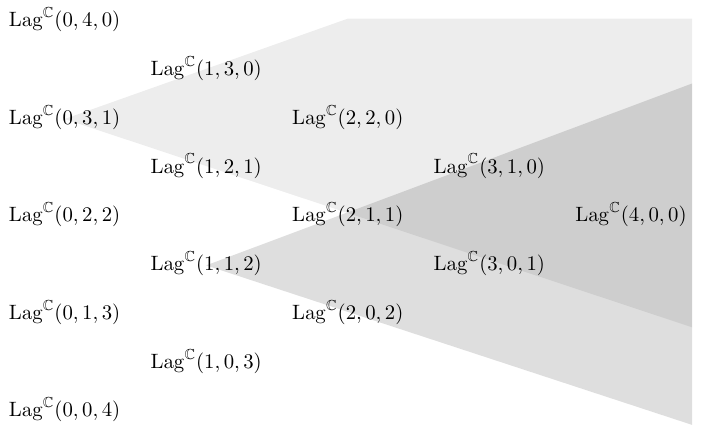}
\caption{The $G$-orbits $\Lag^\C(\vec{n})$ are tabulated with vertical position $n_+ - n_-$ and horizontal position $n_0$. The shaded regions indicate $G$-orbits $\Lag^\C(\vec{n})$ in the closures of $\Lag^\C(0, 3, 1)$, $\Lag^\C(1, 1, 2)$ (and $\Lag^\C(2, 1, 1)$). Only $\Lag^\C(4, 0, 0)$ is compact. Complex conjugation identifies the orbits symmetric about the horizontal line $n_+ = n_-$. }
\label{fig:incidence}
\end{figure}

\begin{figure}[H]
\centering
\includegraphics[width=\textwidth]{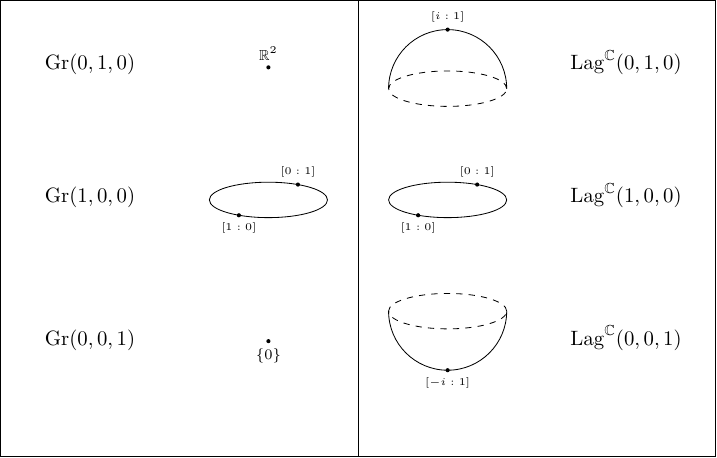}
\caption{Example of Corollary \ref{cor:homotopytype}. $Gr(\vec{n}; \R^2)$ and $\Lag^\C(\vec{n};\R^2)$ are homotopy equivalent.}
\label{fig:assemblyR2}
\end{figure}

\printbibliography

{\ifdebug {\newpage} \else {\end{document}} \fi}

\medskip\noindent\textbf{To do:}
\begin{itemize}
\item blah blah
\item blah blah
\end{itemize}

\medskip\noindent\textbf{To discuss:}
\begin{itemize}
\item Cell complex / cell decomposition / cell structure
\item CW complex / CW decomposition / CW structure
\item Cellular complex / cellular decomposition / cellular structure
\end{itemize}

\medskip\noindent\textbf{Conventions:}
\begin{itemize}
\item blah blah
\item blah blah
\end{itemize}

\medskip\noindent\textbf{Not for action:}
\begin{itemize}
\item blah blah
\item blah blah
\end{itemize}

\medskip\noindent\textbf{Journals to submit to}
\begin{itemize}
    \item Arnold Mathematical Journal (Dmitry Fuchs)
    \item Journal of Lie theory
    \item Advances in Geometry (Kaoru Ono)
\end{itemize}

\end{document}